\documentclass[leqno,12pt]{article} 
\usepackage{amsmath,amssymb,amsbsy,amscd,amsthm}

\theoremstyle{plain} % イタリック体
\newtheorem{theorem}{Theorem}[section] % 見出しはスモールキャップ
\newtheorem{lemma}[theorem]{Lemma}
\newtheorem{corollary}[theorem]{Corollary}
\newtheorem{proposition}[theorem]{Proposition}
\newtheorem{claim}{Claim}

\def\address#1#2{\begingroup
\noindent\parbox[t]{7.8cm}{%
\small{\scshape\ignorespaces#1}\par\vskip1ex
\noindent\small{\itshape E-mail address}%
\/: #2\par\vskip4ex}\hfill%
\endgroup}%

\newcommand{\SU }{\mathop{\rm SU}\nolimits}

\newcommand{\E }{\mathop{\cal E}\nolimits}
\newcommand{\Aut }{\mathop{\rm Aut}\nolimits}

\newcommand{\T }{\mathop{\rm T}\nolimits}
\newcommand{\TT }{\mathop{\cal T}\nolimits}

\newcommand{\rank}{\mathop{\rm rank}\nolimits}
\newcommand{\degw}{\deg}
\newcommand{\zs}{\{ 0\} }
\newcommand{\id}{{\rm id}}
\newcommand{\sm}{\setminus}
\newcommand{\A}{{\cal A}}
\newcommand{\R}{{\bf R}}
\newcommand{\Q}{{\bf Q}}
\newcommand{\N}{{\bf N}}
\newcommand{\Z}{{\bf Z}}
\newcommand{\e}{{\bf e}}
\newcommand{\w}{{\bf w}}
\newcommand{\Zn}{{\Z _{\geq 0}}}
\newcommand{\ep}{\epsilon}
\newcommand{\sym}{\mathfrak{S}}

\newcommand{\kx}{k[{\bf x}]}
\newcommand{\y}{{\bf y}}
\newcommand{\ky}{k[{\bf y}]}
\newcommand{\kyy}{k[{\bf y},{\bf y}^{-1}]}

\begin{document}

\title{Shestakov-Umirbaev reductions 
and Nagata's conjecture on a polynomial automorphism}

\author{Shigeru Kuroda$^{*}$}

\date{}

\maketitle

\footnote{ 
$^{*}$Partly supported by the Grant-in-Aid for 
Young Scientists (Start-up) 19840041, 
The Ministry of Education, Culture, Sports, Science and Technology, Japan. }

\begin{abstract}
In 2003, Shestakov-Umirbaev solved Nagata's conjecture 
on an automorphism of a polynomial ring. 
In the present paper, 
we reconstruct their theory by using the 
``generalized Shestakov-Umirbaev inequality", 
which was recently given by the author. 
As a consequence, 
we obtain a more precise tameness criterion for polynomial automorphisms. 
In particular, 
we deduce that no tame automorphism of a polynomial ring 
admits a reduction of type IV\null. 
\end{abstract}

\section{Introduction}
\label{sect:intro}
\setcounter{equation}{0}

Let $k$ be a field, 
$n$ a natural number, 
and $\kx =k[x_1,\ldots ,x_n]$ 
the polynomial ring in $n$ variables over $k$. 
In the present paper, 
we discuss the structure of the automorphism group 
$\Aut _k\kx $ of $\kx $ over $k$. 
Let $F:\kx \to \kx $ be an endomorphism of $\kx $ over $k$. 
We identify $F$ with 
the $n$-tuple $(f_1,\ldots ,f_n)$ 
of elements of $\kx $, 
where $f_i=F(x_i)$ for each $i$. 
Then, 
$F$ is an automorphism of $\kx $ if and only if 
the $k$-algebra $\kx $ is generated by 
$f_1,\ldots ,f_n$. 
Note that the sum $\deg F:=\sum _{i=1}^n\deg f_i$ 
of the total degrees of $f_1,\ldots ,f_n$ 
is at least $n$ whenever $F$ is an automorphism. 
An automorphism $F$ is said to be {\it affine} if $\deg F=n$, 
in which case there exist $(a_{i,j})_{i,j}\in GL _n(k)$ 
and $(b_i)_i\in k^n$ such that 
$f_i=\sum _{j=1}^na_{i,j}x_j+b_i$ for each $i$. 
We say that $F$ is {\it elementary} if there exist 
$l\in \{ 1,\ldots ,n\} $ and 
$\phi \in k[x_1,\ldots ,x_{l-1},x_{l+1},\ldots ,x_n]$ 
such that $f_l=x_l+\phi $ and $f_i=x_i$ for each $i\neq l$. 
The subgroup $\T _k\kx $ of $\Aut _k\kx $ 
generated by affine automorphisms and elementary automorphisms 
is called the {\it tame subgroup}, 
and elements of $\T _k\kx $ are called {\it tame automorphisms}.

It is a fundamental question in polynomial ring theory 
whether $\T _k\kx =\Aut _k\kx $ holds for each $n$. 
The equality is obvious if $n=1$. 
It also holds true if $n=2$, 
which was shown by Jung~\cite{Jung} 
in 1942 when $k$ is a field of characteristic zero, 
and by van der Kulk~\cite{Kulk} in 1953 when $k$ is an arbitrary field. 
This is a consequence of the result that 
every automorphism but an affine automorphism of $\kx $ 
admits an elementary reduction if $n=2$. 
Here, 
we say that $F$ {\it admits an elementary reduction} 
if $\deg F\circ E<\deg F$ for some 
elementary automorphism $E$, 
that is, 
there exist $l\in \{ 1,\ldots ,n\} $ and 
$\phi \in k[f_1,\ldots ,f_{l-1},f_{l+1},\ldots ,f_n]$ 
such that $\deg (f_l+\phi )<\deg f_l$. 
In case of $n=2$, 
it follows from the result that, 
for each $F\in \Aut _k\kx $ with $\deg F>2$, 
there exist elementary automorphisms 
$E_1,\ldots ,E_r$ for some $r\in \N $ 
such that 
$$
\deg F>\deg F\circ E_1>\cdots >\deg F\circ E_1\circ \cdots \circ E_r=2. 
$$
%for each $F\in \Aut _k\kx $ with $\deg F>2$. 
This implies that $F$ is tame.

When $n=3$, 
the structure of $\Aut _k\kx $ becomes far more difficult. 
In 1972, 
Nagata~\cite{Nagata} conjectured that 
the automorphism 
\begin{equation}\label{eq:nagata}
F=(x_1-2(x_1x_3+x_2^2)x_2-(x_1x_3+x_2^2)^2x_3,
x_2+(x_1x_3+x_2^2)x_3,
x_3)
\end{equation}
is not tame. 
This famous conjecture was finally solved in the affirmative 
by Shestakov-Umirbaev~\cite{SU2} in 2003 
for a field $k$ of characteristic zero. 
Therefore, 
$\T _k\kx $ is not equal to $\Aut _k\kx $ if $n=3$. 
We note that 
the question remains open for $n\geq 4$.

Shestakov-Umirbaev~\cite{SU2} showed that, 
if $\deg F>3$ for $F\in \T _k\kx $, 
then $F$ admits an elementary reduction, 
or there exists a sequence of elementary automorphisms 
$E_1,\ldots ,E_r$ such that 
$\deg F\circ E_1\circ \cdots \circ E_r<\deg F$, 
where $r\in \{ 2,3,4\} $. 
In the latter case, 
$F$ satisfies some special conditions, 
and is said to {\it admit a reduction of type I, II, III or IV} 
according to the conditions. 
Nagata's automorphism is not affine, 
and does not admit neither an elementary reduction nor 
any one of the four types of reductions. 
Therefore, Shestakov-Umirbaev concluded that 
Nagata's automorphism is not tame. 
We note that there exist 
tame automorphisms which 
admit reductions of type I 
(see~\cite{EMW},~\cite{typeI} and \cite{SU2}). 
However, 
it is not known whether there exist 
automorphisms admitting reductions of the other types.

To prove the criterion above, 
Shestakov-Umirbaev~\cite[Theorem 3]{SU1} 
showed an inequality concerning the total degrees of polynomials, 
which was used as a crucial tool. 
This inequality was recently generalized by the author~\cite{SU}. 
The purpose of this paper 
is to reconstruct the Shestakov-Umirbaev theory 
using the generalized inequality. 
As a consequence, 
we obtain a more precise tameness criterion 
for polynomial automorphisms. 
In particular, 
we deduce that no tame automorphism of $\kx $ admits 
a reduction of type IV (Theorem~\ref{thm:type IV}).

The main theorem (Theorem~\ref{thm:main1}) is formulated in Section~\ref{sect:main} 
using the notion of the weighted degree of a differential form. 
In Section~\ref{sect:ineq}, 
we derive some consequences of 
the generalized Shestakov-Umirbaev inequality. 
In Section~\ref{sect:SUred}, 
we investigate properties of ``Shestakov-Umirbaev reductions", 
which is roughly speaking a generalization and refinement of 
the notions of reductions of types I, II and III\null. 
In Section~\ref{sect:heart}, 
we prove some technical propositions 
which form the core of the proof of the main theorem. 
The main theorem is proved in Section~\ref{sect:proof} 
with the aid of the results in Sections~\ref{sect:ineq}, 
\ref{sect:SUred} and \ref{sect:heart}. 
In Section~\ref{sect:remark}, 
we discuss relations with the original theory of Shestakov-Umirbaev. 
We conclude with some remarks and an appendix.

\section{Main result}\label{sect:main}
\setcounter{equation}{0}

In what follows, 
we assume that the field $k$ is of characteristic zero. 
Let $\Gamma $ be a finitely generated totally ordered $\Z $-module, 
and $\w =(w_1,\ldots ,w_n)$ 
an $n$-tuple of elements of $\Gamma $ 
with $w_i>0$ for $i=1,\ldots ,n$. 
Since a finitely generated totally ordered $\Z $-module 
is necessarily free, 
we sometimes regard $\Gamma $ as a $\Z $-submodule of 
$\Q \otimes _{\Z }\Gamma $. 
We define the $\w $-{\it weighted grading} 
$\kx =\bigoplus _{\gamma \in \Gamma }\kx _{\gamma }$  
by setting $\kx _{\gamma }$ 
to be the $k$-vector subspace of $\kx $ generated by monomials 
$x_1^{a_1}\cdots x_n^{a_n}$ with $\sum _{i=1}^na_iw_i=\gamma $ 
for each $\gamma \in \Gamma $. 
For $f\in \kx \sm \zs $, 
we define the $\w $-{\it degree} $\deg _{\w }f$ of $f$ 
to be the maximum among 
$\gamma \in \Gamma $ with $f_{\gamma }\neq 0$, 
where $f_{\gamma }\in \kx _{\gamma }$ for each $\gamma $ 
such that $f=\sum _{\gamma \in \Gamma }f_{\gamma }$. 
We define $f^{\w }=f_{\delta }$, 
where $\delta =\deg _{\w }f$. 
In case $f=0$, 
we set $\deg _{\w } f=-\infty $, 
i.e., 
a symbol which is less than any element of $\Gamma $. 
For example, 
if $\Gamma =\Z $ and $w_i=1$ for $i=1,\ldots ,n$, 
then the $\w $-degree is the same as the total degree. 
For each $k$-vector subspace $V$ of $\kx $, 
we define $V^{\w }$ to be the $k$-vector subspace of $\kx $ 
generated by $\{ f^{\w }\mid f\in V\sm \zs \} $. 
For each $l$-tuple $F=(f_1,\ldots ,f_l)$ 
of elements of $\kx $ for $l\in \N $, 
we define $\deg _{\w }F=\sum _{i=1}^l\deg_{\w }f_i$. 
For each $\sigma \in  \sym _l$, 
we define $F_{\sigma }=(f_{\sigma (1)},\ldots ,f_{\sigma (l)})$, 
where $\sym _l$ is the symmetric group of $\{ 1,\ldots ,l\} $. 
The identity permutation is denoted by $\id $. 
For distinct $i_1,\ldots ,i_r\in \{ 1,\ldots ,l\} $, 
the cyclic permutation with 
$i_1\mapsto i_2,i_2\mapsto i_3,\ldots ,i_r\mapsto i_1$ 
is denoted by $(i_1,\ldots ,i_r)$.

The $\w $-degree of a differential form 
was defined by the author~\cite{SU}. 
Let $\Omega _{\kx /k}$ be the module of differentials of $\kx $ over $k$, 
and $\bigwedge ^l\Omega _{\kx /k}$ 
the $l$-th exterior power of the $\kx $-module $\Omega _{\kx /k}$ 
for $l\in \N $. 
Then, 
we may uniquely express each 
$\omega \in \bigwedge ^l\Omega _{\kx /k}$ as 
$$
\omega =\sum _{1\leq i_1<\cdots <i_l\leq n}
f_{i_1,\ldots ,i_l}dx_{i_1}\wedge \cdots \wedge dx_{i_l},
$$
where $f_{i_1,\ldots ,i_l}\in \kx $ for each $i_1,\ldots ,i_l$. 
Here, 
$df$ denotes the differential of $f$ for each $f\in \kx $. 
We define the $\w $-degree of $\omega $ by 
\begin{equation*}
\deg _{\w }\omega =\max \{ \deg_{\w }
f_{i_1,\ldots ,i_l}x_{i_1}\cdots x_{i_l}
\mid 1\leq i_1<\cdots <i_l\leq n\} . 
\end{equation*}
By the assumption that $\omega _i>0$ for $i=1,\ldots ,n$, 
it follows that 
\begin{equation}\label{eq:differential 1}
\deg _{\w }\omega 
\geq \min 
\{ w_{i_1}+\cdots +w_{i_l}\mid 
1\leq i_1<\cdots <i_l\leq n\} >0
\end{equation}
if $\omega \neq 0$. 
For each $f\in \kx \sm k$, 
we have 
\begin{equation}\label{eq:deg df = deg f}
\deg _{\w }df
=\max \{ \deg _{\w }f_{x_i}x_i\mid i=1,\ldots ,n\} 
=\deg _{\w }f, 
\end{equation}
since $df=\sum _{i=1}^nf_{x_i}dx_i$. 
Here, 
$f_{x_i}$ denotes the partial derivative of $f$ in $x_i$ 
for each $f\in \kx $ and $i\in \{ 1,\ldots ,n\} $. 
We remark that $df_1\wedge \cdots \wedge df_l\neq 0$ 
if and only if 
$f_1,\ldots ,f_l$ are algebraically independent over $k$ 
for $f_1,\ldots ,f_l\in \kx $ 
(cf.~\cite[Proposition 1.2.9]{Essen}). 
By definition, 
it follows that 
\begin{equation}\label{eq:ineq-wedge}
\deg _{\w }df_1\wedge \cdots \wedge df_l
\leq \sum _{i=1}^l\deg _{\w }df_i
=\sum _{i=1}^l\deg _{\w }f_i. 
\end{equation}
In (\ref{eq:ineq-wedge}), 
the equality holds if and only if 
$f_1^{\w },\ldots ,f_l^{\w }$ 
are algebraically independent over $k$. 
Actually, 
we can write 
$df_1\wedge \cdots \wedge df_l
=df_1^{\w }\wedge \cdots \wedge df_l^{\w }+\eta $, 
where $\eta \in \bigwedge ^l\Omega _{\kx /k}$ 
with $\deg _{\w }\eta <\sum _{i=1}^l\deg _{\w }df_i$, 
and $\deg _{\w }df_1^{\w }\wedge \cdots \wedge df_l^{\w }
=\sum _{i=1}^l\deg _{\w }df_i$ 
if $df_1^{\w }\wedge \cdots \wedge df_l^{\w }\neq 0$. 
Therefore, if $f_1,\ldots ,f_n\in \kx $ 
are algebraically independent over $k$, 
then 
\begin{equation}\label{eq:auto deg lower bound}
\sum _{i=1}^n\deg _{\w }f_i
=\sum _{i=1}^n\deg _{\w }df_i\geq 
\deg _{\w }df_1\wedge \cdots \wedge df_n
\geq \sum _{i=1}^nw_i=:|\w |
\end{equation}
by (\ref{eq:differential 1}), 
(\ref{eq:deg df = deg f}) 
and (\ref{eq:ineq-wedge}). 
As will be shown 
in Lemma~\ref{lem:|omega|}(i), 
$F$ is tame 
if $\deg _{\w }F=|\w |$ for $F\in \Aut _k\kx $.

Now, assume that $n\geq 3$, 
and let $\TT $ be the set of triples 
$F=(f_1,f_2,f_3)$ of elements of $\kx $ 
such that $f_1$, $f_2$ and $f_3$ 
are algebraically independent over $k$. 
We identify each $F\in \TT $ with the injective 
homomorphism $F:\ky \to \kx $ of $k$-algebras defined by 
$F(y_i)=f_i$ for $i=1,2,3$, where 
$\ky =k[y_1,y_2,y_3]$ is the polynomial ring 
in three variables over $k$. 
In the case where $n=3$, 
we identify $\ky $ with $\kx $ 
by the identification $y_i=x_i$ for each $i$. 
Let $\E _i$ denote the set of elementary 
automorphisms $E$ of $\ky $ such that $E(y_j)=y_j$ 
for each $j\neq i$ for $i\in \{ 1,2,3\} $, 
and set $\E =\bigcup _{i=1}^3\E _i$. 
We say that 
$F\in \TT $ 
{\it admits an elementary reduction} for the weight $\w $ 
if $\deg _{\w }F\circ E<\deg _{\w }F$ 
for some $E\in \E $, 
and call $F\circ E$ an {\it elementary reduction} of 
$F$ for the weight $\w $.

Let $F=(f_1,f_2,f_3)$ and $G=(g_1,g_2,g_3)$ 
be elements of $\TT $. 
We say that the pair $(F,G)$ satisfies the 
{\it Shestakov-Umirbaev condition} for the weight $\w $ 
if the following conditions hold:

\medskip

(SU1) $g_1=f_1+af_3^2+cf_3$ and $g_2=f_2+bf_3$ 
for some $a,b,c\in k$, 
and $g_3-f_3$ belongs to $k[g_1,g_2]$; 

(SU2) $\deg _{\w }f_1\leq \deg _{\w }g_1$ 
and $\deg _{\w }f_2=\deg _{\w }g_2$; 

(SU3) $(g_1^{\w })^2\approx (g_2^{\w })^s$ 
for some odd number $s\geq 3$; 

(SU4) $\deg _{\w }f_3\leq \deg _{\w }g_1$, 
and $f_3^{\w }$ does not belong to 
$k[g_1^{\w }, g_2^{\w }]$; 

(SU5) 
$\deg _{\w }g_3<\deg _{\w }f_3$; 

(SU6) 
$\deg _{\w }g_3<\deg _{\w }g_1-\deg _{\w }g_2 
+\deg _{\w }dg_1\wedge dg_2$. 

\medskip

Here, 
$h_1\approx h_2$ 
(resp.\ $h_1\not\approx h_2$) denotes that 
$h_1$ and $h_2$ are linearly dependent 
(resp.\ linearly independent) over $k$ 
for each $h_1,h_2\in \kx \sm \zs $. 
We say that $F\in \TT $ 
{\it admits a Shestakov-Umirbaev reduction} 
for the weight $\w $ 
if there exist $G\in \TT $ and $\sigma \in \sym _3$ 
such that $(F_{\sigma },G_{\sigma })$ satisfies 
the Shestakov-Umirbaev condition, 
and call this $G$ a {\it Shestakov-Umirbaev reduction} of 
$F$ for the weight $\w $. 
As will be discussed in Section~\ref{sect:SUred}, 
$F$ and $G$ have various properties when $(F,G)$ 
satisfies the Shestakov-Umirbaev condition. 
For example, 
it follows from (SU1)--(SU6) that 
$\deg _{\w }G<\deg _{\w }F$ (Property (P6)).

Here is the main theorem.

\begin{theorem}\label{thm:main1}
Assume that $n=3$, 
and $\w =(w_1,w_2,w_3)$ 
is a triple of elements of $\Gamma $ with 
$w_i>0$ for $i=1,2,3$. 
If $\deg _{\w }F>|\w |$ for 
a tame automorphism $F$ of $\kx $, 
then $F$ admits an elementary reduction for the weight $\w $
or a Shestakov-Umirbaev reduction for the weight $\w $. 
\end{theorem}

In the case where $n=3$ and $\Gamma =\Z $, 
the condition that $F$ admits a Shestakov-Umirbaev reduction for 
the weight $\w =(1,1,1)$ implies that 
$F$ admits an elementary reduction 
or a reduction of one of the types I, II and III 
(Proposition~\ref{prop:sured to 123}). 
In view of this, 
the reader who is familiar with the theory of Shestakov-Umirbaev 
may notice that no tame automorphism of $\kx $ 
admits a reduction of type IV (Theorem~\ref{thm:type IV}). 
In fact, 
if $F$ admits a reduction of type IV, 
then there exists an elementary automorphism $E$ 
such that $F\circ E$ admits a reduction of type IV, 
but does not admit an elementary reduction 
nor any one of the reductions of types I, II and III (cf.~Appendix). 
In Section~\ref{sect:remark}, 
however, 
we prove this result more directly.

We remark that 
$F$ admits an elementary reduction for the weight $\w $ 
if and only if $f_i^{\w }$ belongs to 
$k[f_j,f_l]^{\w }$ for some $i\in \{ 1,2,3\} $, 
where $j,l\in \{ 1,2,3\} \sm \{ i\} $ with $j<l$. 
In the case where $\deg _{\w }f_1$, 
$\deg _{\w }f_2$ and $\deg _{\w }f_3$ 
are pairwise linearly independent over $\Z $, 
this condition implies that 
$\deg _{\w }f_i$ belongs to the subsemigroup 
of $\Gamma $ generated by $\deg _{\w }f_j$ and $\deg _{\w }f_l$. 
Indeed, 
each $\phi \in k[f_j,f_l]\sm \zs $ 
is a linear combination of $f_j^pf_l^q$ 
for $(p,q)\in (\Zn )^2$ over $k$, 
in which $\deg _{\w }f_j^pf_l^q\neq 
\deg _{\w }f_j^{p'}f_l^{q'}$ 
if and only if $(p,q)\neq (p',q')$. 
Here, $\Zn $ denotes the set of nonnegative integers. 
Hence, 
$\deg _{\w }\phi =\deg _{\w }f_j^pf_l^q=p\deg _{\w }f_i+q\deg _{\w }f_l$ 
for some $p,q\in \Zn $.

Note that 
$\delta :=(1/2)\deg _{\w }f_2=(1/2)\deg _{\w }g_2$ 
belongs to $\Gamma $ by (SU2) and (SU3). 
As will be shown in Section~\ref{sect:SUred}, 
(SU1)--(SU6) imply that 
$\delta <\deg _{\w }f_i\leq s\delta $ 
for each $i\in \{ 1,2,3\} $ (Property (P7)). 
Since $\delta >0$, it follows that $\deg _{\w }f_i<s\deg _{\w }f_j$ 
for each $i,j\in \{ 1,2,3\} $. 
Therefore, 
if $F$ admits a Shestakov-Umirbaev reduction for the weight $\w $, 
then $F$ satisfies the following conditions: 

\smallskip 

(i) One of 
$(1/2)\deg _{\w }f_1$, 
$(1/2)\deg _{\w }f_2$ and 
$(1/2)\deg _{\w }f_3$ 
belongs to $\Gamma $. 

\smallskip 

(ii) 
For each $i,j\in \{ 1,2,3\} $, 
there exists $l\in \N $ such that 
$\deg _{\w }f_i<l\deg _{\w }f_j$. 

\smallskip

Now, 
we deduce that Nagata's automorphism is not tame 
by means of Theorem~\ref{thm:main1}. 
Let $\Gamma =\Z ^3$ equipped with the lexicographic order, i.e., 
the ordering defined by $a\leq b$ for $a,b\in \Z ^3$ 
if the first nonzero component of $b-a$ is positive, 
and let $\w =(\e _1,\e _2,\e _3)$, 
where 
$\e _i$ is the $i$-th standard unit vector of $\R ^3$ 
for each $i$. 
Then, we have 
$$
\deg _{\w }f_1=(2,0,3),\ 
\deg _{\w }f_2=(1,0,2),\ 
\deg _{\w }f_3=(0,0,1). 
$$
Hence, 
$\deg _{\w }F=(3,0,6)>(1,1,1)=|\w |$. 
The three vectors above are pairwise 
linearly independent over $\Z $, 
while any one of them is not contained in the subsemigroup 
of $\Z ^3$ generated by the other two vectors. 
Hence, $F$ does not admit an elementary reduction 
for the weight $\w $. 
Since $(1/2)\deg _{\w }f_i$ does not belong to 
$\Z ^3$ for each $i\in \{ 1,2,3\} $, 
we know that $F$ does not admit a Shestakov-Umirbaev reduction 
for the weight $\w $. 
By the definition of the lexicographic order, 
$l\deg _{\w }f_3=(0,0,l)$ is less than $\deg _{\w }f_i$ 
for $i=1,2$ for any $l\in \N $, 
which also implies that $F$ does not admit a Shestakov-Umirbaev reduction 
for the weight $\w $. 
Therefore, 
we have the following corollary to Theorem~\ref{thm:main1}.

\begin{corollary}\label{cor:Nagata}
Nagata's automorphism defined in $(\ref{eq:nagata})$ is not tame. 
\end{corollary}

We define the {\it rank} of $\w $ 
as the rank of the $\Z $-submodule of $\Gamma $ 
generated by $w_1,\ldots ,w_n$. 
If $\rank \w =n$, 
then the dimension of the $k$-vector space $\kx _{\gamma }$ 
is at most one for each $\gamma $. 
Consequently, 
$\deg _{\w }f=\deg _{\w }g$ if and only if 
$f^{\w }\approx g^{\w }$ 
for each $f,g\in \kx \sm \zs $. 
In such a case, 
the assertion of Theorem~\ref{thm:main1} 
can be proved more easily than the general case. 
In fact, 
a few steps can be skipped in the proof. 
We note that $\w =(\e _1,\e _2,\e _3)$ 
has maximal rank three, 
and therefore it suffices to prove 
the assertion of Theorem~\ref{thm:main1} 
in this special case 
to conclude that Nagata's automorphism is not tame.

\section{Inequalities}\setcounter{equation}{0}
\label{sect:ineq}

In this section, 
we derive some consequences from the generalized 
Shestakov-Umirbaev inequality~\cite[Theorem 2.1]{SU}. 
In what follows, 
we denote ``$\deg _{\w }$" by ``$\degw $" 
for the sake of simplicity. 
Let $g$ be a nonzero element of $\kx $, 
and $\Phi =\sum _i\phi _iy^i$ a nonzero polynomial 
in a variable $y$ over $\kx $, 
where $\phi _i\in \kx $ for each $i\in \Zn $. 
We define $\deg _{\w }^g\Phi $ to be the maximum among 
$\degw \phi _ig^i$ for $i\in \Zn $. 
Then, 
$\deg _{\w }^g\Phi $ is not less than 
the $\w $-degree of 
$\Phi (g):=\sum _i\phi _ig^i$ in general. 
On the other hand, 
$\deg _{\w }^g\Phi ^{(i)}=\degw \Phi ^{(i)}(g)$ 
holds for sufficiently large $i$, 
where $\Phi ^{(i)}$ 
denotes the $i$-th order derivative of $\Phi $ in $y$. 
We define $m_{\w }^g(\Phi )$ 
to be the minimal $i\in \Zn $ such that 
$\deg _{\w }^g\Phi ^{(i)}=\degw \Phi ^{(i)}(g)$.

In the notation above, 
the generalized Shestakov-Umirbaev inequality 
is stated as follows. 
This inequality plays an important role in our theory, 
yet the proof is quite simple and short.

\begin{theorem}[{\cite[Theorem 2.1]{SU}}]\label{thm:inequality}
Assume that $f_1,\ldots ,f_r\in \kx $ 
are algebraically independent over $k$, 
where $1\leq r\leq n$. 
Then, 
$$
\degw \Phi (g)\geq \deg _{\w }^g\Phi +m_{\w }^g(\Phi )
(\degw \omega \wedge dg -\degw \omega -\degw g)
$$
holds for each $\Phi \in k[f_1,\ldots ,f_r][y]\sm \zs $ 
and $g\in \kx \sm \zs $, where $\omega =df_1\wedge \cdots \wedge df_r$. 
\end{theorem}

Here is a remark (see \cite[Section 3]{SU} for detail). 
Define $\Phi ^{\w ,g}=\sum _{i\in I}\phi _i^{\w }y^i$ 
for each $\Phi \in \kx [y]$, 
where $I$ is the set of $i\in \Zn $ such that 
$\degw \phi _ig^i=\deg _{\w }^g\Phi $. 
Then, 
$(\Phi ^{(i)})^{\w ,g}=(\Phi ^{\w ,g})^{(i)}$ 
holds for each $i$. 
Moreover, 
$\deg _{\w }^g\Phi =\degw \Phi (g)$ if and only if 
$\Phi ^{\w ,g}(g^{\w })\neq 0$. 
Hence, 
$m_{\w }^g(\Phi )$ is equal to the minimal $i\in \Zn $ 
such that $(\Phi ^{\w ,g})^{(i)}(g^{\w })\neq 0$. 
Since $k$ is of characteristic zero, 
this implies that $g^{\w }$ 
is a multiple roof of $\Phi ^{\w ,g}$ of order 
$m_{\w }^g(\Phi )$.

Now, let $S=\{ f,g\} $ be a subset of $\kx $ 
such that $f$ and $g$ are 
algebraically independent over $k$, 
and $\phi $ a nonzero element of $k[S]$. 
We can uniquely express 
$\phi =\sum _{i,j}c_{i,j}f^ig^j$, 
where $c_{i,j}\in k$ for each $i,j\in \Zn $. 
We define $\degw ^S\phi $ to be the maximum among 
$\degw f^ig^j$ for $i,j\in \Zn $ with $c_{i,j}\neq 0$. 
We will frequently use the fact that, 
if $\phi ^{\w }$ does not belong to $k[f^{\w },g^{\w }]$, 
or if $\degw \phi <\degw f$ and $\phi $ 
does not belong to $k[g]$, 
then $\degw \phi <\degw ^S\phi $.

The following lemma is a consequence of Theorem~\ref{thm:inequality}. 
The statement (i) is an analogue of 
Shestakov-Umirbaev~\cite[Corollary 1]{SU2}, 
but the statement (ii) is essentially new.

\begin{lemma}\label{lem:degS1}
Let $S=\{ f,g\} $ be as above, 
and $\phi $ a nonzero element of $k[S]$ 
such that $\degw \phi <\degw ^S\phi $. 
Then, there exist $p,q\in \N $ with $\gcd (p,q)=1$ 
such that $(g^{\w })^p\approx (f^{\w })^q$. 
Furthermore, 
the following assertions hold$:$

{\rm (i)} 
$\degw \phi \geq q\degw f+\degw df\wedge dg-\degw f-\degw g$. 

{\rm (ii)} 
Let $h$ be an element of $\kx $ such that 
$f$, $g$ and $h$ are algebraically independent over $k$. 
If $\degw (h+\phi )<\degw h$, 
then 
\begin{align*}
\degw (h+\phi )&\geq q\degw f+\degw df\wedge dg\wedge dh
-\degw df\wedge dh-\degw g. 
\end{align*}
\end{lemma}
\begin{proof}\rm
Let $\Phi =\sum _{i,j}c_{i,j}f^iy^j$ be an element of $k[f][y]$ 
such that $\Phi (g)=\phi $, 
where $c_{i,j}\in k$ for each $i,j\in \Zn $, 
and let $J$ be the set of $(i,j)\in (\Zn )^2$ 
such that $c_{i,j}\neq 0$ and 
$\degw f^ig^j=\degw ^S\phi $. 
Then, we have $\deg _{\w }^g\Phi =\degw ^S\phi $ and 
$$
\Phi ^{\w ,g}=\sum _{(i,j)\in J}c_{i,j}(f^{\w })^iy^j. 
$$
Since $\degw \phi <\degw ^S\phi $ by assumption, 
we get $\degw \Phi (g)<\deg _{\w }^g\Phi $. 
Hence, 
$m_{\w }^g(\Phi )\geq 1$ and 
$\Phi ^{\w ,g}(g^{\w })=0$ as mentioned. 
In particular, 
$J$ contains at least two elements, 
say $(i,j)$ and $(i',j')$, 
since $\Phi ^{\w ,g}\neq 0$, 
$g^{\w }\neq 0$ and $\Phi ^{\w ,g}(g^{\w })=0$. 
Then, 
$(i-i')\degw f=(j'-j)\degw g$. 
Since $\degw f>0$ and $\degw g>0$, 
this implies that $q\degw f=p\degw g$ 
for some $p,q\in \N $ with $\gcd (p,q)=1$. 
For each $(i_1,j_1),(i_2,j_2)\in J$, 
there exists $l\in \Z $ such that 
$i_2-i_1=-ql$ and $j_2-j_1=pl$. 
Hence, we can find 
$(i_0,j_0)\in J$ and $m\in \N $ 
such that 
$J$ is contained in 
$\{ (i_0-ql,j_0+pl)\mid l=0,\ldots ,m\} $, 
and 
$(i_0-qm,j_0+pm)$ belongs to $J$. 
Note that $qm\leq i_0$. 
Putting $c_l'=c_{i_0-ql,j_0+pl}$ for each $l$, 
we get 
$$
\Phi ^{\w ,g}=(f^{\w })^{i_0}y^{j_0}
\sum _{l=0}^mc_l'(f^{\w })^{-ql}y^{pl}
=c_m'(f^{\w })^{i_0}y^{j_0}\prod _{l=1}^m
\left( (f^{\w })^{-q}y^{p}-\alpha _l\right) , 
$$
where $\alpha _1,\ldots ,\alpha _m$ 
are the roots of the equation 
$\sum _{l=0}^mc_l'y^l=0$ in an algebraic closure of $k$. 
Since 
$\Phi ^{\w ,g}(g^{\w })=0$, 
we get $(f^{\w })^{-q}(g^{\w })^{p}=\alpha _l$ 
for some $l$. 
Then, $\alpha _l$ belongs to $k\sm \zs $, 
because $f^{\w }$ and $g^{\w }$ 
are in $\kx \sm \zs $. 
Therefore, 
$(g^{\w })^p\approx (f^{\w })^q$. 
This proves the first assertion. 
By the expression above, 
we know that $\Phi ^{\w ,g}$ cannot have 
a multiple root of order greater than $m$. 
Hence, 
$m_{\w }^g(\Phi )\leq m$. 
Thus, it follows that 
\begin{equation}\label{eq:degm}
\deg _{\w }^g\Phi =\degw ^S\phi =\degw f^{i_0}g^{j_0}
\geq i_0\degw f\geq qm\degw f\geq qm_{\w }^g(\Phi )\degw f. 
\end{equation}
By Theorem~\ref{thm:inequality}, together with 
(\ref{eq:deg df = deg f}) and (\ref{eq:degm}), 
we get 
\begin{align*}
&\degw \phi 
=\degw \Phi (g)
\geq \deg _{\w }^g\Phi +m_{\w }^g(\Phi )
(\degw df\wedge dg-\degw f-\degw g) \\
&\quad \geq qm_{\w }^g(\Phi )\degw f+m_{\w }^g(\Phi )
(\degw df\wedge dg-\degw f-\degw g) 
\geq m_{\w }^g(\Phi )M, 
\end{align*}
where $M=q\degw f+\degw df\wedge dg-\degw f-\degw g$. 
Since $m_{\w }^g(\Phi )\geq 1$, 
the assertion (i) follows from the inequality above 
if $M>0$. 
If $M\leq 0$, 
then (i) is clear, 
since $\degw \phi \geq 0$.

To show (ii), 
consider the polynomial $\Psi :=h+\Phi $ 
in $y$ over $k[f,h]$. 
Recall that 
$\degw \phi <\degw ^S\phi =\deg _{\w }^g\Phi $. 
By the assumption that $\degw (h+\phi )<\degw h$, 
we get $\degw \phi =\degw h$. 
Hence, 
$\degw h<\deg _{\w }^g\Phi $. 
Thus, we have 
$\deg _{\w }^g\Psi =\deg _{\w }^g\Phi $ 
and $\Psi ^{\w ,g}=\Phi ^{\w ,g}$, 
and so $m_{\w }^g(\Psi )=m_{\w }^g(\Phi )$. 
Therefore, 
$\deg _{\w }^g\Psi \geq qm_{\w }^g(\Psi )\degw f$ 
by (\ref{eq:degm}). 
By Theorem~\ref{thm:inequality}, we obtain 
\begin{align*}
&\degw (h+\phi )=\degw \Psi (g)\\ 
&\quad \geq 
\deg _{\w }^g\Psi +m_{\w }^g(\Psi )M'
\geq qm_{\w }^g(\Psi )\degw f+m_{\w }^g(\Psi )M'
%&\quad 
\geq m_{\w }^g(\Psi )(q\degw f+M'), 
\end{align*}
where 
$M'=\degw df\wedge dh\wedge dg-\degw df\wedge dh-\degw g$. 
Since $m_{\w }^g(\Psi )=m_{\w }^g(\Phi )\geq 1$, 
the inequality above implies the inequality in (ii). 
\end{proof}

Let $p$ and $q$ 
be natural numbers such that 
$\gcd (p,q)=1$ and $2\leq p<q$. 
We claim that the following assertions hold: 

(i) $pq-p-q>0$. 

(ii) If $pq-p-q\leq q$, then $p=2$ and $q\geq 3$ is an odd number. 

(iii) If $pq-p-q\leq p$, then $p=2$ and $q=3$. \\
We leave to the reader to check them.

\begin{lemma}\label{lem:degS2}
Let $f$, $g$, $\phi $ 
and $p$, $q$ be as in Lemma~{\rm \ref{lem:degS1}}. 

{\rm (i)} 
Assume that $f^{\w }$ does not belong to $k[g^{\w }]$, 
and $g^{\w }$ does not belong to $k[f^{\w }]$. 
Then, $\degw \phi>\degw df\wedge dg$. 

{\rm (ii)} 
Assume that $\degw f<\degw g$, 
$\degw \phi \leq \degw g$, 
and $g^{\w }$ does not belong to $k[f^{\w }]$. 
Then, 
$p=2$, $q\geq 3$ is an odd number, 
$\delta :=(1/2)\degw f$ belongs to $\Gamma $, and
\begin{equation}\label{eq:lem:degS2:ii}
\degw \phi \geq (q-2)\delta +\degw df\wedge dg,\quad 
\degw d\phi \wedge df\geq q\delta +\degw df\wedge dg. 
\end{equation}
If $\degw \phi \leq \degw f$, 
then $q=3$. 
\end{lemma}
\begin{proof}\rm
Since $p\degw g=q\degw f$ and $\gcd (p,q)=1$, 
it follows that $\delta :=(1/p)\degw f$ belongs to $\Gamma $. 
By Lemma~\ref{lem:degS1}(i), we have 
\begin{align}\label{eq:degS21}
\degw \phi \geq 
p\degw g+\degw df\wedge dg-\degw f-\degw g 
=(pq-p-q)\delta +\degw df\wedge dg. 
\end{align}

Since $(g^{\w})^p\approx (f^{\w })^q$ and $\gcd (p,q)=1$, 
the assumptions of (i) imply $2\leq p<q$ or $2\leq q<p$. 
Hence, $pq-p-q>0$ as claimed above. 
Therefore, 
$\degw \phi>\degw df\wedge dg$ by (\ref{eq:degS21}), 
proving (i).

In case (ii), 
we have $2\leq p<q$, 
since $g^{\w }$ does not belong to $k[f^{\w }]$. 
Since $\degw \phi \leq \degw g=q\delta $ by assumption, 
(\ref{eq:degS21}) yields $pq-p-q<q$. 
Thus, 
$p=2$, and $q\geq 3$ is an odd number 
by the claim. 
By substituting $p=2$, 
we obtain from (\ref{eq:degS21}) 
the first inequality of (\ref{eq:lem:degS2:ii}). 
To show the second inequality of (\ref{eq:lem:degS2:ii}), 
consider $\Phi \in k[f][y]$ defined in the proof of Lemma~\ref{lem:degS1}. 
Recall that $m_{\w }^g(\Phi )\geq 1$, 
and $pm_{\w }^g(\Phi )\degw g
=qm_{\w }^g(\Phi )\degw f\leq \deg _{\w }^g\Phi $ by (\ref{eq:degm}). 
By definition, 
$\deg _{\w }^g\Phi ^{(1)}=\deg _{\w }^g\Phi -\degw g$ 
and $m_{\w }^g(\Phi ^{(1)})=m_{\w }^g(\Phi )-1$. 
Since $p=2$ and 
$\degw f<\degw g$, 
it follows from Theorem~\ref{thm:main1} that
\begin{align*}
\degw \Phi ^{(1)}(g)
&\geq \deg _{\w }^g\Phi ^{(1)}+m_{\w }^g(\Phi ^{(1)})M'' \\
&=\deg _{\w }^g\Phi -\degw g+(m_{\w }^g(\Phi )-1)M'' \\
&\geq 2m_{\w }^g(\Phi )\degw g-\degw g+(m_{\w }^g(\Phi )-1)M'' \\
&=(m_{\w }^g(\Phi )-1)
( \degw df\wedge dg-\degw f+\degw g)+\degw g\\ 
&\geq \degw g=q\delta , 
\end{align*}
where $M''=\degw df\wedge dg-\degw f-\degw g$. 
Since 
$d\phi =\left(\sum _{i,j}c_{i,j}if^{i-1}g^j\right) df
+\Phi ^{(1)}(g)dg$, 
we have $d\phi \wedge df=\Phi ^{(1)}(g)dg\wedge df$. 
Therefore, 
$$
\degw d\phi \wedge df=
\degw \Phi ^{(1)}(g)+\degw df\wedge dg\geq q\delta +\degw df\wedge dg. 
$$
This proves the second inequality of (\ref{eq:lem:degS2:ii}). 
If $\degw \phi \leq \degw f$, 
then $pq-p-q<p$ by (\ref{eq:degS21}), 
since $\deg f=p\delta $. 
Hence, $q=3$ as claimed above. 
\end{proof}

The following remark is useful. 
Assume that $f,g,h\in \kx $ and $\phi \in k[S]$ 
satisfy (i)--(iv) as follows, 
where $S=\{ f,g\} $: 

(i) $f$ and $g$ are algebraically independent over $k$; 

(ii) $\degw f<\degw g$ and $\degw h<\degw g$; 

(iii) $g^{\w }$ and $h^{\w }$ 
do not belong to $k[f^{\w }]$; 

(iv) $\degw (h+\phi )<\degw h$. \\
Then, we claim that 
$\deg \phi <\deg ^S\phi $, 
and that $f$, $g$ and $\phi $ satisfy 
the assumptions of Lemma~\ref{lem:degS2}(ii). 
In fact, 
$\phi ^{\w }\approx h^{\w }$ by (iv), 
and $h^{\w }$ does not belong to $k[f^{\w },g^{\w }]$, 
since $h^{\w }$ does not belong to $k[f^{\w }]$ by (iii), 
and $\deg h<\deg g$ by (ii). 
Hence, 
$\phi ^{\w }$ does not belong to $k[f^{\w },g^{\w }]$. 
Because $\phi $ is an element of $k[f,g]$, 
we get $\degw \phi <\degw ^S\phi $. 
By (ii) and (iii), 
it follows that 
$\degw f<\degw g$, $\deg \phi =\deg h<\deg g$, 
and $g^{\w }$ does not belong to $k[f^{\w }]$. 
Thus, 
$f$, $g$ and $\phi $ 
satisfy the required conditions. 
Therefore, 
the conclusion of Lemma~\ref{lem:degS2}(ii) 
holds in this situation.

The following theorem is a generalization of 
Shestakov-Umirbaev~\cite[Lemma 5]{SU1}.

\begin{theorem}[{\cite[Theorem 5.2]{SU}}]\label{thm:determinant}
For each $\eta _1,\ldots ,\eta _{l}\in \Omega _{\kx /k}$ 
for $l\geq 2$, 
there exist $1\leq i_1<i_2\leq l$ such that 
$$
\degw \eta _{i_1}+\degw \tilde{\eta }_{i_1}
=\degw \eta _{i_2}+\degw \tilde{\eta }_{i_2}
\geq \degw \eta _{i}+\degw \tilde{\eta }_{i}
$$
for $i=1,\ldots ,l$, 
where $\tilde{\eta }_i
=\eta _1\wedge \cdots \wedge \eta _{i-1}\wedge 
\eta _{i+1}\wedge \cdots \wedge \eta _{l}$ for each $i$. 
\end{theorem}

Using Theorem~\ref{thm:determinant},  
we prove a lemma needed later. 
Assume that $k_1,k_2,k_3\in \kx $ 
are algebraically independent over $k$, 
and $k_1':=k_1+ak_3^2+ck_3+\psi $ 
and $k_2':=k_2+\phi $ satisfy the following conditions 
for some $a,c\in k$, $\psi \in k[k_2]$ and $\phi \in k[k_3]$: 

(a) $\degw k_2'<\degw k_1'$; 

(b) $\degw k_1'-\degw k_2'<\degw k_3$; 

(c) $\degw \psi <\degw k_1'-\degw k_2'+\degw k_2$; 

(d) $\degw k_3+\degw dk_1'\wedge dk_2'
<\degw k_1'+\degw dk_2'\wedge dk_3$.

\begin{lemma}\label{lem:determinant}
Under the assumption above, 
we have 
\begin{equation}\label{eq:detfirst}
\degw dk_1\wedge dk_3
=\degw k_1'-\degw k_2'+\degw dk_2\wedge dk_3. 
\end{equation}
If furthermore $\phi =bk_3+d$ for some $b,d\in k$, 
then the following assertions hold$:$ 

{\rm (i)} If $a\neq 0$ 
and $\degw dk_1'\wedge dk_2'<\degw k_3$, 
then 
$$
\degw dk_1\wedge dk_2=\degw k_3+\degw dk_2\wedge dk_3. 
$$ 

{\rm (ii)} 
Assume that $\degw dk_1'\wedge dk_2'<\degw dk_2\wedge dk_3$. 
Then, 
$$
\degw dk_1\wedge dk_2=\left\{
\begin{array}{ll}
\degw k_3+\degw dk_2\wedge dk_3 & \text{ if }a\neq 0 \\
\degw dk_1\wedge dk_3 & \text{ if }a=0\text{ and }b\neq 0 \\
\degw dk_2\wedge dk_3 & \text{ if }a=b=0
\text{ and }c\neq 0 \\
\degw dk_1'\wedge dk_2' & \text{ if }a=b=c=0 .
\end{array}
\right. 
$$

{\rm (iii)} 
Assume that 
$\psi $ belongs to $k$. 
Set $k_1''=k_1+a'k_3^2+c'k_3+\psi '$ 
and $k_2''=k_2+b'k_3+d'$ 
for $a',b',c',d',\psi '\in k$. 
If $\degw dk_1'\wedge dk_2'$ and 
$\degw dk_1''\wedge dk_2''$ are 
less than $\degw dk_2\wedge dk_3$, 
then $(a',b',c')=(a,b,c)$. 
\end{lemma}
\begin{proof}\rm
Put $\eta _1=dk_1'$, $\eta _2=dk_2'$ and $\eta _3=dk_3$. 
Then, 
$\degw \eta _3+\degw \tilde{\eta }_3
<\degw \eta _1+\degw \tilde{\eta }_1$ by (d), 
since $\degw dk_i'=\degw k_i'$ for $i=1,2$ 
and $\degw dk_3=\degw k_3$ by (\ref{eq:deg df = deg f}). 
Hence, we must have 
$\degw \eta _1+\degw \tilde{\eta }_1
=\degw \eta _2+\degw \tilde{\eta }_2$ 
by Theorem~\ref{thm:determinant}. 
Since $\phi $ is an element of $k[k_3]$, 
we get $d\phi \wedge dk_3=0$. 
Hence, 
$dk_2'\wedge dk_3=d(k_2+\phi )\wedge dk_3=dk_2\wedge dk_3$. 
Thus, we obtain
\begin{equation}\label{eq:pfdet1}
\begin{aligned}
\degw dk_1'\wedge dk_3
=\deg \tilde{\eta }_2
&=\deg \eta _1-\deg \eta _2+\deg \tilde{\eta }_1\\
&=\degw k_1'-\degw k_2'+\degw dk_2'\wedge dk_3 \\
&=\degw k_1'-\degw k_2'+\degw dk_2\wedge dk_3. 
\end{aligned}
\end{equation}
We show that 
$\degw dk_1'\wedge dk_3=\degw dk_1\wedge dk_3$, 
which proves (\ref{eq:detfirst}). 
Set $\psi _1=\Psi ^{(1)}(k_2)$, 
where $\Psi \in k[y]$ such that $\Psi (k_2)=\psi $. 
Then, $\degw \psi _1\leq \degw \psi -\degw k_2$, 
and so $\degw \psi _1<\degw k_1'-\degw k_2'$ by (c). 
Hence, 
\begin{equation}\label{eq:pfdet3}
\begin{aligned}
\degw \psi _1\mathit{dk}_2\wedge \mathit{dk}_3
&=\deg \psi _1+\deg dk_2\wedge dk_3 \\
&<\degw k_1'-\degw k_2'+\degw \mathit{dk}_2\wedge \mathit{dk}_3
=\deg dk_1'\wedge dk_3
\end{aligned}
\end{equation}
by (\ref{eq:pfdet1}). 
Since $d\psi =\psi _1dk_2$, 
it follows that
\begin{equation*}
dk_1'\wedge dk_3
=dk_1\wedge dk_3
+2ak_3dk_3\wedge dk_3+cdk_3\wedge dk_3
+d\psi \wedge dk_3
=dk_1\wedge dk_3+\psi _1dk_2\wedge dk_3. 
\end{equation*}
This equality and (\ref{eq:pfdet3}) imply 
$\degw dk_1\wedge dk_3=\degw dk_1'\wedge dk_3$. 
This proves (\ref{eq:detfirst}).

Next, assume that $\phi =bk_3+d$ for some $b,d\in k$. 
Then, we have 
\begin{align}
dk_1\wedge dk_2=dk_1'\wedge dk_2'+2ak_3dk_2\wedge dk_3
-b(dk_1\wedge dk_3+\psi _1dk_2\wedge dk_3)
+c\mathit{dk}_2\wedge dk_3. \label{eq:pfdet4}
\end{align}
By (b), (a) and (\ref{eq:pfdet3}), it follows that 
\begin{multline*}
\deg k_3dk_2\wedge dk_3=\deg k_3+\deg dk_2\wedge dk_3
>\degw k_1'-\degw k_2'+\degw dk_2\wedge dk_3\\ 
>\max\{ \degw dk_2\wedge dk_3,\degw \psi _1dk_2\wedge dk_3\} . 
\end{multline*}
Since the right-hand side 
of the first inequality is equal to 
$\deg dk_1\wedge dk_3$ by (\ref{eq:detfirst}), 
we get 
\begin{align}\label{eq:pfdet6}
\degw k_3dk_2\wedge dk_3
>\deg dk_1\wedge dk_3 
>\max\{ \degw dk_2\wedge dk_3,\degw \psi _1dk_2\wedge dk_3\} . 
\end{align}
In view of (\ref{eq:pfdet6}), 
the assertions (i) and (ii) easily follow from (\ref{eq:pfdet4}).

Finally, we verify (iii). 
A direct forward computation shows that
$$
dk_1''\wedge dk_2''-dk_1'\wedge dk_2'=
2(a-a')k_3dk_2\wedge dk_3-(b-b')dk_1\wedge dk_3
+(c-c')dk_2\wedge dk_3. 
$$
By assumption, 
the $\w $-degree of the left-hand side of 
this equality is less than that of $dk_2\wedge dk_3$, 
while those of 
$k_3dk_2\wedge dk_3$ and $dk_1\wedge dk_3$ are greater than 
that of $dk_2\wedge dk_3$ by (\ref{eq:pfdet6}). 
Therefore, 
it follows that $a=a'$, $b=b'$ and $c=c'$. 
\end{proof}

\section{Shestakov-Umirbaev reductions}\label{sect:SUred}
\setcounter{equation}{0}
In this section, 
we study the properties of Shestakov-Umirbaev reductions. 
In what follows, unless otherwise stated, 
$F=(f_1,f_2,f_3)$ and $G=(g_1,g_2,g_3)$ 
denote elements of $\TT $, 
and $S_i:=\{ f_1,f_2,f_3\} \sm \{ f_i\} $ for each $i$. 
We say that the pair $(F,G)$ satisfies 
the {\it quasi Shestakov-Umirbaev condition} for the weight $\w $ 
if (SU4), (SU5), (SU6) 
and the following three conditions hold: 

\medskip 

(SU$1'$) $g_1-f_1$, 
$g_2-f_2$ and $g_3-f_3$ belong to $k[f_2,f_3]$, 
$k[f_3]$ and $k[g_1,g_2]$, respectively; 

(SU$2'$) $\deg f_i\leq \deg g_i$ for $i=1,2$; 

(SU$3'$) $\deg g_2<\deg g_1$, and 
$g_1^{\w }$ does not belong to $k[g_2^{\w }]$. 

\medskip 

It is easy to see that 
(SU1), (SU2) and (SU3) 
imply (SU$1'$), (SU$2'$) and (SU$3'$), respectively. 
Hence, 
if $(F,G)$ satisfies the Shestakov-Umirbaev condition 
for the weight $\w $, then 
$(F,G)$ satisfies the quasi Shestakov-Umirbaev condition 
for the weight $\w $. 
We say that $F\in \TT $ 
{\it admits a quasi Shestakov-Umirbaev reduction} 
for the weight $\w $ 
if $(F_{\sigma },G_{\sigma })$ satisfies 
the quasi Shestakov-Umirbaev condition 
for the weight $\w $ for some $\sigma \in \sym _3$ and $G\in \TT $, 
and call this $G$ a {\it quasi Shestakov-Umirbaev reduction} of 
$F$ for the weight $\w $. 
The weight $\w $ is fixed throughout, 
and so is not explicitly mentioned in what follows.

We show that $F$ and $G$ 
have the properties (P1)--(P12) as follows 
if $(F,G)$ satisfies the quasi Shestakov-Umirbaev condition: 

\medskip 

{\rm (P1)} $(g_1^{\w })^2\approx (g_2^{\w })^s$ 
for some odd number $s\geq 3$, 
and so $\delta :=(1/2)\degw g_2$ belongs to $\Gamma $. 

{\rm (P2)} $\degw f_3\geq (s-2)\delta +\degw dg_1\wedge dg_2$. 

{\rm (P3)} $\degw f_2=\degw g_2$. 

{\rm (P4)} If $\degw \phi \leq \degw g_1$ for $\phi \in k[S_1]$, 
then there exist $a',c'\in k$ and $\psi '\in k[f_2]$ 
with $\degw \psi '\leq (s-1)\delta $ such that 
$\phi =a'f_3^2+c'f_3+\psi '$. 

{\rm (P5)} If $\degw f_1<\degw g_1$, 
then $s=3$, $g_1^{\w }\approx (f_3^{\w })^2$, 
$\degw f_3=(3/2)\delta $ and 
$$
\degw f_1\geq \frac{5}{2}\delta +\degw dg_1\wedge dg_2. 
$$

{\rm (P6)} $\degw G<\degw F$. 

{\rm (P7)} $\degw f_2<\degw f_1$, $\degw f_3\leq \degw f_1$, 
and $\delta <\degw f_i\leq s\delta $ for $i=1,2,3$.

{\rm (P8)} $f_i^{\w }$ does not belong to $k[f_j^{\w }]$ 
if $i\neq j$ and $(i,j)\neq (1,3)$. 
If $f_1^{\w }$ belongs to $k[f_3^{\w }]$, 
then $s=3$, 
$f_1^{\w }\approx (f_3^{\w })^2$ 
and $\degw f_3=(3/2)\delta $.

{\rm (P9)} If $\degw \phi \leq \degw f_2$ for $\phi \in k[S_2]$, 
then there exist $b',d'\in k$ such that $\phi =b'f_3+d'$. 

{\rm (P10)} Assume that $k[g_1,g_2]\neq k[S_3]$. 
If $\degw \phi \leq \degw f_1$ for $\phi \in k[S_3]$, 
then there exist 
$c''\in k$ and $\psi ''\in k[f_2]$ 
with $\degw \psi ''\leq \min \{ (s-1)\delta ,\degw \phi \} $ 
such that $\phi =c''f_1+\psi ''$. 
If $\degw \phi <\degw f_1$, then $c''=0$.

{\rm (P11)} 
There exist $a,b,c,d\in k$ and 
$\psi \in k[f_2]$ with $\degw \psi \leq (s-1)\delta $ 
such that $g_1=f_1+af_3^2+cf_3+\psi $ 
and $g_2=f_2+bf_3+d$. 
If $a\neq 0$ or $b\neq 0$, then $\degw f_3\leq \degw f_2$. 
If $\degw f_3\leq \degw f_2$, then $s=3$. 
Furthermore, if $\psi $ belongs to $k$, then 
$a$, $b$ and $c$ are uniquely determined by 
$F$ in the following sense$:$ 
If $(F,G')$ satisfies the quasi Shestakov-Umirbaev condition 
for $G'=(g_1',g_2',g_3')\in \TT $, 
where $g_1'=f_1+a'f_3^2+c'f_3+\psi '$ 
and $g_2'=f_2+b'f_3+d'$ with $a',b',c',d',\psi '\in k$, 
then $a'=a$, $b'=b$ and $c'=c$. 

{\rm (P12)} The following equalities and inequality hold$:$ 
\begin{align*}
\degw df_1\wedge df_2&=\left\{
\begin{array}{ll}
\degw f_3+\degw df_2\wedge df_3 & \text{ if }a\neq 0 \\
\degw df_1\wedge df_3 & \text{ if }a=0\text{ and }b\neq 0 \\
\degw df_2\wedge df_3 & \text{ if }a=b=0\text{ and }c\neq 0 \\
\degw dg_1\wedge dg_2 & \text{ if }a=b=c=0  
\end{array}
\right. \\
\degw df_1\wedge df_3&=(s-2)\delta +\degw df_2\wedge df_3\\ 
\degw df_2\wedge df_3&\geq s\delta +\degw dg_1\wedge dg_2. 
\end{align*}

\medskip 

To show these properties, 
we set $\phi _i=g_i-f_i$ for $i=1,2,3$. 
Since $\degw g_3<\degw f_3$ by (SU5), 
we have $\phi _3^{\w }=-f_3^{\w }$ and 
$\degw \phi _3=\degw f_3$. 
Hence, $\degw \phi _3\leq \degw g_1$ and 
$\phi _3^{\w }$ does not belong to $k[g_1^{\w },g_2^{\w }]$ 
by (SU4). 
Set $U=\{ g_1,g_2\} $. 
Since $\phi _3$ is contained in $k[U]$ by (SU$1'$), 
it follows that $\degw \phi _3<\degw ^{U}\phi _3$. 
In view of (SU$3'$), we know that 
the assumptions of Lemma~\ref{lem:degS2}(ii) 
hold for 
$f=g_2$, $g=g_1$ and $\phi =\phi _3$. 
Therefore, 
there exists an odd number $s\geq 3$ 
such that 
$(g_1^{\w })^2\approx (g_2^{\w })^s$ and 
\begin{gather}
\degw f_3=
\degw \phi _3\geq (s-2)\delta +\degw dg_1\wedge dg_2, 
\label{eq:pfSUreduction1}\\
\degw dg_2\wedge d\phi _3\geq s\delta 
+\degw dg_1\wedge dg_2,\label{eq:pfSUreduction1.5}
\end{gather}
where $\delta =(1/2)\degw g_2$. 
This proves (P1) and (P2).

We show that $g_2$ is expressed as in (P11). 
By (SU$1'$), 
$\phi _2=g_2-f_2$ belongs to $k[f_3]$. 
Hence, $\phi _2=\sum _{i=0}^pb_if_3^i$ 
for some $b_0,\ldots ,b_p\in k$ with $b_p\neq 0$, 
where $p\in \Zn $. 
By (SU$2'$), 
$\degw \phi _2\leq \max \{ \degw g_2,\degw f_2\} 
=\degw g_2=2\delta $. 
By (\ref{eq:pfSUreduction1}), 
we get $\degw f_3>\delta $, 
since $s\geq 3$. 
Thus, we must have $p\leq 1$ and $\phi _2=b_1f_3+b_0$, 
for otherwise 
$\deg \phi _2=p\deg f_3>p\delta \geq 2\delta $, 
a contradiction. 
Therefore, $g_2$ is expressed as stated.

We show (P3) 
and the first assertion of (P8) for $(i,j)=(2,3),(3,2)$ 
by contradiction. 
Supposing that $\degw f_2\neq \degw g_2$, 
we have $\degw f_2<\degw g_2$ by (SU$2'$). 
Since $g_2=f_2+bf_3+d$ as shown above, 
it follows that 
$g_2^{\w }=bf_3^{\w }$ and $b\neq 0$. 
Hence, 
$f_3^{\w }$ belongs to $k[g_1^{\w },g_2^{\w }]$, 
a contradiction to (SU4). 
Therefore, 
$\degw f_2=\degw g_2$, proving (P3). 
Next, 
we show that $f_2^{\w }\not\approx f_3^{\w }$. 
Supposing that $f_2^{\w }\approx f_3^{\w }$, 
we have $\degw f_2=\degw f_3$. 
Hence, 
$\degw g_2=\degw f_3$ by (P3). 
Thus, 
$g_2^{\w }=f_2^{\w }+bf_3^{\w }$. 
Since $f_2^{\w }\approx f_3^{\w }$, 
we get $g_2^{\w }\approx f_3^{\w }$. 
This contradicts (SU4). 
Therefore, 
$f_2^{\w }\not\approx f_3^{\w }$. 
Now, suppose that $f_3^{\w }$ belongs to $k[f_2^{\w }]$. 
Then, 
$f_3^{\w }\approx (f_2^{\w })^l$ for some $l\geq 2$, 
since $f_2^{\w }\not\approx f_3^{\w }$. 
Hence, $\degw f_2<\degw f_3$. 
From $\degw f_2=\degw g_2=\degw (f_2+bf_3+d)$, 
we get $b=0$, and $f_2^{\w }=g_2^{\w }$. 
%Hence, $f_2^{\w }=g_2^{\w }$. 
Since $f_3^{\w }\approx (f_2^{\w })^l$, 
it follows that 
$f_3^{\w }\approx (g_2^{\w })^l$, 
a contradiction to (SU4). 
Therefore, 
$f_3^{\w }$ does not belong to $k[f_2^{\w }]$. 
Suppose that $f_2^{\w }$ belongs to $k[f_3^{\w }]$. 
Then, 
$f_2^{\w }\approx (f_3^{\w })^l$ for some $l\in \N $, 
where $l\geq 2$ as above. 
This is impossible, 
because $\deg f_2=2\delta $ by (P3) 
and $\deg f_3>\delta $ by (\ref{eq:pfSUreduction1}). 
Therefore, 
$f_2^{\w }$ does not belong to $k[f_3^{\w }]$.

Since $g_2-f_2$ is contained in $k[f_3]$ by (SU$1'$), 
it follows that  
$df_2\wedge df_3-dg_2\wedge df_3=d(f_2-g_2)\wedge df_3=0$. 
Moreover, 
$df_3=dg_3-d\phi _3$. 
Hence, 
\begin{equation}\label{eq:pfSUreduction2}
df_2\wedge df_3=dg_2\wedge df_3
=dg_2\wedge dg_3-dg_2\wedge d\phi _3. 
\end{equation}
By (\ref{eq:ineq-wedge}), (SU6), (P1) 
and (\ref{eq:pfSUreduction1.5}), 
we get 
\begin{align*}
\degw dg_2\wedge dg_3\leq
\degw g_2+\degw g_3
&<\degw g_1+\degw dg_1\wedge dg_2 \\
&=s\delta +\degw dg_1\wedge dg_2\leq \degw dg_2\wedge d\phi _3. 
\end{align*}
Then, it follows from (\ref{eq:pfSUreduction2}) that 
$\degw df_2\wedge df_3=\degw dg_2\wedge d\phi _3$. 
Therefore, we obtain 
\begin{equation}\label{eq:pfSUreduction4}
\degw df_2\wedge df_3
\geq s\delta +\degw dg_1\wedge dg_2 
\end{equation}
by (\ref{eq:pfSUreduction1.5}). 
This proves the last inequality of (P12).

The following lemma is useful in proving (P4), (P9), (P10) and (P11).

\begin{lemma}\label{lem:expression}
Assume that 
$\degw f_2=2\delta $ and 
$(s-2)\delta <\degw f_3\leq s\delta $ 
for some $\delta \in \Gamma $ and an odd number $s\geq 3$. 
Then, the following assertions hold$:$ 

{\rm (i)} 
If $\degw ^{S_1}\phi \leq s\delta $ 
for $\phi \in k[S_1]$, 
then there exist $a,c\in k$ and 
$\psi \in k[f_2]$ with $\degw \psi \leq (s-1)\delta $ 
such that $\phi =af_3^2+cf_3+\psi $. 
If $a\neq 0$, then $\deg f_3<\deg f_2$. 

{\rm (ii)} 
Assume that $\degw f_1>\deg f_2$. 
If $\degw ^{S_2}\phi \leq \deg f_2$ 
for $\phi \in k[S_2]$, 
then there exist $b,d\in k$ such that $\phi =bf_3+d$. 

{\rm (iii)} 
Assume that $\degw f_1\leq s\delta $. 
If $\degw ^{S_3}\phi \leq \degw f_1$ for $\phi \in k[S_3]$, 
then there exist 
$c'\in k$ and $\psi '\in k[f_2]$ with 
$\degw \psi '\leq \min \{ (s-1)\delta ,\degw ^{S_3}\phi \} $
such that $\phi =c'f_1+\psi '$. 
If $\degw ^{S_3}\phi <\degw f_1$, 
then $c'=0$. 

{\rm (iv)} If $\degw f_3\leq \degw f_2$, then $s=3$. 
\end{lemma}
\begin{proof}\rm
To show (i), 
write $\phi =\sum _{i,j}c_{i,j}f_2^if_3^j$, 
where $c_{i,j}\in k$ for each $i,j\in \Zn $. 
Since $\degw ^{S_1}\phi \leq s\delta $ by assumption, 
$\degw f_2^if_3^j\leq s\delta $ 
if $c_{i,j}\neq 0$ for $i,j\in \Zn $. 
We verify that, if $\deg f_2^if_3^j\leq s\delta $, 
then $i\leq (s-1)/2$ and $j=0$, 
or $i=0$ and $j=1,2$. 
This shows that $\phi $ can be expressed as in (i). 
It follows that 
$\degw f_2^if_3>2i\delta +(s-2)\delta \geq s\delta $ 
if $i\geq 1$. 
If $i>(s-1)/2$, then $2i>s$, 
since $s$ is an odd number. 
Hence,
$\degw f_2^i=2i\delta >s\delta $. 
If $j\geq 3$, then 
$\degw f_3^j>j(s-2)\delta \geq s\delta $, 
since $s\geq 3$. 
Thus, 
if $\deg f_2^if_3^j\leq s\delta $, 
then $(i,j)$ must be as stated above. 
Therefore, 
$\phi $ can be expressed as in (i). 
Assume that $a\neq 0$. 
Then, 
$\deg f_3^2\leq \deg ^{S_1}\phi \leq s\delta $. 
Since $(s-2)\delta <\deg f_3$, 
we get $2(s-2)<s$. 
Thus, $s<4$, and hence $s=3$. 
Therefore, 
$\deg f_3\leq (s/2)\delta =(3/2)\delta <2\delta =\deg f_2$. 
This proves (i).

We can prove (ii) and (iii), similarly. 
Actually, 
if $\deg f_1>\deg f_2$ 
and if $\degw f_1^if_3^j\leq \deg f_2$ 
for $i,j\in \Zn $, 
then $i=0$. 
Moreover, we have $j\leq 1$, 
since $\degw f_3^2>2(s-2)\delta \geq 2\delta =\deg f_2$. 
Therefore, 
$\phi =bf_3+d$ for some $b,d\in k$ in case (ii). 
To show (iii), 
assume that $\degw ^{S_3}\phi \leq \degw f_1$ for $\phi \in k[S_3]$. 
Clearly, 
$i=0$ or $(i,j)=(1,0)$ if 
$\degw f_1^if_2^j\leq \degw f_1$, 
while $i=0$ if 
$\degw f_1^if_2^j<\degw f_1$. 
Hence, 
$\phi =c'f_1+\psi '$ 
for some $c'\in k$ and $\psi '\in k[f_2]$ 
where $c'=0$ if $\degw ^{S_3}\phi <\degw f_1$. 
We note that $\deg \psi '\leq \deg ^{S_3}\phi $. 
Since 
$\deg ^{S_3}\phi \leq \deg f_1\leq s\delta $ by assumption, 
it follows that $\deg \psi '\leq s\delta $. 
This implies that $\deg \psi '\leq (s-1)\delta $, 
because $s$ is an odd number, 
and $\deg \psi '=\deg f_2^l=2l\delta $ 
with $l\in \Zn $ if $\psi '\neq 0$. 
Therefore, we obtain 
$\degw \psi '\leq \min \{ (s-1)\delta ,\degw ^{S_3}\phi \} $.

The assertion (iv) follows from 
$(s-2)\delta <\degw f_3\leq \degw f_2=2\delta $. 
\end{proof}

We show (P4) using Lemma~\ref{lem:expression}(i). 
Since $\degw f_2=\degw g_2=2\delta $ by (P3), 
and since $(s-2)\delta <\degw f_3\leq s\delta $ 
by (\ref{eq:pfSUreduction1}) and (SU4), 
it suffices to check that $\degw ^{S_1}\phi \leq s\delta $. 
Supposing the contrary, 
we have $\deg \phi <\degw ^{S_1}\phi $, 
since $\degw \phi \leq \degw g_1=s\delta $ 
by the assumption of (P3). 
As shown above, 
$f_i^{\w }$ does not belong to $k[f_j^{\w }]$ 
for $(i,j)=(2,3),(3,2)$. 
Hence, 
$\degw \phi >\degw df_2\wedge df_3$ 
by Lemma~\ref{lem:degS2}(i). 
Since 
$\degw df_2\wedge df_3>s\delta $ 
by (\ref{eq:pfSUreduction4}),  
we get $\degw \phi >s\delta $, 
a contradiction. 
Thus, $\degw ^{S_1}\phi \leq s\delta $, 
and thereby proving (P4).

We complete the proof of the former part of (P11). 
Since $\phi _1=g_1-f_1$ belongs to $k[S_1]$ by (SU$1'$), 
and since $\degw \phi _1\leq 
\max \{ \deg g_1,\deg f_1\} =\degw g_1=s\delta $ by (SU$2'$), 
we know by (P4) that $g_1=f_1+\phi _1$ is expressed as in (P11). 
If $a\neq 0$, then $\deg f_3<\deg f_2$ 
by the last assertion of Lemma~\ref{lem:expression}(i). 
Since $\deg f_2=\deg g_2$ and $g_2=f_2+bf_3+d$, 
we get $\deg f_3\leq \deg f_2$ if $b\neq 0$. 
By Lemma~\ref{lem:expression}(iv), 
$\degw f_3\leq \degw f_2$ implies $s=3$. 
We have thus proved the former part of (P11).

We show that 
the conditions listed before Lemma~\ref{lem:determinant} 
and the inequality 
$\degw dk_1'\wedge dk_2'<\degw dk_2\wedge dk_3$ 
hold for 
$k_i=f_i$ for $i=1,2,3$ and $k_i'=g_i$ for $i=1,2$. 
By the former part of (P11), 
$k_1'$ and $k_2'$ are expressed in terms of 
$k_1$, $k_2$ and $k_3$ as required. 
Since $\degw g_2<\degw g_1$ by (SU$3'$), 
we get (a). 
Since $\degw g_1-\degw g_2=(s-2)\delta $, 
(b) follows from (\ref{eq:pfSUreduction1}). 
By (P3), 
(c) is equivalent to $\degw \psi <\degw k_1'$, 
which follows from 
$\degw \psi \leq (s-1)\delta <\degw g_1$. 
The rest of the conditions 
are due to (\ref{eq:pfSUreduction4}), 
since $df_2\wedge df_3=dg_2\wedge df_3$ as mentioned. 
Therefore, we obtain the estimation of 
$\degw df_1\wedge df_2$ described in (P12) 
from Lemma~\ref{lem:determinant}(ii). 
Owing to (\ref{eq:detfirst}), 
we have 
\begin{align}\label{eq:pfSUred5.5}
\degw df_1\wedge df_3&=(s-2)\delta +\degw df_2\wedge df_3, 
\end{align}
the second equality of (P12). 
The uniqueness of $a$, $b$ and $c$ claimed in (P11) 
follows from Lemma~\ref{lem:determinant}(iii). 
This completes the proofs of (P11) and (P12).

Here, we remark that 
\begin{align}\label{eq:pfSUred6}
\degw df_1\wedge df_3\geq 2(s-1)\delta +\degw dg_1\wedge dg_2
\end{align}
follows from 
(\ref{eq:pfSUreduction4}) and (\ref{eq:pfSUred5.5}). 
Since $\degw f_1+\degw f_3\geq \degw df_1\wedge df_3$, 
we obtain that 
\begin{equation}\label{eq:pfSUreduction5}
\degw f_1\geq 
2(s-1)\delta +\degw dg_1\wedge dg_2-\degw f_3. 
\end{equation}

Now, we show (P5). 
By the assumption of (P5), 
we have $\degw f_1<\degw g_1$. 
Hence, 
$g_1^{\w }=(f_1+\phi _1)^{\w }=\phi _1^{\w }$, 
and so $\deg \phi _1=s\delta $. 
Since $g_1^{\w }\not \approx f_3^{\w }$ by (SU4), 
we get $\phi _1^{\w }\not\approx f_3^{\w }$. 
By (P11), 
we have $\phi _1=af_3^2+cf_3+\psi $, 
in which $\degw \psi \leq (s-1)\delta $. 
From this, 
it follows that $a\neq 0$, 
for otherwise $\phi _1^{\w }=cf_3^{\w }$, 
a contradiction. 
Hence, $s=3$ by (P11). 
Moreover, 
$\phi _1^{\w }\approx (f_3^{\w })^2$, 
and thus $g_1^{\w }\approx (f_3^{\w })^2$. 
Therefore, $\degw f_3=(1/2)\deg g_1=(3/2)\delta $. 
The last inequality of (P5) follows 
from (\ref{eq:pfSUreduction5}).

We show (P6) and (P7) with the aid of (P5). 
If $\degw g_1=\degw f_1$, 
then (P6) is clear, 
since $\degw g_2=\degw f_2$ by (P3), 
and $\degw g_3<\degw f_3$ by (SU5). 
Assume that $\degw f_1<\degw g_1$. 
Then, 
$$
\deg f_1+\deg f_3
>\frac{5}{2}\delta +\deg dg_1\wedge dg_2+\frac{3}{2}\delta 
=4\delta +\deg dg_1\wedge dg_2
$$
by (P5). 
On the other hand, 
since $\deg g_2=2\delta $, 
and $\deg g_1=s\delta =3\delta $ by (P5), 
it follows from (SU6) that 
$$
\degw g_1+\degw g_3<\deg g_1+\deg g_1-\deg g_2+\degw dg_1\wedge dg_2
=4\delta +\degw dg_1\wedge dg_2. 
$$
Therefore, 
$\deg G<\deg F$ by (P3). 
This proves (P6). 
If $\degw f_1=\degw g_1$, 
then $\degw f_2<\degw f_1$ by (SU$3'$), 
and $\degw f_3\leq \degw f_1$ by (SU4). 
If $\deg f_1<\deg g_1$, 
then $\degw f_1>(5/2)\delta $ 
and $\degw f_3=(3/2)\delta $ by (P5). 
Hence, 
$\degw f_i<\degw f_1$ for $i=2,3$. 
This proves the first two statements of (P7). 
The last statement of (P7) 
follows from the conditions that 
$(5/2)\delta <\degw f_1\leq \degw g_1=s\delta $, 
$\degw f_2=2\delta $ 
and $(s-2)\delta <\degw f_3\leq \degw g_1$.

Let us complete the proof of (P8). 
First, 
we show that 
$\deg f_i\neq l\deg f_j$ holds for any $l\in \N $ 
for $(i,j)=(1,2),(2,1)$, 
which proves that 
$f_i^{\w }$ does not belong to $k[f_j^{\w }]$. 
In case $\degw f_1=\degw g_1$, 
we have $2\deg f_1=s\deg f_2$ by (P1) and (P3). 
Since $s\geq 3$ is an odd number, 
the assertion is true. 
In case $\degw f_1<\degw g_1$, 
we have $(5/2)\delta <\degw f_1<3\delta $ by (P5). 
Since $\degw f_2=2\delta $, 
the assertion is readily verified. 
Thus, 
$f_i^{\w }$ does not belong to $k[f_j^{\w }]$ 
for $(i,j)=(1,2),(2,1)$. 
Next, 
suppose to the contrary that $f_3^{\w }$ belongs to $k[f_1^{\w }]$. 
Since $\degw f_3\leq \degw f_1$ by (P7), 
it follows that $f_3^{\w }\approx f_1^{\w }$. 
In view of (P5), 
we get $\degw f_1=\degw g_1$. 
Hence, 
$g_1^{\w }=f_1^{\w }+cf_3^{\w }$. 
Consequently, 
we obtain $f_3^{\w }\approx g_1^{\w }$, 
a contradiction to (SU4). 
Therefore, 
$f_3^{\w }$ does not belong to $k[f_1^{\w }]$. 
Since the cases $(i,j)=(2,3),(3,2)$ are done, 
this completes the proof of the former part of (P8). 
For the latter part, 
assume that 
$f_1^{\w }$ belongs to $k[f_3^{\w }]$. 
Then, $f_1^{\w }\approx (f_3^{\w })^l$ 
for some $l\in \N $. 
Since $f_3^{\w }$ does not belong to $k[f_1^{\w }]$, 
it follows that $l\geq 2$. 
Then, 
we must have $s=3$ and $l=2$. 
In fact, if $s\geq 5$ or $l\geq 3$, 
then $s\leq l(s-2)$, and so 
$$
\degw f_1\leq \degw g_1=s\delta \leq l(s-2)\delta 
<l\degw f_3, 
$$
which contradicts $f_1^{\w }\approx (f_3^{\w })^l$. 
Thus, 
$f_1^{\w }\approx (f_3^{\w })^2$. 
If $\degw f_3\neq (3/2)\delta $, 
then $\degw f_1=\degw g_1$ by (P5), 
and hence 
$$
\degw f_3=\frac{1}{2}\degw f_1=\frac{1}{2}\degw g_1
=\frac{1}{2}s\delta =\frac{3}{2}\delta , 
$$
a contradiction. 
Therefore, 
$\degw f_3=(3/2)\delta $. 
This completes the proof of (P8).

We show (P9) using Lemma~\ref{lem:expression}(ii). 
Since $\deg f_2<\deg f_1$ by (P7), 
we verify that, 
if $\deg \phi \leq \deg f_2$ for $\phi \in k[S_2]$, 
then $\degw ^{S_2}\phi \leq \degw f_2$. 
Supposing the contrary, 
we get $\degw \phi <\degw ^{S_2}\phi $. 
By Lemma~\ref{lem:degS1}(i), 
there exist $p,q\in \N $ with $\gcd (p,q)=1$ such that 
$(f_3^{\w })^p\approx (f_1^{\w })^q$ and 
\begin{align}
2\delta =\deg f_2\geq \degw \phi 
&\geq q\degw f_1+\degw df_1\wedge df_3-\degw f_1-\degw f_3 \notag \\
&\geq (q-1)\degw f_1-\degw f_3+2(s-1)\delta +\degw dg_1\wedge dg_2, 
\label{eq:pfSUred7}
\end{align}
where the last inequality is due to (\ref{eq:pfSUred6}). 
Since $f_3^{\w }$ does not belong to $k[f_1^{\w }]$ by (P8), 
we have $p\geq 2$. 
If $\degw f_1<\degw g_1$, 
then $s=3$, $\degw f_1>(5/2)\delta $ 
and $\degw f_3=(3/2)\delta $ by (P5), 
and hence the right-hand side of (\ref{eq:pfSUred7}) 
is greater than
$$
(q-1)\frac{5}{2}\delta -\frac{3}{2}\delta +4\delta 
+\degw dg_1\wedge dg_2
>\frac{5}{2}q\delta >2\delta ,
$$
a contradiction. 
Thus, 
$\degw f_1=\degw g_1=s\delta $. 
Then, the right-hand side of (\ref{eq:pfSUred7}) is 
at least 
$$
(q-1)s\delta -\frac{q}{p}s\delta +2(s-1)\delta 
+\degw dg_1\wedge dg_2
>\frac{qs}{p}(p-1)\delta +(s-2)\delta, 
$$
which is less than $2\delta $ by (\ref{eq:pfSUred7}). 
Hence, 
$s=3$ and $(3q/p)(p-1)<1$. 
Since $p\geq 2$, 
it follows that $3\leq 3q<1+1/(p-1)\leq 2$, 
a contradiction. 
Therefore, 
we conclude that $\degw ^{S_2}\phi \leq \degw f_2$, 
and thereby proving (P9).

To show (P10), 
assume that $k[S_3]\neq k[g_1,g_2]$, 
and take $\phi \in k[S_3]$ such that 
$\degw \phi \leq \degw f_1$. 
By virtue of Lemma~\ref{lem:expression}(iii), 
it suffices to check that $\degw \phi =\degw ^{S_3}\phi $. 
Supposing the contrary, 
we get $\degw \phi <\degw ^{S_3}\phi $. 
By (P8), 
$f_i^{\w }$ does not belong to $k[f_j^{\w }]$ 
for $(i,j)=(1,2),(2,1)$. 
Hence, 
$\degw \phi >\degw df_1\wedge df_2$ by Lemma~\ref{lem:degS2}(i). 
Since $k[S_3]\neq k[g_1,g_2]$, 
we must have $(a,b,c)\neq (0,0,0)$. 
Hence, 
$\degw df_1\wedge df_2\geq \degw df_2\wedge df_3>s\delta $ 
by (P12). 
Thus, 
$\degw \phi >s\delta $. 
This is a contradiction, 
because $\degw \phi \leq \degw f_1$ 
and $\deg f_1\leq \deg g_1=s\delta $. 
Therefore, 
$\degw \phi =\degw ^{S_3}\phi $, 
and thereby (P10) is proved.

We have thus proved the following theorem.

\begin{theorem}\label{thm:SUreduction}
If $(F,G)$ satisfies the quasi 
Shestakov-Umirbaev condition 
for $F,G\in \TT $, 
then {\rm (P1)--(P12)} hold for $F$ and $G$. 
\end{theorem}

The following proposition 
is a consequence of Theorem~\ref{thm:SUreduction}.

\begin{proposition}\label{prop:equivalence}
{\rm (i)} 
If $(F,G)$ satisfies the quasi Shestakov-Umirbaev condition 
for $F,G\in \TT $, 
then there exist $E_i\in \E _i$ for $i=1,2$ 
with $\degw G\circ E_1=\degw G$ such that 
$(F,G\circ E_1\circ E_2)$ satisfies the Shestakov-Umirbaev condition. 

{\rm (ii)} For $F\in \TT $, 
it follows that 
$F$ admits a Shestakov-Umirbaev reduction if and only if 
$F$ admits a quasi Shestakov-Umirbaev reduction. 
\end{proposition}
\begin{proof}\rm
(i) Assume that $g_1$ and $g_2$ 
are expressed as in (P11). 
Take $\Psi \in k[y]$ such that $\Psi (f_2)=\psi $, 
and define $E_i\in \E _i$ for $i=1,2$ 
by $E_1(y_1)=y_1-\Psi (y_2-d)$ 
and $E_2(y_2)=y_2-d$. 
Then, 
$(E_1\circ E_2)(y_i)=E_i(y_i)$ for $i=1,2$. 
Set $G'=G\circ E_1\circ E_2$ 
and $g_i'=G'(y_i)$ for each $i$. 
We show that $(F,G')$ satisfies (SU1)--(SU6). 
By definition, 
$g_2'=g_2-d=f_2+bf_3$. 
If $b=0$, 
then $\Psi (g_2-d)=\Psi (f_2)=\psi $. 
Hence, 
$g_1'=g_1-\Psi (g_2-d)=f_1+af_3^2+cf_3$. 
Assume that $b\neq 0$. 
Then, 
$s=3$ by (P11). 
Hence, 
$\deg \psi \leq (s-1)\delta=2\delta $. 
Since $\psi $ belongs to $k[f_2]$ 
and since $\deg f_2=2\delta $ by (P3), 
we may write $\psi =ef_2+e'$, 
where $e,e'\in k$. 
Then, $\Psi =ey_2+e'$, and so 
\begin{equation}\label{eq:pf:equiv}
g_1'=g_1-(e(g_2-d)+e')=f_1+af_3^2+(c-be)f_3. 
\end{equation}
Thus, 
$g_1'$ and $g_2'$ are expressed as in (SU1). 
From the construction of $g_1'$ and $g_2'$, 
it follows that $k[g_1',g_2']=k[g_1,g_2]$. 
Since $(F,G)$ satisfies (SU$1'$) by assumption, 
$g_3'-f_3=g_3-f_3$ belongs to $k[g_1,g_2]$, 
and hence belongs to $k[g_1',g_2']$. 
Therefore, 
$(F,G')$ satisfies (SU1). 
We remark that $(F,G)$ satisfies 
(SU2) and (SU3) on account of (P3), (SU$2'$), and (P1), 
and satisfies (SU4)--(SU6) 
by the definition of the quasi Shestakov-Umirbaev condition. 
From this, 
we can easily conclude that $(F,G')$ satisfies (SU2)--(SU6) 
on the assumption that $dg_1'\wedge dg_2'=dg_1\wedge dg_2$ 
and $(g_i')^{\w }=g_i^{\w }$ for $i=1,2$. 
So, we verify these equalities. 
Since $g_2'=g_2-d$, 
we have $(g_2')^{\w }=g_2^{\w }$ and $dg_2'=dg_2$. 
Since $dg_1'=dg_1-\Psi ^{(1)}(g_2-d)dg_2$, 
we get $dg_1'\wedge dg_2'=dg_1\wedge dg_2$. 
If $b=0$, 
then $g_1'=g_1-\psi $. 
Since $\deg \psi \leq (s-1)\delta <s\delta =\deg g_1$, 
we have $(g_1')^{\w }=g_1^{\w }$. 
If $b\neq 0$, 
then $\deg f_3\leq \deg f_2$ by (P11), 
and so $\deg f_3<\deg g_1$ by (SU$3'$) and (P3). 
Hence, 
$(g_1')^{\w }=(g_1-\psi -bef_3)^{\w }=g_1^{\w }$. 
Thus, it holds that 
$dg_1'\wedge dg_2'=dg_1\wedge dg_2$ 
and $(g_i')^{\w }=g_i^{\w }$ for $i=1,2$. 
Thereby, $(F,G')$ satisfies (SU2)--(SU6). 
Therefore, 
$(F,G')$ satisfies the Shestakov-Umirbaev condition. 
Since $G\circ E_1=(g_1',g_2,g_3)$ and $\degw g_1'=\degw g_1$, 
we have $\degw G\circ E_1=\degw G$.

(ii) It is clear that $F$ admits a quasi Shestakov-Umirbaev 
reduction if $F$ admits a Shestakov-Umirbaev reduction. 
The converse follows from (i). 
\end{proof}

The following remark is readily verified. 
If $(F,G)$ satisfies 
(SU$2'$), (SU$3'$), (SU4), (SU5) and (SU6), 
then so does $(F',G')$. 
Here, 
$F'=(f_1',f_2',f_3')$ is an element of $\TT$ 
such that $\degw f_i'\leq \degw f_i$ for $i=1,2$ 
and $(f_3')^{\w }\approx f_3^{\w }+h$ 
for some $h\in k[g_1^{\w },g_2^{\w }]$, 
and $G'=(c_1g_1,c_2g_2,c_3g_3)$, 
where $c_1,c_2,c_3\in k\sm \zs $. 
Note that 
$F':=F\circ E$ satisfies this condition for $E\in \E _i$ 
such that $\degw F\circ E\leq \degw F$ if $i\in \{ 1,2\} $, 
and $(F\circ E)(y_3)^{\w }\approx f_3^{\w }+h$ 
for some $h\in k[g_1^{\w },g_2^{\w }]$ if $i=3$. 
Moreover, 
$(F',G')$ satisfies (SU$1'$) 
%when $F'=F\circ E$ for $E\in \E _i$ 
if the following conditions hold: 

(i) $c_1g_1-f_1'$ belongs to $k[f_2,f_3]$ if $i=1$ 
and $c_2=c_3=1$; 

(ii) $c_1g_1-f_1$ 
and $c_2g_2-f_2'$ respectively belong to $k[f_2',f_3]$ and 
$k[f_3]$ if $i=2$ and $c_3=1$; 

(iii) $c_1g_1-f_1$, $c_2g_2-f_2$ and $c_3g_3-f_3'$ 
respectively belong to $k[f_2,f_3']$, $k[f_3']$ 
and $k[g_1,g_2]$ if $i=3$.

To end this section, 
we prove a proposition which will be 
used in the proof of Theorem~\ref{thm:main1}. 
We note that the case (ii) does not arise 
if $\rank \w =n$, 
since $\deg f_j=\deg f_3$ implies 
$f_j^{\w }\approx f_3^{\w }$ if $\rank \w =n$, while 
$f_j^{\w }\not\approx f_3^{\w }$ for $j=1,2$ by (P8).

\begin{proposition}\label{prop:c1-c6}
Assume that $(F,G)$ satisfies 
the quasi Shestakov-Umirbaev condition. 
If $\degw F\circ E\leq \degw F$ for $E\in \E _i$, 
then the following assertions hold for $F':=F\circ E$, 
where $i\in \{ 1,2,3\} $. 

{\rm (i)} 
If $i=1$ or $i=2$, 
or if $i=3$, $k[f_1,f_2]\neq k[g_1,g_2]$ 
and $\degw f_j\neq \degw f_3$ for $j=1,2$, 
then $(F',G)$ satisfies 
the quasi Shestakov-Umirbaev condition. 

{\rm (ii)} 
If $i=3$, $k[f_1,f_2]\neq k[g_1,g_2]$ 
and $\degw f_j=\degw f_3$ for some $j\in \{ 1,2\} $, 
then there exists $u\in k\sm \zs $ such that 
$(F',G')$ or $(F'_{\tau },G'')$ 
satisfies the quasi Shestakov-Umirbaev condition. 
Here, 
$G'=(g_1',g_2',ug_3)$ and $G''=(g_1',g_2',-ug_3)$ 
with $g_j'=u^{-1}g_j$ and $g_l'=g_l$ 
for $l\in \{ 1,2\} \sm \{ j\} $, 
and $\tau =(j,3)$. 
\end{proposition}
\begin{proof}\rm
Set $f'_i=F'(y_i)$ and $\phi _i=f_i'-f_i$. 
Then, $\degw f_i'\leq \degw f_i$, 
since $\degw F'\leq \degw F$ by assumption. 
Hence, 
$\degw \phi _i\leq \max \{ \deg f_i',\deg f_i\} \leq \degw f_i$. 
We note that 
$\phi _i$ belongs to $k[S_i]$. 
Besides, 
$g_1-f_1$, $g_2-f_2$ and $g_3-f_3$ belong to 
$k[f_2,f_3]$, $[f_3]$ and $k[g_1,g_2]$ by (SU$1'$), 
respectively, 
since $(F,G)$ satisfies 
the quasi Shestakov-Umirbaev condition.

(i) First, assume that $i\in \{ 1,2\} $, 
or $i=3$ and $\phi _3$ is contained in $k$. 
Since $\degw F'\leq \degw F$, 
we know by the remark above that 
$(F',G)$ satisfies (SU$2'$), (SU$3'$), (SU4), (SU5) and (SU6) 
if $i\in \{ 1,2\} $. 
If $i=3$, then $(f_3')^{\w }=f_3^{\w }$, 
since $f_3'-f_3=\phi _3$ belongs to $k$ by assumption. 
Hence, 
$(F',G)$ satisfies the five conditions similarly. 
We check that $(F',G)$ satisfies (SU$1'$). 
If $i=1$, 
then $g_1-f_1'=(g_1-f_1)-\phi _1$ belongs to $k[S_1]$, 
since so do $g_1-f_1$ and $\phi _1$. 
If $i=2$, 
then $\phi _2$ belongs to $k[f_3]$ by (P9), 
because $\phi _2$ is an element of $k[S_2]$ 
such that $\degw \phi _2\leq \degw f_2$. 
Hence, 
$k[f_2',f_3]=k[f_2,f_3]$, 
to which $g_1-f_1$ belongs. 
Moreover, 
$g_2-f_2'=(g_2-f_2)-\phi _2$ belongs to $k[f_3]$, 
since so does $g_2-f_2$. 
If $i=3$, then $\phi _3$ is contained in $k$. 
Hence, 
$g_1-f_1$ and $g_2-f_2$ belong to 
$k[f_2,f_3']=k[f_2,f_3]$ and $k[f_3']=k[f_3]$, 
respectively. 
Moreover, 
$g_3-f_3'=(g_3-f_3)-\phi _3$ 
belongs to $k[g_1,g_2]$, 
since so does $g_3-f_3$. 
Thus, 
$(F',G)$ satisfies (SU$1'$) in each case. 
Therefore, 
$(F',G)$ satisfies the quasi Shestakov-Umirbaev condition.

Next, assume that $i=3$ and 
$\phi _3$ is not contained in $k$. 
We show that $(f_3')^{\w }=f_3^{\w }+\alpha (g_2^{\w })^p$ 
for some $\alpha \in k$ and $p\in \N $, 
which implies that 
$(G',F)$ satisfies (SU$2'$), (SU$3'$), (SU4), (SU5) and (SU6) 
by the remark. 
Since $f_3'=f_3+\phi _3$, $\deg \phi _3\leq \deg f_3$, 
and $f_3^{\w }$ does not belong to $k[g_2^{\w }]$ by (SU4), 
it suffices to check that $\phi _3^{\w }\approx (g_2^{\w })^p$ 
for some $p\in \N $. 
We establish that $\phi _3$ belongs to $k[f_2]$, 
and $f_2^{\w }=g_2^{\w }$. 
Since $\degw f_1\neq \degw f_3$ by assumption, 
we have $\degw f_3<\degw f_1$ by (P7). 
Hence, 
$\degw \phi _3<\degw f_1$. 
Since $k[f_1,f_2]\neq k[g_1,g_2]$ by assumption, 
it follows from (P10) that 
$\phi _3$ belongs to $k[f_2]$. 
Since $\phi _3$ is not contained in $k$, 
we get $\degw f_2\leq \deg \phi _3$. 
Hence, 
$\deg f_2\leq \deg f_3$. 
Since $\deg f_2\neq \deg f_3$ by assumption, 
we get $\deg f_2<\deg f_3$. 
By (P11), it follows that $b=0$, 
where we write $g_2=f_2+bf_3+d$. 
Hence, $g_2=f_2+d$, 
and so $g_2^{\w }=f_2^{\w }$. 
Thus, we have proved that 
$(f_3')^{\w }=f_3^{\w }+\alpha (g_2^{\w })^p$ 
for some $\alpha \in k$ and $p\in \N $, 
and thereby proved that 
$(G',F)$ satisfies the five conditions. 
As for (SU$1'$), 
$g_2-f_2=d$ clearly belongs to $k[f_3']$. 
Since $\phi _3$ is contained in $k[f_2]$, 
we know that 
$g_1-f_1$ and $g_3-f_3'=(g_3-f_3)-\phi _3$ 
belong to $k[f_2,f_3']=k[f_2,f_3]$ and $k[g_1,g_2]=k[g_1,g_2,f_2]$, 
respectively. 
Thus, 
$(F',G)$ satisfies (SU$1'$), 
and therefore satisfies the quasi Shestakov-Umirbaev condition.

(ii) 
By (P7), 
$\degw f_2<\degw f_1=\degw f_3$ if $j=1$, 
and $\degw f_3=\degw f_2<\degw f_1$ if $j=2$. 
In view of (P5), 
$\degw f_1=\degw g_1$ in either case. 
Furthermore, 
in case $j=1$, 
we can  write $g_1=f_1+cf_3+\psi $ and $g_2=f_2+d$ by (P11), 
since $a=b=0$ if $\deg f_2<\deg f_3$. 
We claim that $g_j=f_j+\alpha f_3+\psi ^1$ 
and $\phi _3=\beta f_j+\psi ^2$ 
for some $\alpha ,\beta \in k$, 
and $\psi ^p\in k[f_2]$ for $p=1,2$ such that 
$\degw \psi ^p<\degw f_1$ if $j=1$, 
and $\deg \psi ^p\leq 0$ if $j=2$. 
In fact, 
$g_1$ has such an expression if $j=1$ as mentioned, 
since $\deg \psi \leq (s-1)\delta <s\delta =\deg g_1=\deg f_1$. 
If $j=1$, 
then $\deg \phi _3\leq \deg f_3=\deg f_1$. 
Hence, 
it follows from (P10) that 
$\phi _3$ is expressed as claimed. 
If $j=2$, 
then $\degw \phi _3\leq \degw f_3<\degw f_1$, 
and so $\phi _3$ belongs to $k[f_2]$ by (P10). 
Since $\degw f_2=\degw f_3$ 
and $\deg \phi _3\leq \deg f_3$, 
we have $\phi _3=\beta f_2+\psi ^2$ 
for some $\beta ,\psi ^2\in k$. 
The expression of $g_2$ is due to (P11). 
Therefore, 
$g_j$ and $\phi _3$ have expressions as claimed. 
Observe that 
$\degw \psi ^p<\degw f_j$ for $p=1,2$. 
Moreover, 
$\deg f_j=\deg f_3$, 
while $f_j^{\w }\not\approx f_3^{\w }$ by (P8). 
Thus, we have 
\begin{equation}\label{eq:fjo}
g_j^{\w }=f_j^{\w }+\alpha f_3^{\w },\quad 
(f_3')^{\w }
=(f_3+\phi _3)^{\w }
=f_3^{\w }+\beta f_j^{\w }
=(1-\alpha \beta )f_3^{\w }+\beta g_j^{\w }. 
\end{equation}

First, assume that $\alpha \beta \neq 1$. 
We show that $(F',G')$ satisfies 
the quasi Shestakov-Umirbaev condition for $u=1-\alpha\beta $. 
From the second equality of (\ref{eq:fjo}), 
we get $(f_3')^{\w }\approx f_3^{\w }+u^{-1}\beta g_j^{\w }$. 
Hence, 
$(F',G')$ satisfies 
(SU$2'$), (SU$3'$), (SU4), (SU5) and (SU6) 
as remarked. 
We check (SU$1'$). 
If $j=1$, 
then $g_2'=g_2$, 
and $g_2'-f_2=g_2-f_2=d$ belongs to $k[f_3']$. 
If $j=2$, 
then $f_3'-f_3=\phi _3$ is contained in $k[f_2]$ 
by (P10) as mentioned. 
Hence, $k[f_2,f_3']=k[f_2,f_3]$, 
to which $g_1'-f_1=g_1-f_1$ belongs. 
A direct forward computation shows that 
\begin{align*}
g_j'-f_j
&=\frac{1}{u}g_j-f_j
=\frac{1}{1-\alpha \beta }(f_j+\alpha f_3+\psi ^1)-f_j
=\frac{1}{1-\alpha \beta }(\alpha f_3'+\psi ^1-\alpha \psi ^2), \\
ug_3-f_3'
&=(1-\alpha \beta )g_3-(f_3+\beta f_j+\psi ^2)
=(1-\alpha \beta )(g_3-f_3)-\beta g_j+\beta \psi ^1-\psi ^2. 
\end{align*}
By the first expression, 
$g_j'-f_j$ belongs to $k[f_2,f_3']$ if $j=1$, 
and to $k[f_3']$ if $j=2$, 
since $\psi ^1$ and $\psi ^2$ 
belong to $k[f_2]$ if $j=1$, 
and to $k$ if $j=2$. 
We show that $ug_3-f_3'$ belongs to $k[g_1,g_2]$. 
Since $g_3-f_3$ and $g_j$ belong to $k[g_1,g_2]$, 
it suffices to check that $\psi ^1$ and $\psi ^2$ 
belong to $k[g_1,g_2]$. 
This is obvious if $j=2$. 
If $j=1$, then 
$g_2=f_2+d$. 
Hence, $k[g_2]=k[f_2]$, to which 
$\psi ^1$ and $\psi ^2$ belong. 
Thus, 
$ug_3-f_3'$ belongs to $k[g_1,g_2]$. 
This proves that 
$(F',G')$ satisfies (SU$1'$), 
Therefore, 
$(F',G')$ satisfies the quasi Shestakov-Umirbaev condition.

Next, assume that $\alpha \beta =1$. 
We show that 
$(F'_{\tau },G'')$ satisfies the quasi Shestakov-Umirbaev condition 
for $u=\alpha $. 
Write $F'_{\tau }=(h_1,h_2,h_3)$. 
Then, 
$\degw h_j=\degw f_3'\leq \degw f_3=\degw f_j$ 
and $\degw h_l=\degw f_l$ for $l\in \{ 1,2\} \sm \{ j\} $. 
By the first equality of (\ref{eq:fjo}), 
we get $h_3^{\w }=f_j^{\w }=-\alpha f_3^{\w }+g_j^{\w }$, 
since $\beta ^{-1}=\alpha $. 
Hence, 
$(F'_{\tau },G'')$ satisfies 
(SU$2'$), (SU$3'$), (SU4), (SU5) and (SU6) by the remark. 
We check (SU$1'$). 
As in case of $\alpha \beta \neq 1$ above, 
$g_2''-h_2=g_2-f_2=d$ belongs to $k[h_3]$ if $j=1$, 
and $g_1''-h_1=g_1-f_1$ belongs to $k[h_2,h_3]=k[f_3',f_2]=k[f_2,f_3]$ 
if $j=2$. 
A direct forward computation shows that 
\begin{align*}
g_j''-h_j&=\frac{1}{\alpha }g_j-f_3'=
\frac{1}{\alpha }(f_j+\alpha f_3+\psi ^1)-(f_3+\beta f_j+\psi ^2)
=\frac{1}{\alpha }\psi ^1-\psi ^2, \\
-ug_3-h_3&=-\alpha g_3-f_j=-\alpha (g_3-f_3)-\alpha f_3-f_j
=-\alpha (g_3-f_3)-g_j+\psi ^1. 
\end{align*}
By the first expression, 
$g_j''-h_j$ belongs to $k[h_2,h_3]=k[f_2,f_1]$ if $j=1$, 
and to $k[h_3]$ if $j=2$. 
As in case of $\alpha \beta \neq 1$ above, 
$-ug_3-h_3$ belongs to $k[g_1,g_2]$ by the second expression. 
Thus, 
$(F',G)$ satisfies (SU$1'$). 
Therefore, 
$(F',G)$ satisfies the quasi Shestakov-Umirbaev condition. 
\end{proof}

\section{Analysis of reductions}\label{sect:heart}
\setcounter{equation}{0}

In this section, 
we prove some technical propositions which will be needed 
in the proof of Theorem~\ref{thm:main1}. 
First, 
we show a useful lemma.

\begin{lemma}\label{lem:su}
Assume that $(F_{\sigma },G)$ 
satisfies the quasi Shestakov-Umirbaev condition 
for some $\sigma \in \sym _3$. 
Then, the following assertions hold$:$

{\rm (i)} If $\degw f_i<\degw f_1$ for $i=2,3$, 
then $\sigma (1)=1$. 

{\rm (ii)} 
If $(F_{\sigma },G)$ 
satisfies the Shestakov-Umirbaev condition, 
and if $\sigma (1)=1$ and 
$\degw df_1\wedge df_2<\degw f_1$, 
then $\sigma =\id $ 
and $(f_1,f_2)=(g_1,g_2)$. 

{\rm (iii)} 
If $\degw f_3<\degw f_2<\degw f_1$ and 
$2\degw f_1<3\degw f_2$, 
then either $3\degw f_2=4\degw f_3$, 
or $2\degw f_1=s\degw f_3$ 
for some odd number $s\geq 3$. 

{\rm (iv)} 
If $\degw df_2\wedge df_3<\degw df_1\wedge df_3
<\degw df_1\wedge df_2$, 
then one of the following holds$:$ 

{\rm (1)} $\sigma =\id $ and $2\degw g_1=3\degw f_2$. 

{\rm (2)} 
$\sigma =(1,2,3)$ and $2\degw f_2=s\degw f_3$ 
for some odd number $s\geq 3$. 
\end{lemma}
\begin{proof}\rm
%We remark that (P1)--(P12) hold for $F_{\sigma }$ and $G$ 
%by Theorem~\ref{thm:SUreduction}, 
%since $(F_{\sigma },G)$ 
%satisfies the Shestakov-Umirbaev condition by assumption, 
%and hence satisfies the quasi Shestakov-Umirbaev condition. 
(i) By (P7), 
we have $\degw f_{\sigma (i)}\leq \degw f_{\sigma (1)}$ for $i=2,3$. 
Hence, $\sigma (1)=1$ if $\deg f_i<\deg f_1$ for $i=2,3$.

(ii) 
Since $\sigma (1)=1$ by assumption, 
we have 
$$
\deg f_1=\degw f_{\sigma (1)}\leq \deg g_1=s\delta <
\degw df_{\sigma (2)}\wedge df_{\sigma (3)}<
\degw df_{\sigma (1)}\wedge df_{\sigma (3)}=
\degw df_1\wedge df_{\sigma (3)}
$$
by (SU$2'$) and the last two conditions of (P12). 
Since $\degw df_1\wedge df_2<\degw f_1$ by assumption, 
we get $\sigma (3)\neq 2$. 
Hence, $\sigma (3)=3$, 
and so $\sigma =\id $. 
Because $(F,G)=(F_{\sigma },G)$ 
satisfies the Shestakov-Umirbaev condition by assumption, 
we may write $g_1$ and $g_2$ as in (SU1). 
It follows from the inequality above that 
the $\w $-degrees of $df_1\wedge df_3$ and $df_2\wedge df_3$ 
are greater  than $\deg f_1$, 
and hence greater than $\degw df_1\wedge df_2$. 
This implies that $a=b=c=0$ by the first equality of (P12). 
Therefore, we obtain $(f_1,f_2)=(g_1,g_2)$.

(iii) 
Since $\deg f_i<\deg f_1$ for $i=2,3$ by assumption, 
we have $\sigma (1)=1$ by (i). 
Hence, $\sigma =\id $ or $\sigma =(2,3)$. 
First, assume that $\sigma =\id $. 
Then, $\degw f_2=\degw g_2=2\delta $ by (P3). 
Since $2\degw f_1<3\degw f_2$ by assumption, 
we have 
$\degw f_1<(3/2)\degw f_2=3\delta \leq s\delta =\degw g_1$. 
Hence, $\deg f_3=(3/2)\delta $ by (P5). 
Therefore, 
we obtain 
$3\degw f_2=6\delta =4\degw f_3$. 
Next, 
assume that $\sigma =(2,3)$. 
Then, 
$$
\frac{3}{2}\delta <2\delta 
=\deg f_{\sigma (2)}=\deg f_3<\deg f_2=\deg f_{\sigma (3)}. 
$$ 
Hence, 
$\degw f_1=\degw g_1$ in view of (P5). 
By (P1), 
we have $2\degw g_1=s\deg g_2$ 
for some odd number $s\geq 3$. 
By (P3), 
$\degw g_2=\deg f_{\sigma (2)}=\degw f_3$. 
Therefore, 
$2\deg f_1=s\deg f_3$.

(iv) 
Set $\gamma _i=\degw df_p\wedge df_q$ for each $i$, 
where $p,q\in \N \sm \{ i\} $ with $1\leq p<q\leq 3$. 
By the first equality of (P12), 
we know that four possibilities exist for 
$\gamma _{\sigma (3)}=\deg df_{\sigma (1)}\wedge df_{\sigma (2)}$. 
Since 
$\gamma _1<\gamma _2<\gamma _3$ by assumption, 
we haves $\gamma _{\sigma (3)}\neq \gamma _{\sigma (i)}$ for $i=1,2$. 
Hence, 
the second and the third cases do not arise. 
Accordingly, 
$\gamma _{\sigma (3)}$ must be either 
$\degw f_{\sigma (3)}+\gamma _{\sigma (1)}$ 
or $\degw dg_1\wedge dg_2$, 
where $a\neq 0$ or $a=b=c=0$, respectively. 
In the former case, 
$\gamma _{\sigma (2)}=(s-2)\delta +\gamma _{\sigma (1)}
<\deg f_{\sigma (3)} +\gamma _{\sigma (1)}=\gamma _{\sigma (3)}$ 
by the second equality of (P12) and (P2). 
Hence, 
$\gamma _{\sigma (1)}<\gamma _{\sigma (2)}<\gamma _{\sigma (3)}$. 
Thus, we get $\sigma =\id $. 
Since $a\neq 0$, 
we have $s=3$ by (P11). 
Therefore, 
$2\deg g_1=3\deg g_2=3\deg f_2$ 
by (P1) and (P3). 
In the latter case, 
$\gamma _{\sigma (3)}=\deg dg_1\wedge dg_2
<\gamma _{\sigma (1)}<\gamma _{\sigma (2)}$ 
by the last two conditions of (P12). 
Hence, 
we get $\sigma =(1,2,3)$. 
Since $a=0$, 
we have $\deg f_2=\degw f_{\sigma (1)}=\degw g_1$ 
in view of (P5).  
By (P3), 
$\deg f_3=\deg f_{\sigma (2)}=\deg g_2$. 
Therefore, 
$2\degw f_2=s\degw f_3$ 
for some odd number $s\geq 3$ by (P1). 
\end{proof}

From Lemma~\ref{lem:su}(i) and (ii), we get the following proposition.

\begin{proposition}\label{prop:structure1}
Assume that 
\begin{equation}\label{eq:structure1}
\degw f_i<\degw f_1\quad(i=2,3)
\quad \text{and}\quad 
\degw df_1\wedge df_2<\degw f_1. 
\end{equation}
If $(F_{\sigma },G)$ 
satisfies the Shestakov-Umirbaev condition 
for some $\sigma \in \sym _3$ and $G\in \TT $, 
then there exists $E\in \E _3$ such that $F\circ E=G$. 
\end{proposition}
\begin{proof}\rm
Since $\degw f_i<\degw f_1$ for $i=2,3$, 
we have $\sigma (1)=1$ by Lemma~\ref{lem:su}(i). 
Since $\degw df_1\wedge df_2<\degw f_1$, 
we get $\sigma =\id $ and 
$(f_1,f_2)=(g_1,g_2)$ by Lemma~\ref{lem:su}(ii). 
Then, 
(SU1) implies that $G=F\circ E$ for some $E\in \E _3$. 
\end{proof}

In the rest of this section, 
we assume that $f_3^{\w }$ does not belong to $k[f_2^{\w }]$, and 
\begin{equation}\label{eq:ana condition}
\degw f_1=s\delta ,\quad 
\degw f_2=2\delta, \quad 
(s-2)\delta <\degw f_3<s\delta 
\end{equation}
for some odd number $s\geq 3$ and $\delta \in \Gamma $. 
Under the assumption, 
$f_2^{\w }$ does not belong to $k[f_3^{\w }]$, 
because $f_2^{\w }\not\approx f_3^{\w }$ 
and $\degw f_2=2\delta \leq 2(s-2)\delta <\degw f_3^2$. 
Furthermore, 
$f_1^{\w }$ belongs to $k[f_3^{\w }]$ 
if and only if $f_1^{\w }\approx (f_3^{\w })^2$, 
in which case $s=3$. 
In fact, 
if $f_1^{\w }$ belongs to $k[f_3^{\w }]$, 
then $f_1^{\w }\approx (f_3^{\w })^l$ for some $l\in \N $. 
Since $\degw f_3<\degw f_1$ by assumption, $l\geq 2$. 
If $l\geq 3$ or $s\geq 5$, 
then $\degw f_1=s\delta \leq l(s-2)\delta <l\degw f_3$, 
a contradiction. 
Thus, 
$l=2$ and $s=3$. 
If $f_1^{\w }\approx (f_3^{\w })^2$, 
then $f_1^{\w }$ clearly belongs to $k[f_3^{\w }]$.

Under the assumption above, 
the following two propositions hold.

\begin{proposition}\label{prop:structure2}
Assume that 
\begin{equation}\label{eq:pfstr}
\degw df_1\wedge df_2\leq \degw f_3-(s-2)\delta +\ep , 
\end{equation}
where $\ep :=\degw df_1\wedge df_2\wedge df_3>0$. 
If $f_2^{\w }$ belongs to $k[S_2]^{\w }$, 
then $f_1^{\w }\approx (f_3^{\w })^2$. 
\end{proposition}
\begin{proof}\rm
By assumption, 
there exists $\phi _2\in k[S_2]$ such that 
$\phi _2^{\w }=f_2^{\w }$. 
As mentioned after (\ref{eq:ana condition}), 
$f_2^{\w }$ does not belong to $k[f_3^{\w }]$. 
Since $\deg f_2<\deg f_1$ by (\ref{eq:ana condition}), 
$f_2^{\w }$ does not belong to 
$k[f_1^{\w },f_3^{\w }]\sm k[f_3^{\w }]$. 
Thus, 
$f_2^{\w }$ does not belong to 
$k[f_1^{\w },f_3^{\w }]$, 
and hence neither does $\phi _2^{\w }$. 
Therefore, 
we have $\degw \phi _2<\deg^{S_2}\phi _2$. 
By Lemma~\ref{lem:degS1}(ii), 
there exist $p,q\in \N $ with $\gcd (p,q)=1$ 
such that $(f_1^{\w })^p\approx (f_3^{\w })^q$ 
and 
\begin{align}
2\delta =\deg f_2>\degw (f_2-\phi _2)
&\geq p\degw f_1+\ep -\degw df_1\wedge df_2-\degw f_3 \notag \\
&\geq p\degw f_1-(\degw f_3-(s-2)\delta )-\degw f_3\notag \\
&=\left( s\left( p+1-\frac{2p}{q}\right) -2\right) \delta 
\label{eq:pfstr2}. 
\end{align}
Here, we use (\ref{eq:pfstr}) for the last inequality, 
and $\degw f_3=(p/q)\degw f_1$ 
and $\degw f_1=s\delta $ 
for the last equality. 
Now, 
suppose to the contrary that 
$f_1^{\w }\not\approx (f_3^{\w })^2$. 
Then, the assumptions 
of Lemma~\ref{lem:degS2}(ii) hold for 
$f=f_3$ and $g=f_1$. 
In fact, 
$f_1^{\w }$ does not belong to $k[f_3^{\w }]$ 
if $f_1^{\w }\not\approx (f_3^{\w })^2$ 
as remarked after (\ref{eq:ana condition}). 
By (\ref{eq:ana condition}), 
it follows that $\deg f_3<\deg f_1$ 
and $\deg \phi _2=\deg f_2<\deg f_1$. 
Thus, 
we may conclude by Lemma~\ref{lem:degS2}(ii) that $p=2$, 
and $q\geq 3$ is an odd number. 
Consequently, 
the right-hand side of (\ref{eq:pfstr2}) is at least 
$(3(2+1-2\cdot 2/3)-2)\delta =3\delta $, 
a contradiction. 
Therefore, we must have 
$f_1^{\w }\approx (f_3^{\w })^2$ if $f_2^{\w }$ 
belongs to $k[S_2]^{\w }$. 
\end{proof}

The following proposition 
forms the core of the proof of Theorem~\ref{thm:main1}.

\begin{proposition}\label{prop:structure3}
Assume that 
\begin{equation}\label{eq:ppfstr2}
\degw df_1\wedge df_2<\degw f_3-(s-2)\delta +\min \{ \delta ,\ep \} .
\end{equation}
If there exists $\phi _1\in k[S_1]$ such that 
$\degw f_1'<\degw f_1$, then either 
$f_1^{\w }\approx (f_3^{\w })^2$, 
or $(f_2^{\w })^2\approx (f_3^{\w })^3$ 
and $F'$ does not admit a Shestakov-Umirbaev reduction, 
where $f_1'=f_1+\phi _1$ and $F'=(f_1',f_2,f_3)$. 
Assume further that 
$(f_1')^{\w }$ does not belong to $k[S_1]^{\w }$. 
Then, the following assertions hold$:$ 

{\rm (1)} 
$f_i^{\w }$ does not belong to $k[S_i']^{\w }$ for $i=2,3$, 
where $S_i'=\{ f_1',f_2,f_3\} \sm \{ f_i\} $. 
Hence, 
$F'$ does not admit an elementary reduction.

$(2)$ 
If $f_1^{\w }\approx (f_3^{\w })^2$ and 
$(F'_\sigma ,G)$ satisfies the quasi Shestakov-Umirbaev condition 
for some $\sigma \in \sym _3$ and $G\in \TT $, 
then $\sigma =\id $ and $(F,G)$ satisfies 
the quasi Shestakov-Umirbaev condition. 
\end{proposition}
\begin{proof}\rm
To begin with, 
we show that 
$\degw \phi _1<\degw ^{S_1}\phi _1$ 
if $f_1^{\w }\not\approx (f_3^{\w })^2$. 
Since $\phi _1$ is an element of $k[S_1]$, 
we check that $\phi _1^{\w }$ 
does not belong to $k[f_2^{\w },f_3^{\w }]$. 
By the assumption that 
$\deg (f_1+\phi _1)<\deg f_1$, 
we have $\phi _1^{\w }\approx f_1^{\w }$. 
Since $\degw f_1=(s/2)\deg f_2$ 
for an odd number $s$ by (\ref{eq:ana condition}), 
$f_1^{\w }$ does not belong to $k[f_2^{\w }]$. 
Since $f_1^{\w }\not\approx (f_3^{\w })^2$ by assumption, 
$f_1^{\w }$ does not belong to $k[f_3^{\w }]$ 
as mentioned after (\ref{eq:ana condition}). 
By (\ref{eq:ana condition}), 
it follows that 
$$
\degw f_1=2\delta +(s-2)\delta <\degw f_2+\degw f_3. 
$$
Hence, 
$f_1^{\w }$ does not belong to 
$k[f_2^{\w },f_3^{\w }]\sm 
(k[f_2^{\w }]\cup k[f_3^{\w }])$. 
Thus, $f_1^{\w }$ 
does not belong to $k[f_2^{\w },f_3^{\w }]$, 
and hence neither does $\phi _3^{\w }$. 
Therefore, 
$\degw \phi _1<\degw ^{S_1}\phi _1$ 
if $f_1^{\w }\not\approx (f_3^{\w })^2$.

First, 
we show that 
$(f_2^{\w })^2\approx (f_3^{\w })^3$ 
and $F'$ does not admit a Shestakov-Umirbaev reduction 
%under the assumption that 
in the case where $\degw \phi _1<\degw ^{S_1}\phi _1$. 
Then, 
we obtain the first part of the proposition as a consequence, 
since $\degw \phi _1<\degw ^{S_1}\phi _1$ if 
$f_1^{\w }\not\approx (f_3^{\w })^2$ as shown above. 
By Lemma~\ref{lem:degS1}(ii), 
there exist $p,q\in \N $ with $\gcd (p,q)=1$ 
such that $(f_3^{\w })^p\approx (f_2^{\w })^q$ 
and 
\begin{align}\label{eq:pfstr3}
s\delta =\deg f_1>\degw (f_1+\phi _1)
&\geq q\degw f_2+\ep -\degw df_1\wedge df_2-\degw f_3\notag \\
&>q\degw f_2-(\degw f_3-(s-2)\delta )-\degw f_3\notag \\
&=\left( q\left( 2-\frac{4}{p}\right) +s-2\right) \delta ,
\end{align}
where we use (\ref{eq:ppfstr2}) for the last inequality, 
and $\deg f_3=(q/p)\deg f_2$ and $\deg f_2=2\delta $ 
for the last equality. 
Recall that we are assuming that 
$f_3^{\w }$ does not belong to $k[f_2^{\w }]$, 
while 
$f_2^{\w }$ does not belong to $k[f_3^{\w }]$ 
as mentioned after (\ref{eq:ana condition}). 
Hence, $p\geq 2$ and $q\geq 2$. 
We show that $p=3$ and $q=2$ by contradiction. 
Supposing that $p=2$, 
we have $\degw f_3=(q/2)\degw f_2=q\delta $. 
Hence, $(s-2)\delta <q\delta <s\delta $ by (\ref{eq:ana condition}), 
and so $q=s-1$. 
Since $p=2$ and $s$ is an odd number, 
we get $\gcd (p,q)=2$, 
a contradiction. 
If $p\geq 4$, 
then the right-hand side of (\ref{eq:pfstr3}) 
would be at least $(q+s-2)\delta \geq s\delta $, 
since $q\geq 2$. 
This is a contradiction. 
Thus, we get $p=3$. 
If $q\geq 3$, 
then the right-hand side of (\ref{eq:pfstr3}) 
would be at least $s\delta $, 
a contradiction. 
Hence, we have $q=2$. 
Therefore, 
we obtain 
$(f_3^{\w })^3\approx (f_2^{\w })^2$. 
From this, 
we know that 
$\degw f_3=(2/3)\degw f_2=(4/3)\delta $. 
Since $\degw f_3>(s-2)\delta $ by (\ref{eq:ana condition}), 
it follows that $s=3$. 
Consequently, 
the right-hand side of (\ref{eq:pfstr3}) 
is equal to $(7/3)\delta $. 
Thus, we get
\begin{equation}\label{eq:degrelation}
\deg f_3=\frac{4}{3}\delta 
<\deg f_2=2\delta <\frac{7}{3}\delta 
<\deg f_1'<3\delta . 
\end{equation}
It follows that $2\degw f_1'<6\delta =3\degw f_2$. 
Then, by Lemma~\ref{lem:su}(iii), 
we can conclude that $F'$ does not admit 
a quasi Shestakov-Umirbaev reduction, 
since 
$$
3\degw f_2=6\delta \neq \frac{16}{3}\delta =4\degw f_3, \ 
3\deg f_3=4\delta <\frac{14}{3}\delta 
<2\degw f_1'<6\delta <\frac{20}{3}\delta =5\degw f_3. 
$$ 
Therefore, 
$F'$ does not admit a Shestakov-Umirbaev reduction. 

In this situation, 
assume further that 
$(f_1')^{\w }$ does not belong to $k[S_1]^{\w }$. 
We show that $f_i^{\w }$ does not belong to $k[S_i']$ for $i=2,3$ 
by contradiction. 
Suppose that there exists $\phi _i\in k[S_i']$ 
such that $\phi _i^{\w }=f_i^{\w }$ for some $i\in \{ 2,3\} $. 
Then, 
the conditions (i)--(iv) after Lemma~\ref{lem:degS2} 
are fulfilled for $f=f_j$, $g=f_1'$, $h=f_i$ 
and $\phi =\phi _i$, 
where $j\in \{ 2,3\} \sm \{ i\} $. 
Actually, 
$f_1'=f_1+\phi _1$, $f_2$ and $f_3$ are 
algebraically independent over $k$, 
since so are $f_1$, $f_2$ and $f_3$, 
and $\phi _1$ is an element of $k[S_1]$. 
Moreover, 
$\degw f_l<\degw f_1'$ for $l=2,3$ 
by (\ref{eq:degrelation}), 
and $f_i^{\w }$ does not belong to $k[f_j^{\w }]$ 
by assumption, 
since $(i,j)$ is $(2,3)$ or $(3,2)$. 
By assumption, 
$(f_1')^{\w }$ does not belong to $k[S_1]^{\w }$, 
and hence does not belong to $k[f_j^{\w }]$. 
By the choice of $\phi _i$, 
we have $\deg (f_i-\phi _i)<\deg f_i$. 
Thus, (i)--(iv) are satisfied. 
By Lemma~\ref{lem:degS2}(ii) and the remark following it, 
we may conclude that 
$((f_1')^{\w })^2\approx (f_j^{\w })^q$ 
for some odd number $q\geq 3$. 
Hence, 
$\deg f_1'=(q/2)\deg f_j$ 
is equal to $(2q/3)\delta $ if $j=3$, 
and $q\delta $ if $j=2$. 
Since no odd number $q\geq 3$ 
satisfies $7/3<2q/3<3$ or $7/3<q<3$, 
we get a contradiction by (\ref{eq:degrelation}). 
Therefore, 
$f_i^{\w }$ does not belong to $k[S_i']^{\w }$ for $i=2,3$. 
Since $(f_1')^{\w }$ does not belong to $k[S_1]^{\w }$, 
it follows that $F'$ does not admit an elementary reduction. 
This proves (1) 
in the case where $\degw \phi _1<\degw ^{S_1}\phi _1$. 
The assumption of (2) does not hold 
in this situation, 
since $\deg f_1=3\delta \neq (8/3)\delta =\deg f_3^2$ 
by (\ref{eq:degrelation}).

Next, 
we show (1) and (2) 
in the case where 
$\degw \phi _1=\degw ^{S_1}\phi _1$ 
and $(f_1')^{\w }$ does not belong to $k[S_1]^{\w }$. 
By the remark in the first paragraph, 
we know that $f_1^{\w }\approx (f_3^{\w })^2$ 
if $\degw \phi _1=\degw ^{S_1}\phi _1$. 
As mentioned after (\ref{eq:ana condition}), 
$f_1^{\w }\approx (f_3^{\w })^2$ implies $s=3$. 
Hence, 
$\degw f_1=s\delta =3\delta $, 
and so 
$\degw f_3=(1/2)\degw f_1=(3/2)\delta $. 
Since $\deg ^{S_1}\phi _1=\deg \phi _1$ and 
$\deg \phi _1=\deg f_1$, 
we have $\deg ^{S_1}\phi _1=3\delta $. 
By Lemma~\ref{lem:expression}(i), 
we may write $\phi _1=af_3^2+cf_3+\psi $, 
where $a,c\in k$ and $\psi \in k[f_2]$ 
with $\degw \psi \leq (3-1)\delta =2\delta $. 
Since $\deg f_2=2\delta $, 
we get $\psi =ef_2+e'$ for some $e,e'\in k$. 
Note that $a\neq 0$, 
for otherwise 
$\deg \phi _1\leq 
\max \{ \deg f_3,\deg \psi \} <\deg f_1$, 
a contradiction. 
We claim that the conditions (a)--(d) 
before Lemma~\ref{lem:determinant} hold for 
$k_i=f_i$ for $i=1,2,3$ 
and $k_i'=k_i$ for $i=1,2$. 
In fact, 
(a)--(c) follow from 
$\degw k_1=\degw k_1'=3\delta $, 
$\degw k_2=\degw k_2'=2\delta $ 
and $\degw k_3=(3/2)\delta $. 
The left-hand side of (d) is less than $3\delta $, 
since $\degw df_1\wedge df_2<\deg f_3=(3/2)\delta $ 
by (\ref{eq:ppfstr2}) with $s=3$. 
Because the right-hand side of (d) 
is greater than $\deg k_1=3\delta $, 
we know that (d) holds true. 
Therefore, by (\ref{eq:detfirst}), 
we obtain 
\begin{equation}\label{eq:pfstrdet1}
\degw df_1\wedge df_3
=\deg f_1-\deg f_2+\degw df_2\wedge df_3
=\delta +\degw df_2\wedge df_3. 
\end{equation}
Hence, 
$\deg df_2\wedge df_3<\deg df_1\wedge df_3$. 
Since 
$d\phi _1\wedge df_3=d\psi \wedge df_3=edf_2\wedge df_3$, 
we have  
$df_1'\wedge df_3=df_1\wedge df_3+edf_2\wedge df_3$. 
Thus, $\degw df_1'\wedge df_3=\degw df_1\wedge df_3$, 
and so 
\begin{equation}\label{eq:pfstrdet2}
\degw df_1'\wedge df_3
=\delta +\degw df_2\wedge df_3 
\end{equation}
by (\ref{eq:pfstrdet1}). 
For the same reason as above, 
the conditions (a)--(d) before 
Lemma~\ref{lem:determinant} hold for 
$k_1=f_1'$, $k_i=f_i$ for $i=2,3$, 
$k_1'=f_1=k_1-ak_3^2-ck_3-\psi $ 
and $k_2'=k_2$, 
for $k_1$ is not involved in the conditions. 
Since $a\neq 0$ and 
$\degw df_1\wedge df_2<\degw f_3$, 
we know by Lemma~\ref{lem:determinant}(i) that 
\begin{equation}\label{eq:pfstrdet3}
\degw df_1'\wedge df_2
=\deg f_3+\degw df_2\wedge df_3
=\frac{3}{2}\delta +\degw df_2\wedge df_3. 
\end{equation}
Set $\Phi =f_1+ay^2+cy+ef_2+e'$. 
Then, 
$\deg _{\w }^{f_3}\Phi =\degw f_1$, 
while $\degw \Phi (f_3)=\degw f_1'<\degw f_1$. 
Since $\Phi ^{(1)}=2ay+c$ and $a\neq 0$, 
we have 
$\deg _{\w }^{f_3}\Phi ^{(1)}
=\degw f_3=\degw \Phi ^{(1)}(f_3)$. 
Hence, $m_{\w }^{f_3}(\Phi )=1$. 
By Theorem~\ref{thm:inequality}, 
it follows that 
\begin{equation}\label{eq:delta bound}
\begin{aligned}
\degw f_1'=\degw \Phi (f_3)
&\geq \deg _{\w }^{f_3}\Phi 
+m_{\w }^{f_3}(\Phi )(\ep 
-\degw df_1\wedge df_2-\degw f_3)\\
&=\deg f_1
+\ep -\degw df_1\wedge df_2-\degw f_3\\
&>\degw f_1 -2\degw f_3+(s-2)\delta =\delta , 
\end{aligned}
\end{equation}
where the last inequality is due to (\ref{eq:ppfstr2}). 
With the aid of (\ref{eq:delta bound}), 
we show the following: 
\begin{equation*}
{\rm (i)}\ (f_1')^{\w }\not\in k[f_2^{\w },f_3^{\w }]. \qquad 
{\rm (ii)}\ f_2^{\w }\not\in k[(f_1')^{\w },f_3^{\w }]. \qquad 
{\rm (iii)}\ f_3^{\w }\not\in k[(f_1')^{\w },f_2^{\w }].
\end{equation*}
Since $k[f_2^{\w },f_3^{\w }]$ is contained in $k[S_1]^{\w }$, 
(i) follows from the assumption that 
$(f_1')^{\w }$ does not belong to $k[S_1]^{\w }$. 
In particular, 
$f_2^{\w }\not\approx (f_1')^{\w }$. 
By (\ref{eq:delta bound}), 
$\degw f_2=2\delta <\degw (f_1')^2$. 
Hence, 
$f_2^{\w }$ does not belong to $k[(f_1')^{\w }]$. 
Since 
$\degw f_3=(3/2)\delta <\degw f_2<3\delta =\degw f_3^2$, 
it follows that 
$f_2^{\w }$ does not belong to $k[f_3^{\w }]$. 
By (\ref{eq:delta bound}), 
$\degw f_2<\delta +(3/2)\delta <\degw f_1'f_3$, 
and so $f_2^{\w }$ does not belong to 
$k[(f_1')^{\w },f_3^{\w }]\sm 
(k[(f_1')^{\w }]\cup k[f_3^{\w }])$. 
Thus, 
$f_2^{\w }$ does not belong to 
$k[(f_1')^{\w },f_3^{\w }]$, 
proving (ii). 
It follows that $f_3^{\w }\not\approx (f_1')^{\w }$ by (i), 
and $\degw f_3<2\delta <\degw (f_1')^2$ 
by (\ref{eq:delta bound}). 
Hence, 
$f_3^{\w }$ does not belong to $k[(f_1')^{\w }]$. 
Since $\degw f_3<\degw f_2$, 
we get that $f_3^{\w }$ does not belong to 
$k[(f_1')^{\w },f_2^{\w }]\sm k[(f_1')^{\w }]$. 
This proves (iii).

Now, we show that 
$f_2^{\w }$ does not belong to $k[S_2']^{\w }$ 
by contradiction. 
Supposing the contrary, 
there exists $\phi _2\in k[S_2']$ such that 
$\phi _2^{\w }=f_2^{\w }$. 
Then, 
$\phi _2^{\w }$ does not belong to 
$k[(f_1')^{\w },f_3^{\w }]$ by (ii). 
Hence, 
$\degw \phi _2<\degw ^{S_2'}\phi _2$. 
By Lemma~\ref{lem:degS1}(i), 
there exist $p,q\in \N $ with $\gcd (p,q)=1$ 
for which $((f_1')^{\w })^q\approx (f_3^{\w })^p$ and 
\begin{equation}\label{eq:pfstrdet4}
\begin{aligned}
2\delta =\deg f_2=\degw \phi _2
&\geq pq\gamma +\degw df_1'\wedge df_3-p\gamma -q\gamma \\
&=pq\gamma +\delta +\degw df_2\wedge df_3-p\gamma -q\gamma , 
\end{aligned}
\end{equation}
where $\gamma \in \Gamma $ such that 
$\degw f_1'=p\gamma $ and $\degw f_3=q\gamma $, 
and the last equality is due to (\ref{eq:pfstrdet2}). 
From (\ref{eq:pfstrdet4}), 
it follows that 
$(pq-p-q)\gamma <\delta$. 
Since $\deg f_1'>\delta $ by (\ref{eq:delta bound}), 
and since $\deg f_3=(3/2)\delta >\delta $, 
we have 
$\delta <\min \{ \deg f_1',\deg f_3\} 
=\min \{ p,q\} \gamma $. 
Hence, $pq-p-q<\min \{ p,q\} $. 
By (iii) and (i), 
$f_3^{\w }$ and $(f_1')^{\w }$ do not belong to 
$k[(f_1')^{\w }]$ and $k[f_3^{\w }]$, 
respectively. 
As $\gcd (p,q)=1$, 
we get $2\leq p<q$ or $2\leq q<p$. 
It follows from the claim before Lemma~\ref{lem:degS2} 
that $(p,q)=(2,3)$ or $(p,q)=(3,2)$. 
If $(p,q)=(2,3)$, 
then 
$3\delta <3\degw f_1'=2\degw f_3=3\delta $ 
by (\ref{eq:delta bound}), 
a contradiction. 
Thus, $(p,q)=(3,2)$. 
Then, $\degw f_1'=(3/2)\degw f_3=(9/4)\delta $ 
and $\gamma =(1/2)\degw f_3=(3/4)\delta $, and so 
\begin{equation}\label{eq:513}
\degw df_2\wedge df_3\leq 
2\delta -pq\gamma -\delta +p\gamma +q\gamma 
=2\delta -6\gamma -\delta 
+3\gamma +2\gamma =\frac{1}{4}\delta 
\end{equation}
by (\ref{eq:pfstrdet4}). 
By Lemma~\ref{lem:degS1}(ii) and (\ref{eq:513}), 
we get 
\begin{align*}
\degw (f_2-\phi _2)\geq 
3\degw f_3+\ep -\degw df_2\wedge df_3-\degw f_1'
>\frac{9}{2}\delta -\frac{1}{4}\delta -\frac{9}{4}\delta 
=2\delta .
\end{align*}
However, 
since $\phi _2^{\w }=f_2^{\w }$, 
we have $\degw (f_2-\phi _2)<\degw f_2=2\delta $, 
a contradiction. 
Therefore, 
$f_2^{\w }$ does not belong to $k[S_2']^{\w }$.

Similarly, suppose to the contrary 
that there exists $\phi _3\in k[S_3']$ 
such that $\phi _3^{\w }\approx f_3^{\w }$. 
Then, $\phi _3^{\w }$ does not belong to 
$k[(f_1')^{\w },f_2^{\w }]$ by (iii). 
Hence, 
$\degw \phi _3<\degw ^{S_3'}\phi _3$. 
By (i) and (ii), 
$(f_1')^{\w }$ and $f_2^{\w }$ 
do not belong to $k[f_2^{\w }]$ and $k[(f_1')^{\w }]$, 
respectively. 
Thus, 
$$
\degw df_1'\wedge df_2<\degw \phi _3=\degw f_3=\frac{3}{2}\delta 
$$ 
by Lemma~\ref{lem:degS2}(i). 
This contradicts (\ref{eq:pfstrdet3}). 
Therefore, 
$f_3^{\w }$ does not belong to $k[S_3']^{\w }$. 
This completes the proof of (1).

Finally, we show (2). 
Assume that $(F'_{\sigma },G)$ 
satisfies the quasi Shestakov-Umirbaev condition 
for some $\sigma \in \sym _3$ and $G\in \TT $. 
By (\ref{eq:pfstrdet2}) and (\ref{eq:pfstrdet3}), 
we have 
$$
\degw df_2\wedge df_3<\degw df_1'\wedge df_3
<\degw df_1'\wedge df_2. 
$$
In addition, 
$2\deg f_2=4\delta \neq (3/2)r\delta =r\deg f_3$ 
for any odd number $r\geq 3$. 
Hence, we get 
$\sigma =\id $ and $2\degw g_1=3\degw f_2$ 
by Lemma~\ref{lem:su}(iv). 
Thus, 
$(F',G)$ satisfies the quasi Shestakov-Umirbaev condition, 
and $\degw g_1=(3/2)\degw f_2=\degw f_1$. 
Then, it is immediate that $(F,G)$ satisfies 
(SU$2'$), (SU$3'$), (SU4), (SU5) and (SU6). 
As for (SU$1'$), 
we have only to check that 
$g_1-f_1$ belongs to $k[f_2,f_3]$. 
Since $(F',G)$ satisfies the quasi Shestakov-Umirbaev condition, 
$g_1-f_1'$ belongs to $k[f_2,f_3]$ by (SU$1'$). 
Hence, 
$g_1-f_1=(g_1-f_1')+\phi _1$ belongs to $k[f_2,f_3]$, 
since so does $\phi _1$. 
Thus, 
$(F,G)$ satisfies (SU$1'$). 
Therefore, 
$(F,G)$ satisfies the quasi Shestakov-Umirbaev condition. 
This completes the proof of (2). 
\end{proof}

We note that (\ref{eq:delta bound}) is the key estimation 
which guarantees that no tame automorphism 
admits a reduction of type IV\null.

\section{Proof of Theorem~\ref{thm:main1}}
\label{sect:proof}
\setcounter{equation}{0}

We begin with the following lemma. 

\begin{lemma}\label{lem:|omega|}
{\rm (i)} If $\degw F=|\w |$ for $F\in \Aut _k\kx $, 
then $F$ is tame. 

{\rm (ii)} 
$\Sigma :=\{ a_1w_1+\cdots +a_nw_n\mid a_1,\ldots ,a_n\in \Zn \} $ 
is a well-ordered subset of $\Gamma $. 
\end{lemma}
\begin{proof}\rm
(i) 
We may assume that 
$w_1\leq \cdots \leq w_n$ and 
$\deg f_1\leq \cdots \leq \deg f_n$ 
by changing the indices of 
$w_1,\ldots ,w_n$ and $f_1,\ldots ,f_n$ if necessary. 
Write $f_i=b_i+\sum _{j=1}^na_{i,j}x_j+f_i'$ for each $i$, 
where $b_i,a_{i,j}\in k$ for each $j$, 
and $f_i'$ is an element of the ideal $Q$ 
of $\kx $ generated by all the quadratic monomials. 
Clearly, 
$F$ is tame if and only if so is 
$F\circ G'$ or $G'\circ F$ for some $G'\in \T _k\kx $. 
Since $\deg F\circ G=\deg F$ 
for $G=(x_1-b_1,\ldots ,x_n-b_n)$, 
we may assume that $b_i=0$ for each $i$ 
by replacing $F$ by $F\circ G$. 
Note that $\det (a_{i,j})_{i,j}$ 
is equal to the Jacobian of $F$, 
so $(a_{i,j})_{i,j}$ is invertible. 
Let $H$ be an affine automorphism of $\kx $ defined by 
$H(x_i)=\sum _{j=1}^na_{i,j}x_j$ for each $i$. 
Then, 
$\deg H(x_i)\leq 
\deg f_i$ for each $i$, 
since $f_i=H(x_i)+f_i'$. 
We claim that $\deg f_i=w_i$ for each $i$. 
In fact, if not, we can find $i$ such that $\deg f_i<w_i$, 
since $\deg F=|\w |$ by assumption. 
Then, 
$\deg H(x_j)\leq \deg f_j\leq \deg f_i<w_i$ for $j\leq i$, 
while $\deg x_l=w_l\geq w_i$ 
for $l\geq i$ by assumption. 
Hence, 
$H(x_j)$ is contained in the 
$(i-1)$-dimensional $k$-vector space 
$\bigoplus _{l=1}^{i-1}kx_l$ 
for $j=1,\ldots ,i$. 
This contradicts that 
$H(x_1),\ldots ,H(x_i)$ are linearly independent over $k$. 
Thus, we get $\deg f_i=w_i$, 
and hence $\deg H(x_i)\leq w_i$ for each $i$. 
We show that $\deg H^{-1}(x_i)\leq w_i$ for each $i$. 
Let $m$ be the maximal number for which $w_m=w_i$. 
Then, 
$H(x_j)$ belongs to 
$\bigoplus _{l=1}^mkx_l$ for $j=1,\ldots ,m$. 
Hence, 
$H$ induces an automorphism of 
$\bigoplus _{l=1}^mkx_l$. 
Thus, $H^{-1}(x_i)$ belongs to 
$\bigoplus _{l=1}^mkx_l$. 
Therefore, 
$\deg H^{-1}(x_i)\leq w_m=w_i=\deg x_i$. 
This implies that 
$\deg H^{-1}(g)\leq \deg g$ 
holds for each $g\in \kx $. 
Consequently, 
$$
|\w |\leq \deg H^{-1}\circ F=\sum _{i=1}^n\deg H^{-1}(f_i)
\leq \sum _{i=1}^n\deg f_i=\deg F=|\w |. 
$$
Therefore, 
$\deg H^{-1}\circ F=|\w |$, 
and so we may replace $F$ by $H^{-1}\circ F$. 
It follows that $f_i=x_i+f_i''$ for each $i$, 
where $f_i''=H^{-1}(f_i')\in Q$. 
We show that $f_i''$ belongs to 
$k[x_1,\ldots ,x_{i-1}]$ for every $i$ 
by contradiction. 
Suppose that there appears in $f_i''$ 
a monomial $x_{a_1}\cdots x_{a_n}$, 
where $a_1,\ldots ,a_n\in \{ 1,\ldots ,n\} $ 
with $a_1\geq i$. 
Since $x_{a_1}\cdots x_{a_n}$ belongs to $Q$, 
we have $n\geq 2$. 
Hence, 
$$
w_i=\deg f_i\geq \deg f_i''\geq \deg x_{a_1}\cdots x_{a_n}
=\sum _{i=1}^nw_{a_i}>w_{a_1}\geq w_i, 
$$
a contradiction. 
Thus, 
$f_i''$ belongs to $k[x_1,\ldots ,x_{i-1}]$ for each $i$. 
This means that $F$ is triangular. 
Here, we say that 
$(h_1,\ldots ,h_n)\in \Aut _k\kx $ is {\it triangular} 
if there exists $\sigma \in \sym _n$ such that 
$h_{\sigma (i)}=x_{\sigma (i)}+\phi _i$ for some 
$\phi _i\in k[x_{\sigma (1)},\ldots ,x_{\sigma (i-1)}]$ 
for $i=1,\ldots , n$. 
Since a triangular automorphism is tame, 
we conclude that $F$ is tame. 

(ii) 
We show that each nonempty subset $S$ of $\Sigma $ 
has the minimum element. 
As mentioned, 
we may regard $\Gamma =\Z ^r$ for some $r\in \N $. 
Let $\kyy $ be the Laurent polynomial ring in 
$y_1,\ldots,y_r$ over $k$, 
and $R$ the $k$-subalgebra of $\kyy $ generated by 
$\y ^{w_i}$ for $i=1,\ldots ,n$, 
where $\y ^{\alpha }=y_1^{\alpha _1}\cdots y_r^{\alpha _r}$ 
for each $\alpha =(\alpha _1,\ldots ,\alpha _r)$. 
Then, $R$ is Noetherian, 
and contains $\y ^{\alpha }$ 
for each $\alpha \in \Sigma $. 
Consider the ideal $I$ of $R$ 
generated by $\{ \y ^{\alpha }\mid \alpha \in S\} $. 
Since $R$ is Noetherian, 
there exists a finite subset $S'$ of $S$ 
with minimum element $\mu $ 
such that $I$ is generated by 
$\{ \y ^{\alpha }\mid \alpha \in S'\} $. 
Then, 
$\mu $ becomes the minimum element of $S$. 
In fact, 
for each $\alpha \in S$, 
there exist $\beta \in S'$ and $\gamma \in \Sigma $ 
such that $\y ^{\alpha }=\y ^{\beta }\y ^{\gamma }$. 
Then, $\beta \geq \mu $, $\gamma \geq 0$ 
and $\alpha =\beta +\gamma $. 
Hence, 
$\alpha \geq \beta \geq \mu $. 
Thus, $\mu $ is the minimum element of $S$. 
Therefore, 
$\Sigma $ is a well-ordered subset of $\Gamma $. 
\end{proof}

In the rest of the paper, 
we assume that $n=3$, 
and identify $\ky $ with $\kx $. 
Let ${\cal A}$ be the set of $F\in \Aut _k\kx $ 
for which there exists $G_i\in \Aut _k\kx $ for $i=1,\ldots ,l$ 
with $G_1=F$ and $\degw G_l=|\w |$ such that 
$G_{i+1}$ is an elementary reduction 
or a quasi Shestakov-Umirbaev reduction 
of $G_i$ for $i=1,\ldots ,l-1$, 
where $l\in \N $. 
Then, each element of $\A $ is tame, 
since $G_l$ is tame if $\deg G_l=|\w |$ by Lemma~\ref{lem:|omega|}(i), 
and $G_{i}$ is tame if and only if so is $G_{i+1}$ for each $i$. 
Hence, 
$\A $ is contained in $\T _k\kx $. 
By definition, 
if $\deg F>|\w |$ for $F\in \A $, 
then $F$ admits 
an elementary reduction 
or a quasi Shestakov-Umirbaev reduction. 
By Proposition~\ref{prop:equivalence}(ii), 
$F$ admits a quasi Shestakov-Umirbaev reduction 
if and only if 
$F$ admits a Shestakov-Umirbaev reduction. 
Thus, 
if $\deg F>|\w |$ for $F\in \A $, 
then $F$ admits 
an elementary reduction 
or a Shestakov-Umirbaev reduction. 
The goal of this section is to establish that $\A =\T _k\kx $, 
which implies Theorem~\ref{thm:main1} immediately.

We remark that, 
if $F$ belongs to $\A $, 
then so do $F_{\sigma }$ and $F\circ H$, 
where $\sigma \in \sym _3$ and 
$H=(c_1x_1,c_2x_2,c_3x_3)$ 
with $c_1,c_2,c_3\in k\sm \zs $. 
If $\deg F=|\w |$ 
or if there exists $G\in \A $ such that 
$G$ is an elementary reduction 
or a quasi Shestakov-Umirbaev reduction of $F$, 
then $F$ belongs to $\A $.

The following is a key proposition.

\begin{proposition}\label{prop:key}
If $\degw F\circ E\leq \degw F$ 
for $F\in \A $ and $E\in \E $, 
then $F\circ E$ belongs to $\A $. 
\end{proposition}

Note that, 
if $\degw F\circ E>\degw F$ for $F\in \A $ 
and $E\in \E $, 
then $F\circ E$ belongs to $\A $. 
Actually, 
$(F\circ E)\circ E^{-1}=F$ 
is an elementary reduction of $F\circ E$.

We deduce from Proposition~\ref{prop:key} 
that $\T _k\kx $ is contained in ${\cal A}$. 
Take any $F\in \T _k\kx $. 
Then, 
we can express $F=H\circ E_1\circ \cdots \circ E_l$, 
where $H=(c_1x_1,c_2x_2,c_3x_3)$ 
with $c_1,c_2,c_3\in k\sm \zs $, 
$l\in \Zn $, and $E_i\in \E $ for $i=1,\ldots ,l$. 
We show that $F$ belongs to ${\cal A}$ by induction on $l$. 
The assertion is true if $l=0$, i.e., $F=H$, 
since $\degw H=|\w |$. 
Assume that $l>0$. 
By induction assumption, 
$F':=H\circ E_1\circ \cdots \circ E_{l-1}$ 
belongs to ${\cal A}$. 
Then, 
$F=F'\circ E_l$ belongs to ${\cal A}$ by 
Proposition~\ref{prop:key} and the note following it. 
Therefore, $\T _k\kx $ is contained in ${\cal A}$ 
on the assumption that Proposition~\ref{prop:key} is true.

The following proposition is necessary to prove 
Proposition~\ref{prop:key}.

\begin{proposition}\label{prop:key2}
Assume that $F=(f_1,f_2,f_3)\in {\cal A}$ 
satisfies 
\begin{equation}\label{eq:key2}
\degw f_1=s\delta ,\ \degw f_2=2\delta ,\ 
(s-2)\delta +\degw df_1\wedge df_2\leq \degw f_3<s\delta 
\end{equation}
for some odd number $s\geq 3$ and $\delta \in \Gamma $, 
and that $f_3^{\w }$ does not belong to $k[f_2^{\w }]$. 
Then, 
there exists $E\in \E _3$ such that $\degw F\circ E<\degw F$ 
and $F\circ E$ belongs to ${\cal A}$. 
\end{proposition}

We note that (\ref{eq:key2}) implies 
(\ref{eq:structure1}), (\ref{eq:ana condition}), 
(\ref{eq:pfstr}) and (\ref{eq:ppfstr2}). 
Furthermore, 
$f_1^{\w }$ and $f_2^{\w }$ 
are algebraically dependent over $k$ 
in this situation, 
for otherwise 
$$
\degw df_1\wedge df_2=\degw f_1+\degw f_2=(s+2)\delta 
$$ 
as mentioned after (\ref{eq:ineq-wedge}), 
which contradicts the last inequality of (\ref{eq:key2}).

We establish Propositions~\ref{prop:key} and~\ref{prop:key2} 
simultaneously by induction on $\degw F$. 	
Since $\Sigma $ is well-ordered by Lemma~\ref{lem:|omega|}(ii), 
so is the subset $\Delta :=\{ \degw H\mid H\in {\cal A}\} $, 
where $\min \Delta =|\w |$. 
Assume that $F\in \A $ satisfies $\degw F=|\w |$. 
If $\deg F\circ E\leq \deg F$ for $E\in \E $, 
then $\deg F\circ E=|\w |$, 
since $\deg F\circ E\geq |\w |$ by (\ref{eq:auto deg lower bound}). 
Hence, $F\circ E$ belongs to $\A $. 
Thus, 
the statement of Proposition~\ref{prop:key} 
holds for $F\in {\cal A}$ with $\deg F=|\w |$. 
Note that $f_1^{\w }$, $f_2^{\w }$ and $f_3^{\w }$ 
are algebraically independent over $k$ if $\deg F=|\w |$, 
for otherwise 
$\degw df_1\wedge df_2\wedge df_3<\sum _{i=1}^3\degw f_i=|\w |$, 
a contradiction. 
Therefore, 
the assumption of Proposition~\ref{prop:key2} is not fulfilled.

Let $\mu $ be an element of $\Delta $ such that $\mu >|\w |$, 
and assume that the statement of Proposition~\ref{prop:key} 
holds for each $F\in \A $ with $\degw F<\mu $. 
For $F\in \Aut _k\kx $, 
we define $I_F$ to be the set of 
$i\in \{ 1,2,3\} $ for which there exists $E\in \E _i$ 
such that $\degw F\circ E<\degw F$ 
and $F\circ E$ belongs to ${\cal A}$. 
Note that, if $\deg F>|\w |$ for $F\in \A $, 
then either $I_F\neq \emptyset $, 
or $(F_{\sigma },G)$ satisfies the quasi Shestakov-Umirbaev condition 
for some $\sigma \in \sym _3$ and $G\in \A $.

\begin{claim}\label{claim:1}
Let $F$ be an element of $\A $ such that $\degw F=\mu $. 

{\rm (i)} If $E$ is an element of $\E _i$ for some $i\in I_F$, 
then $F\circ E$ belongs to ${\cal A}$.

{\rm (ii)} 
If there exist $E',E''\in \E $ and 
$E_i\in \E _i$ with  $\degw F\circ E_i<\degw F$ 
for some $i\in I_F$ such that 
$E\circ E'=E_i\circ E''$ for $E\in \E $, 
then $F\circ E$ belongs to ${\cal A}$.

{\rm (iii)} 
For a triangular automorphism $H$ of $\kx $, 
we define $E_i\in \E _i$ by $E_i(x_i)=H(x_i)$ for each $i$. 
If $\deg (F\circ H)(x_i)<\deg f_i$, 
or equivalently $\degw F\circ E_i<\degw F$, 
for some $i\in I_F$, 
then $F\circ E_j$ belongs to $\A $ for $j=1,2,3$.

{\rm (iv)} 
If $I_F\sm \{ i\} \neq \emptyset $ and 
$f_j^{\w }$ belongs to $k[f_i^{\w }]$ 
for some $i,j\in \{ 1,2,3\} $ with $i\neq j$, 
then $j$ belongs to $I_F$. 

{\rm (v)} If $(F,G)$ satisfies the quasi Shestakov-Umirbaev condition 
for some $G\in \A $, 
then there exists $G'\in \A $ such that 
$(F,G')$ satisfies the Shestakov-Umirbaev condition. 
\end{claim}
\begin{proof}\rm
(i) Since $i$ is an element of $I_F$, 
there exists $E_i\in \E _i$ such that 
$\degw F\circ E_i<\degw F$ and $F\circ E_i$ 
belongs to ${\cal A}$. 
Then, 
we have $\deg F\circ E_i<\mu $, 
since $\deg F=\mu $ by assumption. 
For each $E\in \E _i$, 
it follows that $E':=E_i^{-1}\circ E$ 
is an element of $\E _i$. 
Hence, 
$F\circ E=(F\circ E_i)\circ E'$ belongs to ${\cal A}$ 
by the induction assumption of Proposition~\ref{prop:key}.

(ii) We may assume that 
$E$ is contained in $\E _j$ for some $j\neq i$ by (i), 
and $\degw F\circ E\leq \degw F$ 
by the note after Proposition~\ref{prop:key}. 
Then, 
$E'$ and $E''$ belong to $\E _i$ and $\E _j$, 
respectively, 
since $E\circ E'=E_i\circ E''$ by assumption. 
Hence, 
$(E_i\circ E'')(x_j)=(E\circ E')(x_j)=E(x_j)$, 
and $(E_i\circ E'')(x_l)=E_i(x_l)$ for $l\neq j$. 
Since $\degw F\circ E_i<\degw F$ and 
$\degw F\circ E\leq \degw F$, 
we have 
$$
\degw (F\circ E_i\circ E'')(x_l)=
\left\{ 
\begin{array}{ll}
\degw (F\circ E_i)(x_i)<\degw f_i& \text{ if }l=i\\
\degw (F\circ E)(x_j)\leq \degw f_j& \text{ if }l=j\\
\degw (F\circ E_i)(x_l)=\degw f_l& \text{ otherwise.}
\end{array}
\right.
$$
Thus, 
$\degw F\circ E_i\circ E''<\degw F$. 
Note that $F\circ E_i\circ E''$ belongs to $\A $ 
by the induction assumption of Proposition~\ref{prop:key}, 
since $\degw F\circ E_i<\degw F=\mu $, 
and $F\circ E_i$ belongs to $\A $ by (i). 
Therefore, 
$(F\circ E_i\circ E'')\circ (E')^{-1}$ belongs to $\A $ 
for the same reason. 
This shows that 
$F\circ E$ belongs to $\A $, since 
$F\circ E_i\circ E''\circ (E')^{-1}
=F\circ E\circ E'\circ (E')^{-1}=F\circ E$.

(iii) Without loss of generality, 
we may assume that $i\neq j$ by (i). 
We may also assume that 
$H(x_l)=x_l+\phi _l$ for each $l$, 
where $\phi _l\in k[x_1,\ldots ,x_{l-1}]$. 
Then, 
$E_p\circ E'=E_q\circ E_p$ holds for each $p<q$, 
where 
$E'\in \E _q$ such that 
$E'(x_q)=x_q+E_p^{-1}(\phi _q)$. 
In view of this, 
we can find $E',E''\in \E $ 
such that $E_j\circ E'=E_i\circ E''$. 
By assumption, 
$\deg F\circ E_i<\deg F$, 
and $i$ is an element of $I_F$. 
Hence, we conclude that 
$F\circ E_j$ belongs to $\A $ by (ii).

(iv) 
Since $I_F\sm \{ i\} \neq \emptyset $ by assumption, 
we can find $l\in I_F\sm \{ i\} $ 
and $E_l\in \E _l$ such that $\degw F\circ E_l<\degw F$. 
Clearly, we may assume that $j\neq l$. 
Since $f_j^{\w }$ belongs to $k[f_i^{\w }]$ by assumption, 
there exist $c\in k\sm \zs $ and $r\in \N $ such that 
$f_j^{\w }=c(f_i^{\w })^r$. 
Then, 
we can define a triangular automorphism $H$ of $\kx $ by 
$H(x_i)=x_i$, $H(x_j)=x_j-cx_i^r$ 
and $H(x_l)=E_l(x_l)$. 
Define $E_j\in {\cal E}_j$ by $E_j(x_j)=H(x_j)$. 
Since $\degw F\circ E_l<\degw F$ for $l\in I_F$, 
it follows from (iii) that $F\circ E_j$ belongs to $\A $. 
Moreover, 
since $\degw (f_j-cf_i^r)<\degw f_j$, 
we have $\degw F\circ E_j<\degw F$. 
Therefore, 
$j$ belongs to $I_F$.

(v) 
Since $(F,G)$ satisfies the Shestakov-Umirbaev condition 
by assumption, 
there exists $E_i\in \E _i$ for $i=1,2$ 
such that $\degw G\circ E_1=\degw G$, 
and $(F,G')$ satisfies the Shestakov-Umirbaev condition 
by Proposition~\ref{prop:equivalence}(i), 
where $G'=G\circ E_1\circ E_2$. 
We show that $G'$ belongs to $\A $. 
Since $G$ is an element of $\A $, 
and since $\deg G<\deg F=\mu $ by (P6), 
it follows that $G\circ E_1$ belongs to $\A $ 
by the induction assumption of Proposition~\ref{prop:key}. 
Then, 
$(G\circ E_1)\circ E_2$ belongs to ${\cal A}$ 
for the same reason, 
since $\degw G\circ E_1=\degw G<\mu$. 
Therefore, 
the assertion holds for $G'=G\circ E_1\circ E_2$. 
\end{proof}

Now, 
we show that the statement of Proposition~\ref{prop:key2} 
holds for each $F\in \A $ with $\degw F=\mu $. 
Since $\mu >|\w |$, 
we have $\degw F>|\w |$. 
Hence, 
$I_F\neq \emptyset $ 
or $(F_{\sigma },G)$ satisfies 
the quasi Shestakov-Umirbaev condition 
for some $\sigma \in \sym _3$ and $G\in \A $ as noted. 
The conclusion of Proposition~\ref{prop:key2} 
is obvious if $I_F$ contains $3$. 
If $I_F$ contains $2$, 
then $\deg F\circ E_2<\deg F$ for some $E_2\in {\cal E}_2$. 
Hence, 
$f_2^{\w }$ belongs to $k[S_2]^{\w }$. 
Then, we get 
$f_1^{\w }\approx (f_3^{\w })^2$ 
by Proposition~\ref{prop:structure2}. 
Here, 
we remind that the assumption 
of Proposition~\ref{prop:key2} implies 
(\ref{eq:structure1}), (\ref{eq:ana condition}), 
(\ref{eq:pfstr}) and (\ref{eq:ppfstr2}). 
Thus, 
$f_1^{\w }$ belongs to $k[f_3^{\w }]$. 
Since $I_F\sm \{ 3\} \neq \emptyset $, 
this implies that $I_F$ contains $1$ 
by Claim~\ref{claim:1}(iv). 
So, 
assume that $I_F$ contains $1$. 
Then, 
there exists 
$E_1\in \E _1$ such that $\degw F'<\degw F$ 
and $F'$ belongs to $\A $, 
where $F'=F\circ E_1$. 
Clearly, 
$F'(x_1)=f_1+\phi _1$ for some $\phi _1\in k[S_1]$ 
and $\deg F'(x_1)<\deg f_1$. 
On account of Claim~\ref{claim:1}(i), 
we may assume that $F'(x_1)^{\w }$ 
does not belong to $k[S_1]^{\w }$ 
by replacing $E_1$ if necessary. 
Then, $F$ and $F'$ satisfy 
all the assumptions of Proposition~\ref{prop:structure3}. 
By the first part of this proposition, 
we may conclude that either $f_1^{\w }\approx (f_3^{\w })^2$, 
or $(f_2^{\w })^2\approx (f_3^{\w })^3$ 
and $F'$ does not admit a Shestakov-Umirbaev reduction. 
We show that $F'$ admits a Shestakov-Umirbaev reduction, 
and hence the latter case is impossible. 
Observe that 
$f_2^{\w }$ and $f_3^{\w }$ 
are algebraically dependent over $k$ in either case, 
since so are $f_1^{\w }$ and $f_2^{\w }$ due to (\ref{eq:key2}). 
This implies that $\degw F'>|\w |$ by (\ref{eq:auto deg lower bound}). 
Since $F'$ is an element of $\A $, 
it follows that 
$I_{F'}\neq \emptyset $ or 
$(F'_{\sigma '},G')$ satisfies the quasi Shestakov-Umirbaev condition 
for some $\sigma '\in \sym _3$ and $G'\in \A $. 
By Proposition~\ref{prop:structure3}(1), 
$F'$ does not admit an elementary reduction. 
Hence, 
$I_{F'}=\emptyset $. 
Thus, 
$(F'_{\sigma '},G')$ satisfies the quasi Shestakov-Umirbaev condition 
for some $\sigma '\in \sym _3$ and $G'\in \A $. 
Accordingly, 
$F'$ admits a quasi Shestakov-Umirbaev reduction. 
Therefore, 
$F'$ admits a Shestakov-Umirbaev reduction 
by Proposition~\ref{prop:equivalence}(ii). 
As a result, 
we get $f_1^{\w }\approx (f_3^{\w })^2$. 
Then, it follows from 
Proposition~\ref{prop:structure3}(2) 
that $\sigma '=\id $ and 
$(F,G')$ satisfies the quasi Shestakov-Umirbaev condition. 
So, we are reduced  to the case where 
$(F_{\sigma },G)$ satisfies 
the quasi Shestakov-Umirbaev condition 
for some $\sigma \in \sym _3$ and $G\in \A $. 
By Claim~\ref{claim:1}(iv), 
we may assume that $(F_{\sigma },G)$ 
satisfies the Shestakov-Umirbaev condition 
by replacing $G$ if necessary. 
Then, there exists $E\in \E _3$ 
such that $F\circ E=G$ by Proposition~\ref{prop:structure1}. 
Since $\deg G<\deg F$ by (P6), 
and since $G$ is an element of $\A $, 
it follows that $\deg F\circ E<\deg F$, 
and $F\circ E$ belongs to $\A $. 
Thus, 
we arrive at the conclusion of Proposition~\ref{prop:key2}. 
Therefore, 
we have proved the assertion of Proposition~\ref{prop:key2} 
in the case where $\deg F=\mu $ 
on the assumption that 
the assertion of Proposition~\ref{prop:key} 
is true if $\deg F<\mu $.

To complete the induction, 
we next show the assertion of Proposition~\ref{prop:key} 
in the case where $\deg F=\mu $ 
on the assumption that 
the assertions of Propositions~\ref{prop:key} and~\ref{prop:key2} 
are true if $\deg F<\mu $ and $\deg F\leq \mu $, 
respectively. 
First, 
assume that $I_F=\emptyset $. 
Then, 
$(F_{\sigma },G)$ satisfies the quasi Shestakov-Umirbaev condition 
for some $\sigma \in \sym _3$ and $G\in \A $. 
Without loss of generality, 
we may assume that $\sigma =\id $. 
By Claim~\ref{claim:1}(iv), 
we may also assume that 
$(F,G)$ satisfies the Shestakov-Umirbaev condition 
by replacing $G$ if necessary. 
Since $I_F=\emptyset $, 
it follows that $F$ 
does not admit an elementary reduction. 
In view of (SU1), 
this implies that $(f_1,f_2)\neq (g_1,g_2)$ 
and $k[f_1,f_2]\neq k[g_1,g_2]$. 
Then, 
we know by the following claim that 
$F\circ E$ belongs to $\A $ for $E\in \E $ 
if $\degw F\circ E\leq \degw F$.

\begin{claim}\label{claim:c1-c6}
Assume that $(F,G)$ satisfies 
the quasi Shestakov-Umirbaev condition 
for some $G\in \A $, 
and $E\in \E _i$ satisfies 
$\degw F\circ E\leq \degw F$, 
where $i\in \{ 1,2,3\} $. 
If $i=1$ or $i=2$, 
or if $i=3$ and $k[f_1,f_2]\neq k[g_1,g_2]$, 
then $F\circ E$ belongs to $\A $. 
\end{claim}
\begin{proof}\rm
In the notation of Proposition~\ref{prop:c1-c6}, 
one of the pairs $(F\circ E,G)$, $(F\circ E,G')$ 
and $((F\circ E)_{\tau },G'')$ 
satisfies the quasi Shestakov-Umirbaev condition. 
Since $G$ belongs to $\A $, so do $G'$ and $G''$. 
Hence, in each case, 
$F\circ E$ 
admits a quasi Shestakov-Umirbaev reduction 
to an element of $\A $. 
Therefore, $F\circ E$ belongs to $\A $. 
\end{proof}

Therefore, 
the assertion of Proposition~\ref{prop:key} 
is true if $\deg F=\mu $ and $I_F=\emptyset $.

Next, assume that $I_F\neq \emptyset $, 
say $I_F$ contains $3$. 
We have to check that $F\circ E_i$ belongs to $\A $ 
for any $E_i\in \E _i$ with $\deg F\circ E_i\leq \deg F$ 
for each $i\in \{ 1,2,3\} $. 
By Claim~\ref{claim:1}(i), 
this is clear if $i=3$. 
Since the cases $i=1$ and $i=2$ are similar, 
we only consider the case where $i=1$. 
Since we assume that $I_F$ contains $3$, 
there exists $E_3\in \E _3$ 
such that $G:=F\circ E_3$ belongs to $\A $ 
and $\degw G<\degw F$. 
By Claim~\ref{claim:1}(i), 
we may assume that $g_3^{\w }$ 
does not belong to $k[S_3]^{\w }$ 
by replacing $E_3$ if necessary. 
Set $\phi _i=F(E_i(x_i)-x_i)$ for $i=1,3$. 
Then, 
$\phi _i$ belongs to $k[S_i]$ for $i=1,3$, 
and $g_3=f_3+\phi _3$. 
Since $\deg F\circ E_1\leq \deg F$ 
and $\deg G<\deg F$, 
we have $\degw \phi _1\leq \degw f_1$, 
$\phi _3^{\w }=-f_3^{\w }$ 
and $\deg g_3<\deg f_3$.

\begin{claim}\label{claim:2}
$F\circ E_1$ belongs to~${\cal A}$ 
if one of the following conditions holds$:$

{\rm (i)} $E_1(x_1)-x_1$ belongs to $k[x_2]$, 
or equivalently, 
$\phi _1$ belongs to $k[f_2]$.

{\rm (ii)} $f_1^{\w }$ or $f_3^{\w }$ 
belongs to $k[f_2^{\w }]$.

{\rm (iii)} $f_3^{\w }\approx f_1^{\w }+c(f_2^{\w })^p$ 
for some $c\in k$ and $p\in \N $. 
\end{claim}
\begin{proof}\rm
(i) If $E_1(x_1)-x_1$ belongs to $k[x_2]$, 
then we can define a triangular automorphism 
$H$ of $\kx $ by $H(x_2)=x_2$ and 
$H(x_i)=E_i(x_i)$ for $i=1,3$. 
Since $\deg F\circ E_3<\deg F$ 
and $3$ is contained in $I_F$, 
it follows from Claim~\ref{claim:1}(iii) that 
$F\circ E_1$ belongs to ${\cal A}$.

(ii) If $f_3^{\w }$ belongs to $k[f_2^{\w }]$, 
then $\degw (f_3-cf_2^r)<\degw f_3$ 
for some $c\in k\sm \zs $ and $r\in \N $. 
Define a triangular automorphism 
$H$ of $\kx $ by $H(x_2)=x_2$, $H(x_3)=x_3-cx_2^r$ 
and $H(x_1)=E_1(x_1)$. 
Since $\deg (F\circ H)(x_3)<\deg f_3$ 
and $3$ is contained in $I_F$, 
it follows from Claim~\ref{claim:1}(iii) 
that $F\circ E_1$ belongs to ${\cal A}$. 
If $f_1^{\w }$ belongs to $k[f_2^{\w }]$, 
then $I_F$ contains $1$ by Claim~\ref{claim:1}(iv), 
since $I_F\sm \{ 2\} \neq \emptyset $. 
Therefore, 
$F\circ E_1$ belongs to ${\cal A}$ by Claim~\ref{claim:1}(i).

(iii) 
By assumption, 
there exists $c'\in k\sm \zs $ such that 
$\degw f'<\degw f_3$, 
where $f'=f_3+c'(f_1+cf_2^p)$. 
Define $E_1',E_1''\in \E _1$ and $E_3'\in \E _3$ 
by $E_1'(x_1)=x_1+cx_2^p-(1/c')x_3$, 
$E_1''(x_1)=(c')^{-1}(x_3+c'(x_1+cx_2^p))$ 
and $E_3'(x_3)=x_3+c'(x_1+cx_2^p)$. 
Then, $\degw F\circ E_3'<\degw F$, 
because $(F\circ E_3')(x_3)=f'$. 
Since $3$ is contained in $I_F$ by assumption, 
$F\circ E_3'$ belongs to ${\cal A}$ by Claim~\ref{claim:1}(i). 
Hence, 
$F':=(F\circ E_3')\circ E_1'$ belongs to ${\cal A}$ 
by the induction assumption of Proposition~\ref{prop:key}. 
Since $F'=(-(1/c')f_3,f_2,f')$, 
this implies that 
$F\circ E_1''=((1/c')f',f_2,f_3)$ belongs to ${\cal A}$. 
By assumption, 
it follows that $\deg f_3=\deg f_1$. 
Hence, 
$\deg F\circ E_1''<\deg F$. 
Thus, $1$ belongs to $I_F$. 
Therefore, 
$F\circ E_1$ belongs to $\A $ by Claim~\ref{claim:1}(i). 
\end{proof}

In the case where 2 belongs to $I_F$ besides 3, 
the statement of Claim~\ref{claim:2} is true 
if we interchange $f_2$ and $f_3$. 
Hence, 
we obtain the following claim. 

\begin{claim}\label{claim:2'}
Assume that $2$ is contained in $I_F$. 
If $\phi _1$ belongs to $k[f_3]$, 
or if $f_1^{\w }$ or $f_2^{\w }$ belongs to $k[f_3^{\w }]$, 
then $F\circ E_1$ belongs to $\A $. 
\end{claim}

Now, 
there exist five cases to be considered as follows: 
\begin{align*}
&{\rm (1)} \degw f_1=\degw f_2=\degw f_3; \qquad 
{\rm (2)} \degw f_1<\degw f_2=\degw f_3; \\
&{\rm (3)} \degw f_3<\degw f_1=\degw f_2; \qquad 
{\rm (4)} \degw f_2<\degw f_3=\degw f_1; \\
&{\rm (5)} \degw f_l<\degw f_m\text{ for each }
l\in \{ 1,2,3\} \sm \{ m\} 
\text{ for some }m\in \{ 1,2,3\} . 
\end{align*}

Here, we remark that 
the cases (1)--(4) can be excluded from consideration 
in the case where $\rank \w =3$. 
In fact, 
$\degw f_i=\degw f_j$ implies $f_i^{\w }\approx f_j^{\w }$ 
for each $i$ and $j$ if $\rank \w =3$. 
Hence, 
it immediately follows from 
Claim~\ref{claim:2}(ii) and (iii) that 
$F\circ E_1$ belongs to $\A $ in cases (1)--(4). 
For this reason, 
Claim~\ref{claim:3} and the statement (I) of Claim~\ref{claim:4} below 
are not necessary when considering $\w $ with $\rank \w =3$.

\begin{claim}\label{claim:3}
$F\circ E_1$ belongs to $\A $ if one of the following holds$:$ 

{\rm (i)} $f_1^{\w }$ and $f_2^{\w }$ 
are algebraically independent over $k$. 

{\rm (ii)}  $F$ satisfies one of 
{\rm (1)}, {\rm (2)} and {\rm (3)}. 
\end{claim}
\begin{proof}\rm
By Claim~\ref{claim:2}(i), 
we may assume that $\phi _1$ belongs to $k[f_2,f_3]\sm k[f_2]$. 
Then, it follows that, 
if $\degw f_1<\degw f_3$, 
then $f_2^{\w }$ and $f_3^{\w }$ 
are algebraically dependent over $k$. 
In fact, 
since $\degw \phi _1\leq \degw f_1<\deg f_3$, 
and since $\phi _1$ belongs to $k[f_2,f_3]\sm k[f_2]$, 
we have $\deg \phi _1<\degw ^{S_1}\phi _1$. 
Hence, 
$(f_2^{\w })^p\approx (f_3^{\w })^q$ 
for some $p,q\in \N $ by Lemma~\ref{lem:degS1}.

(i) 
Recall that $f_3^{\w }\approx \phi _3^{\w }$ 
and $\phi _3$ is an element of $k[S_3]$. 
Hence, $f_3^{\w }$ belongs to $k[S_3]^{\w }$. 
Since 
$f_1^{\w }$ and $f_2^{\w }$ 
are algebraically independent over $k$, 
we have $k[S_3]^{\w }=k[f_1^{\w },f_2^{\w }]$. 
Thus, 
$f_3^{\w }$ is a polynomial in 
$f_1^{\w }$ and $f_2^{\w }$ over $k$. 
By Claim~\ref{claim:2}(ii), 
we may assume that $f_3^{\w }$ 
does not belong to $k[f_2^{\w }]$. 
Then, it follows that 
$\degw f_1\leq \degw f_3$. 
We show that $\degw f_1=\degw f_3$ by contradiction. 
Supposing $\degw f_1<\degw f_3$, 
we get that $f_2^{\w }$ and $f_3^{\w }$ 
are algebraically dependent over $k$ 
as remarked above. 
Since $f_3^{\w }$ is an element of 
$k[f_1^{\w },f_2^{\w }]\sm k[f_2^{\w }]$, 
it follows that 
$f_1^{\w }$ and $f_2^{\w }$ 
are algebraically dependent over $k$, 
a contradiction. 
Thus, 
$\degw f_1=\degw f_3$. 
This implies that 
$f_3^{\w }\approx f_1^{\w }+c(f_2^{\w })^p$ 
for some $c\in k$ and $p\in \N $. 
Therefore, 
$F\circ E_1$ belongs to $\A $ by Claim~\ref{claim:2}(iii).

(ii) By (i), we may assume that $f_1^{\w }$ and $f_2^{\w }$ 
are algebraically dependent over $k$. 
Then, 
$f_1^{\w }\approx f_2^{\w }$ follows from $\degw f_1=\degw f_2$ 
in cases (1) and (3). 
In case (2), 
it follows from $\deg f_1<\deg f_3$ that 
$f_2^{\w }$ and $f_3^{\w }$ 
are algebraically dependent over $k$ as remarked above. 
Then, $f_2^{\w }\approx f_3^{\w }$ 
follows from $\degw f_3=\degw f_2$. 
By Claim~\ref{claim:2}(ii), 
$F\circ E_1$ belongs to $\A $ in every case. 
\end{proof}

Let us complete the proof of 
Proposition~\ref{prop:key} by contradiction. 
Suppose to the contrary that 
$F\circ E_1$ does not belong to $\A $. 
Then, 
the conditions (i), (ii) and (iii) of Claim~\ref{claim:2} 
and (i) and (ii) of Claim~\ref{claim:3} cannot be satisfied. 
In particular, $F$ satisfies (4) or (5). 
Furthermore, 
$f_1^{\w }$ and $f_3^{\w }$ 
must be algebraically independent over $k$ in case (4). 
We show that, 
if $F$ satisfies (5) for $m=2$, 
and if $f_2^{\w }$ does not belong to $k[f_1^{\w }]$, 
then $f_3^{\w }$ does not belong to $k[f_1^{\w }]$. 
Supposing the contrary, we have 
$f_3^{\w }\approx (f_1^{\w })^p$ for some $p\in \N $. 
Then, $p\geq 2$ in view of Claim~\ref{claim:2}(iii). 
Hence, $\deg f_1<\deg f_3$. 
We verify that $f=f_3$, $g=f_2$ and $\phi =\phi _1$ 
satisfy the assumptions of Lemma~\ref{lem:degS2}(ii) 
with $\deg \phi <\deg f$. 
Recall that 
$\phi _1$ is an element of $k[f_2,f_3]$ 
such that $\deg \phi _1\leq \deg f_1$. 
Since $\deg f_1<\deg f_3$, 
we have $\deg \phi _1<\deg f_3$. 
On account of Claim~\ref{claim:2}(i), 
$\phi _1$ cannot belong to $k[f_2]$. 
Thus, it follows that 
$\degw \phi _1<\degw ^{S_1}\phi _1$. 
By assumption, 
$f_2^{\w }$ does not belong to $k[f_1^{\w }]$. 
Since $f_3^{\w }\approx (f_1^{\w })^p$, 
it follows that 
$f_2^{\w }$ does not belong to $k[f_3^{\w }]$. 
By the condition (5) for $m=2$, 
we have $\deg f_3<\deg f_2$. 
Thus, 
the assumptions of Lemma~\ref{lem:degS2}(ii) are satisfied, 
and so we conclude that 
$$
\degw \phi _1\geq 
(3-2)\frac{1}{2}\degw f_3+\degw df_2\wedge df_3
>\frac{1}{2}\deg f_3=\frac{p}{2}\degw f_1\geq \degw f_1. 
$$
This contradicts that $\deg \phi _1\leq \deg f_1$. 
Therefore, 
$f_3^{\w }$ does not belong to $k[f_1^{\w }]$ 
if $F$ satisfies (5) for $m=2$, 
and  $f_2^{\w }$ does not belong to $k[f_1^{\w }]$. 

\begin{claim}\label{claim:4}
If $F\circ E_1$ does not belong to $\A $, 
then one of the following holds$:$ 

{\rm (I)} $\degw f_2<\degw f_1$, 
$\degw f_1=\degw f_3$, 
$f_1^{\w }\not\approx f_3^{\w }$, 
and $f_1^{\w }$ and $f_3^{\w }$
do not belong to $k[f_2^{\w }]$ 
and  $k[f_1^{\w },f_2^{\w }]$, 
respectively.

{\rm (II)} 
$\degw f_i<\degw f_j$, $\degw f_3<\degw f_j$, 
and $f_j^{\w }$ and $f_3^{\w }$ 
do not belong to $k[f_i^{\w }]$ for some 
$(i,j)\in \{ (1,2),(2,1)\} $.

{\rm (III)} $\degw f_1<\degw f_j$, $\degw f_i<\degw f_j$, 
$f_1^{\w }$ and $f_j^{\w }$ do not belong to $k[f_i^{\w }]$, 
and $\phi _1$ belongs to $k[S_1]\sm k[f_i]$ 
for some $(i,j)\in \{ (2,3),(3,2)\} $. 
\end{claim}
\begin{proof}\rm
We show that $F$ satisfies (I) in case (4), 
where $\deg f_2<\deg f_1$ and $\deg f_1=\deg f_3$. 
On account of Claim~\ref{claim:2}(ii) and (iii), 
$f_l^{\w }$ does not belong to $k[f_2^{\w }]$ for $l=1,3$, 
and $f_3^{\w }\not\approx f_1^{\w }$. 
We show that $f_3^{\w }$ 
does not belong to $k[f_1^{\w },f_2^{\w }]$ by contradiction. 
Supposing the contrary, 
we have $f_3^{\w }=af_1^{\w }+b(f_2^{\w })^p$ 
for some $a,b\in k$ with $(a,b)\neq (0,0)$ and $p\geq 2$, 
since $\deg f_3=\deg f_1$ and $\deg f_1>\deg f_2$. 
If $a=0$ or $b=0$, 
then $f_3^{\w }$ belongs to $k[f_2^{\w }]$ 
or $f_3^{\w }\approx f_1^{\w }$, 
contradictions. 
Hence, 
$a\neq 0$ and $b\neq 0$. 
It follows that $\degw f_1^{\w }=\degw (f_2^{\w })^p$. 
Owing to Claim~\ref{claim:3}(i), 
$f_1^{\w }$ and $f_2^{\w }$ 
must be algebraically dependent over $k$. 
Thus, 
$f_1^{\w }\approx (f_2^{\w })^p$, 
and so $f_1^{\w }$ belongs to $k[f_2^{\w }]$, 
a contradiction. 
Therefore, 
$f_3^{\w }$ 
does not belong to $k[f_1^{\w },f_2^{\w }]$. 
This proves that $F$ satisfies (I) in case (4). 

We show that $F$ satisfies (II) or (III) in case (5). 
Since the conditions (i), (ii) and (iii) 
of Claim~\ref{claim:2} are not satisfied by supposition, 
(II) holds for $(i,j)=(2,1)$ if $m=1$, 
and (III) holds for $(i,j)=(2,3)$ if $m=3$. 
Assume that $m=2$. 
As shown before this claim, 
if $f_2^{\w }$ does not belong to $k[f_1^{\w }]$, 
then neither does $f_3^{\w }$. 
Hence, (II) holds for $(i,j)=(1,2)$. 
If $f_2^{\w }$ belongs to $k[f_1^{\w }]$, 
then $I_F$ contains $2$ by Claim~\ref{claim:1}(iv), 
since $I_F\sm \{ 1\} \neq \emptyset $. 
By Claim~\ref{claim:2'}, 
we know that 
$\phi _1$ belongs to $k[S_1]\sm k[f_3]$, 
and $f_1^{\w }$ and $f_2^{\w }$ do not belong to $k[f_3^{\w }]$. 
Therefore, (III) holds for $(i,j)=(3,2)$. 
\end{proof}

We consider the cases (I) and (II) together. 
Recall that 
$\phi _3^{\w }\approx f_3^{\w }$, 
$\deg g_3<\deg f_3$, 
$g_3^{\w }$ does not belong to $k[S_3]^{\w }$, 
and $G=(f_1,f_2,g_3)$ belongs to $\A $. 
We establish the inequality  
\begin{equation}\label{eq:proofineq1}
\degw g_3<\degw f_j-\degw f_i+\degw df_1\wedge df_2
\end{equation}
by contradiction, 
where we set $(i,j)=(2,1)$ in case (I). 
In case (I), 
$f_3^{\w }$ does not belong to $k[f_1^{\w },f_2^{\w }]$, 
and hence neither does $\phi _3^{\w }$. 
The same holds true in case (II) 
because $k[f_1^{\w },f_2^{\w }]=k[f_i^{\w },f_j^{\w }]$, 
$\degw f_3<\degw f_j$ and 
$f_3^{\w }$ does not belong to $k[f_i^{\w }]$. 
Since $\phi _3$ is an element of $k[S_3]$, 
it follows that $\degw \phi _3<\degw ^{S_3}\phi _3$ 
in both cases. 
We show that $G':=(f_j,f_i,g_3)$ 
satisfies the assumptions of Proposition~\ref{prop:key2}. 
Clearly, 
$G'$ is an element of $\A $, 
since so is $G$ by assumption. 
By the conditions in (I) and (II), 
$\deg f_i<\deg f_j$, 
$\degw \phi _3=\degw f_3\leq \degw f_j$, 
and $f_j^{\w }$ does not belong to $k[f_i^{\w }]$. 
Hence, 
it follows from Lemma~\ref{lem:degS2}(ii) that 
$\degw f_i=2\delta $ and $\degw f_j=s\delta $ 
for some $\delta \in \Gamma $ and an odd number $s\geq 3$. 
Since (\ref{eq:proofineq1}) 
is supposed to be false, 
we get 
$$
(s-2)\delta +\degw df_1\wedge df_2
=\deg f_j-\deg f_i+\degw df_1\wedge df_2\leq 
\degw g_3<\degw f_3\leq \degw f_j=s\delta . 
$$
Since $k[S_3]^{\w }$ does not contain $g_3^{\w }$, 
neither does $k[f_i]^{\w }$. 
Thus, $G'$ 
satisfies the assumptions of Proposition~\ref{prop:key2}. 
Because $\deg G'<\deg F=\mu $, 
we may conclude that 
there exists $E_3'\in \E _3$ such that 
$\degw G'\circ E_3'<\degw G'$ 
by induction assumption. 
This contradicts that 
$g_3^{\w }$ does not belong to $k[S_3]^{\w }$, 
thereby proves that 
(\ref{eq:proofineq1}) is true. 
We show that 
$(F',G')$ satisfies the quasi Shestakov-Umirbaev condition, 
where $F'=(f_j,f_i,f_3)$. 
The first two conditions of (SU$1'$), and (SU$2'$) are obvious. 
The last condition of (SU$1'$), and (SU5) 
follow from the construction of $g_3$. 
(SU$3'$) and the first condition of (SU4) 
are included in (I) and (II). 
As mentioned after (\ref{eq:proofineq1}), 
$f_3^{\w }$ does not belong to $k[f_1^{\w },f_2^{\w }]$, 
which is the last condition of (SU4). 
(SU6) is due to (\ref{eq:proofineq1}). 
Thus, 
$(F',G')$ satisfies the quasi Shestakov-Umirbaev condition. 
It follows from Claim~\ref{claim:c1-c6} that 
$F'\circ E$ belongs to $\A $ for each $E\in \E _l$ for $l=1,2$ 
if $\degw F'\circ E\leq \degw F'$. 
In particular, 
$(F\circ E_1)\circ H=F'\circ (H\circ E_1\circ H)$ 
belongs to $\A $, 
where $H=(x_j,x_i,x_3)$. 
Actually, 
$H\circ E_1\circ H$ belongs to $\E _j$, 
and 
$$
\deg F'\circ H\circ E_1\circ H=\deg F\circ E_1\circ H
=\deg F\circ E_1\leq \deg F=\deg F'. 
$$
This implies that $F\circ E_1$ belongs to $\A $. 
Therefore, 
we are led to a contradiction.

Finally, 
we derive a contradiction in case (III). 
It follows that $\degw \phi _1<\degw ^{S_1}\phi _1$, 
since $\phi _1$ is an element of $k[f_i,f_j]\sm k[f_i]$ 
with $\deg \phi _1\leq \deg f_1<\deg f_j$. 
Since $\degw f_i<\degw f_j$, 
and $f_j^{\w }$ does not belong to $k[f_i^{\w }]$, 
we know that $f_i$, $f_j$ and $\phi _1$ 
satisfy the assumptions of Lemma~\ref{lem:degS2}(ii). 
Hence, 
there exist $\delta \in \Gamma $ and 
an odd number $s\geq 3$ 
such that $\degw f_i=2\delta $, $\degw f_j=s\delta $ and 
$$
(s-2)\delta +\degw df_2\wedge df_3
=(s-2)\delta +\degw df_i\wedge df_j\leq 
\degw \phi _1\leq \degw f_1<\degw f_j=s\delta . 
$$
Thus, $F_{\tau }$ satisfies (\ref{eq:key2}) 
for 
$\tau \in \sym _3$ with $\tau (1)=j$, $\tau (2)=i$ and $\tau (3)=1$. 
Note that $F_{\tau }$ is an element of $\A $ 
with $\deg F_{\tau }=\mu $, 
since so is $F$. 
As $f_1^{\w }$ does not belong to $k[f_i^{\w }]$, 
the assumptions of Proposition~\ref{prop:key2} 
are fulfilled for $F_{\tau }$. 
Hence, 
by induction assumption, 
we conclude that 
$\degw F_{\tau }\circ E_3'<\degw F_{\tau }$ 
and $F_{\tau }\circ E_3'$ belongs to ${\cal A}$ 
for some $E_3'\in \E _3$. 
Thus, $I_{F_{\tau }}$ contains $3$, 
and so $I_F$ contains $1$. 
Therefore, 
$F\circ E_1$ belongs to ${\cal A}$ by Claim~\ref{claim:1}(i), 
a contradiction.

This proves that the statement of Proposition~\ref{prop:key} 
holds for each $F\in \A $ with $\deg F=\mu $. 
Thus, 
the proofs of Propositions~\ref{prop:key} and~\ref{prop:key2} 
are completed by induction. 
Thereby, we have completed the proof Theorem~\ref{thm:main1}.

\section{Relations with the theory of Shestakov-Umirbaev}
\label{sect:remark}
\setcounter{equation}{0}

In this section, 
we discuss relations with the original theory 
of Shestakov-Umirbaev. 
Throughout this section, 
we assume that $\Gamma =\Z $ and $\w =(1,1,1)$. 
Hence, 
$\deg F\geq |\w |=3$ for each $F\in \Aut _k\kx $. 
First, we recall the notions of 
reductions of types I, II, III and IV 
defined by Shestakov-Umirbaev~\cite[Definitions 1, 2, 3 and 4]{SU2}.

\smallskip

Let $F=(f_1,f_2,f_3)$ be an element of $\Aut _k\kx $ 
such that $\degw f_1=2l$ and $\degw f_2=sl$ 
for some $l\in \N $ and an odd number $s\geq 3$. 

\smallskip

(1) $F$ is said to {\it admit a reduction of type I} 
if $2l<\degw f_3\leq sl$, 
$f_3^{\w }$ does not belong to $k[f_1^{\w },f_2^{\w }]$, 
and there exists $\alpha \in k\sm \zs $ for which 
$g_1:=f_1$ and $g_2:=f_2-\alpha f_3$ 
satisfy the following conditions: 

(i) $\degw g_2=sl$, and 
$g_1^{\w }$ and $g_2^{\w }$ 
are algebraically dependent over $k$. 

(ii) $\degw g_3<\degw f_3$ and 
$\degw dg_1\wedge dg_3<sl+\degw dg_1\wedge dg_2$ 
for some $\phi \in k[g_1,g_2]$, 
where $g_3=f_3+\phi $. 

\smallskip

(2) $F$ is said to {\it admit a reduction of type II} if $s=3$, 
$(3/2)l<\degw f_3\leq 2l$, 
$f_1^{\w }\not\approx f_3^{\w }$, 
and there exist $\alpha ,\beta \in k$ 
with $(\alpha ,\beta )\neq (0,0)$ for which 
$g_1:=f_1-\alpha f_3$ and 
$g_2:=f_2-\beta f_3$ satisfy the following conditions: 

(iii) $\degw g_1=2l$, $\degw g_2=3l$, and 
$g_1^{\w }$ and $g_2^{\w }$ 
are algebraically dependent over $k$. 

(iv) $\degw g_3<\degw f_3$ and 
$\degw dg_1\wedge dg_3<3l+\degw dg_1\wedge dg_2$ 
for some $\phi \in k[g_1,g_2]$, 
where $g_3=f_3+\phi $. 

\smallskip

Next, 
let $F=(f_1,f_2,f_3)$ be an element of $\Aut _k\kx $ 
such that $\degw f_1=2l$, 
and either $\degw f_2=3l$ and $l<\degw f_3\leq (3/2)l$, 
or $(5/2)l<\degw f_2\leq 3l$ and $\degw f_3=(3/2)l$ 
for some $l\in \N $. 
Assume that there exist $\alpha ,\beta ,\gamma \in k$ 
such that $g_1:=f_1-\beta f_3$ and 
$g_2:=f_2-\gamma f_3-\alpha f_3^2$ 
satisfy the following conditions: 

(v) $\degw g_1=2l$, $\degw g_2=3l$, 
and $g_1^{\w }$ and $g_2^{\w }$ 
are algebraically dependent over $k$.

(vi) $\degw g_3\leq (3/2)l$ and 
$\degw dg_1\wedge dg_3<3l+\degw dg_1\wedge dg_2$ 
for some $\sigma \in k\sm \zs $ and $g\in k[g_1,g_2]\sm k$, 
where $g_3=\sigma f_3+g$. 

\smallskip
(3) $F$ is said to {\it admit a reduction of type III} 
if we can choose $\alpha $, $\beta $, $\gamma $, $\sigma $ 
and $g$ 
so that 
$(\alpha ,\beta ,\gamma )\neq (0,0,0)$ and 
$\degw g_3<l+\degw dg_1\wedge dg_2$. 

\smallskip

(4) $F$ is said to {\it admit a reduction of type IV} 
if we can choose $\alpha $, $\beta $, $\gamma $, $\sigma $ and $g$ 
so that $\degw (g_2-\mu g_3^2)\leq 2l$ 
for some $\mu \in k\sm \zs $. 

\smallskip

We also say that $F$ admits a reduction of 
type I (resp.\ II, III and IV) 
if $F_{\sigma }$ satisfies (1) (resp.\ (2), (3) and (4)) 
for some $\sigma \in \sym _3$. 

\smallskip

Here, we note that the conditions (i), (iii) and (v) 
are equivalent to the condition that 
$g_1$, $g_2$ is a ``two-reduced pair", 
since the conditions on $\deg g_1$ and $\deg g_2$ 
imply $g_1^{\w }\not\in k[g_2^{\w }]$ 
and $g_2^{\w }\not\in k[g_1^{\w }]$. 
Although Shestakov-Umirbaev~\cite{SU2} considered the 
``Poisson bracket'' $[f,g]$ instead of $df\wedge dg$ 
for $f,g\in \kx $, 
the degrees of $[f,g]$ and $df\wedge dg$ 
are defined in the same way.

The following theorem is a consequence of Theorem~\ref{thm:main1} 
and Proposition~\ref{prop:structure3}. 

\begin{theorem}\label{thm:type IV}
No tame automorphism of $\kx $ 
admits a reduction of type IV. 
\end{theorem}
\begin{proof}
Suppose to the contrary that 
$F$ satisfies (4) for some $F\in \T _k\kx $. 
Then, 
$g_1$ and $g_2$ appearing in the condition satisfy 
$\degw g_1=2l$ and $\degw g_2=3l$. 
Moreover, 
since $\degw (g_2-\mu g_3^2)\leq 2l<(5/2)l<\deg g_2$ 
for some $\mu \in k\sm \zs $, 
we have $g_2^{\w }\approx (g_3^{\w })^2$. 
Hence, 
$\degw g_3=(3/2)l$. 
Since $F$ belongs to $\T _k\kx $, so does $H:=(g_2,g_1,g_3)$. 
We show that $H$ satisfies 
the assumptions of 
Proposition~\ref{prop:structure3} for $s=3$ and $\delta =l$. 
The degrees of $g_2$, $g_1$ and $g_3$ satisfy 
(\ref{eq:ana condition}), 
and $g_3^{\w }$ does not belong to $k[g_1^{\w }]$, 
since $\deg g_3<\deg g_1$. 
We verify that $\degw dg_1\wedge dg_2\leq (1/2)l$, 
which gives (\ref{eq:ppfstr2}) that 
$$
\degw dg_1\wedge dg_2\leq \frac{1}{2}l
<\frac{3}{2}l-l+1\leq 
\degw g_3-(3-2)l+\min \{ l,\ep \} , 
$$
since 
$\ep =\degw dg_1\wedge dg_2\wedge dg_3=3$ and $l\geq 1$. 
By definition, 
$g$ is an element of $k[g_1,g_2]\sm k$ 
such that 
$\deg g\leq \max \{ \deg f_3,\deg g_3\} =(3/2)l<\deg g_i$ for $i=1,2$. 
Hence, 
$g^{\w }$ does not belong to $k[g_1^{\w },g_2^{\w }]$, 
and so $\deg g<\deg ^Ug$, 
where $U=\{ g_1,g_2\} $. 
Since $\deg g_1=2l$ and $\deg g_2=3l$, 
it follows that $\deg g_1<\deg g_2$ and $g_2^{\w }$ 
does not belong to $k[g_1^{\w }]$. 
Thus, 
$$
\deg g\geq (3-2)l+\degw dg_1\wedge dg_2
=l+\degw dg_1\wedge dg_2
$$
by Lemma~\ref{lem:degS2}(ii). 
Since $\deg g\leq (3/2)l$, 
we conclude that 
$\degw dg_1\wedge dg_2\leq (1/2)l$. 
Therefore, 
$H$ satisfies the assumptions of Proposition~\ref{prop:structure3}. 
Take $\phi _2\in k[g_1,g_3]$ so that 
$(g_2')^{\w }$ does not belong to $k[g_1,g_3]^{\w }$, 
where $g_2'=g_2+\phi _2$. 
Then, 
$\degw g_2'\leq 2l$, 
since $\degw (g_2-\mu g_3^2)\leq 2l$. 
By Proposition~\ref{prop:structure3}(1), 
$H':=(g_2',g_1,g_3)$ does not admit an elementary reduction. 
Since $H$ belongs to $\T _k\kx $, 
so does $H'$. 
Furthermore, 
$\degw H'>3$, 
because $\degw g_i>l\geq 1$ for $i=1,3$. 
Thus, 
$H'$ admits a Shestakov-Umirbaev reduction 
by Theorem~\ref{thm:main1}. 
Hence, 
there exist $\sigma \in \sym _3$ and $K\in \Aut _k\kx $ 
such that 
$(H'_{\sigma },K)$ satisfies the Shestakov-Umirbaev condition. 
Since $g_2^{\w }\approx (g_3^{\w })^2$ as mentioned, 
we know that $\sigma =\id $ by Proposition~\ref{prop:structure3}(2). 
Hence, 
$(H',K)$ satisfies the Shestakov-Umirbaev condition. 
Consequently, we get $\degw g_1<\degw g_2'$ by (P7). 
This contradicts that 
$\deg g_1=2l$ and $\deg g_2'\leq 2l$. 
Therefore, 
$F$ does not admit a reduction of type IV\null. 
\end{proof}

To conclude that Nagata's automorphism is not tame, 
Shestakov-Umirbaev~\cite[Theorem 1]{SU2} showed 
that, 
if $\degw F>3$ for $F\in \T _k\kx $, 
then $F$ admits an elementary reduction 
or a reduction of one of the types I, II, III and IV\null. 
With the aid of the following proposition, 
the criterion of Shestakov-Umirbaev 
is derived from Theorem~\ref{thm:main1}.

\begin{proposition}\label{prop:sured to 123}
Assume that $(F,G)$ satisfies the Shestakov-Umirbaev condition 
for $F,G\in \Aut _k\kx $. 
If $(f_1,f_2)=(g_1,g_2)$, 
then $F$ admits an elementary reduction. 
If $(f_1,f_2)\neq (g_1,g_2)$, 
then $F$ admits a reduction of 
one of the types I, II and III. 
\end{proposition}
\begin{proof}\rm
The first assertion follows from (SU1) and (SU5). 
We show the last assertion. 
By (SU1), we may write 
$g_1=f_1+af_3^2+cf_3$, $g_2=f_2+bf_3$ and $g_3=f_3+\phi $, 
where $a,b,c\in k$ and $\phi \in k[g_1,g_2]$. 
Since $(f_1,f_2)\neq (g_1,g_2)$ by assumption, 
we have $(a,b,c)\neq (0,0,0)$. 
By (SU3), 
there exist $l\in \N $ and an odd number $s\geq 3$ 
such that $\degw g_1=sl$ and $\degw g_2=2l$. 
Then, it follows that $l<\degw f_3\leq sl$ by (P7). 
Put $\tau =(1,2)$. 
We show that $F_{\tau }$ satisfies 
(1) for $\alpha =-c$ if $2l<\degw f_3\leq sl$, 
(2) for $(\alpha ,\beta )=(-b,-c)$ if $(3/2)l<\degw f_3\leq 2l$, 
and (3) for $(\alpha ,\beta ,\gamma )=(-a,-b,-c)$, 
$\sigma =1$ and $g=\phi $ if $l<\degw f_3\leq (3/2)l$.

Note that $\deg f_2=2l$ by (SU2), 
and $\deg f_1=sl$ if $\deg f_3\neq (3/2)l$, 
and $(5/2)l<\deg f_1\leq 3l$ otherwise by (P5). 
Moreover, 
$s=3$ if $\deg f_3\leq 2l$ by (P11). 
From this, 
we see  that the conditions on the degrees of $f_1$ and $f_2$ 
are satisfied in every case. 
It follows that $a=b=0$ if $2l<\degw f_3\leq sl$ by (P11), 
and $a=0$ if $(3/2)l<\deg f_3\leq 2l$, 
since $\deg f_3^2>3l=\deg g_1$. 
Hence, 
$g_2=f_2$ and $g_1=f_1-\alpha f_3$ 
for $\alpha =-c$ if $2l<\degw f_3\leq sl$, 
$g_2=f_2-\alpha f_3$ and $g_1=f_1-\beta f_3$ 
for $(\alpha ,\beta )=(-b,-c)$ if $(3/2)l<\deg f_3\leq 2l$, 
and $g_2=f_2-\beta f_3$ and $g_1=f_1-\gamma f_3-\alpha f_3^2$ 
for $(\alpha ,\beta ,\gamma )=(-a,-b,-c)$ if $l<\degw f_3\leq (3/2)l$, 
in which $\alpha \neq 0$, 
$(\alpha ,\beta )\neq (0,0)$, 
and $(\alpha ,\beta ,\gamma )\neq (0,0,0)$, 
respectively. 
Besides, 
$g=\phi $ in (iv) cannot be an element of $k$, 
since $\deg g_3<\deg f_3$ by (SU5). 
So, 
we verify that (i)--(vi) are satisfied for $g_2$, $g_1$ and $g_3$. 
As mentioned, 
$\deg g_2=2l$ and $\deg g_1=sl$, 
where $s=3$ if $\deg f_3\leq 2l$. 
By (SU3), 
$g_2^{\w }$ and $g_1^{\w }$ are algebraically dependent over $k$. 
Thus, 
(i), (iii) and (v) are satisfied. 
The first conditions in (ii) and (iv) are the same as (SU5). 
If $\deg f_3\leq (3/2)l$, 
then $\deg g_3<\deg f_3\leq (3/2)l$ by (SU5), 
the first condition in (vi). 
The second conditions in (ii), (iv) and (vi) follow from (SU6), 
since 
\begin{equation*}
\degw dg_2\wedge dg_3
\leq \degw g_2+\degw g_3
<\deg g_2+(\deg g_1-\deg g_2+\degw dg_1\wedge dg_2)
=sl+\degw dg_1\wedge dg_2. 
\end{equation*}
Therefore, 
(i)--(vi) are satisfied for $g_2$, $g_1$ and $g_3$.

Let us check the other conditions. 
It follows from (P8) that $f_2^{\w }\not\approx f_3^{\w }$. 
Hence, 
$F_{\tau }$ satisfies (2) in case $(3/2)l<\degw f_3\leq 2l$. 
We have already shown that 
$F_{\tau }$ satisfies the assumption of (3) 
in case $l<\degw f_3\leq (3/2)l$. 
Since the last condition in (3) is the same as (SU6), 
$F_{\tau }$ satisfies (3) in this case. 
We show that 
$f_3^{\w }$ does not belong to 
$k[f_1^{\w },f_2^{\w }]$ as required in (1). 
By (P8), 
$f_3^{\w }$ does not belong to 
$k[f_1^{\w }]$ nor $k[f_2^{\w }]$. 
Since $\degw f_3\leq \degw f_1$ by (P7), 
we have $\degw f_3<\degw f_1+\degw f_2=\deg f_1f_2$. 
Hence, 
$f_3^{\w }$ does not belong to 
$k[f_1^{\w },f_2^{\w }]$. 
This proves that $F_{\tau }$ satisfies (1) in case 
$2l<\degw f_3\leq sl$. 
Therefore, 
$F$ admits a reduction of 
one of the types I, II and III 
if $(f_1,f_2)\neq (g_1,g_2)$. 
\end{proof}

\section{Remarks}\setcounter{equation}{0}
\label{sect:last}

In closing, 
we make some remarks on Shestakov-Umirbaev reductions. 
As established in Section~\ref{sect:proof}, 
for each $F\in \T _k\kx $ with $\deg _{\w }F>|\w |$, 
there exists a sequence $(G_i)_{i=0}^r$ 
of elements of $\T _k\kx $ for some $r\in \N $ 
such that $G_0=F$, $\degw G_r=|\w |$, 
and $G_{i+1}$ is an elementary reduction 
or a {\it quasi} Shestakov-Umirbaev 
reduction of $G_i$ for each $i$. 
We have a more precise result as follows.

\begin{corollary}\label{cor:well-ordered}
For each $F\in \T _k\kx $ with $\deg F>|\w |$, 
there exists a sequence $(G_i)_{i=0}^r$ 
of elements of $\T _k\kx $ for some $r\in \N $ 
such that $G_0=F$, $\degw G_r=|\w |$, 
and $G_{i+1}$ is an elementary reduction 
or a Shestakov-Umirbaev 
reduction of $G_i$ for each $i$. 
\end{corollary}
\begin{proof}\rm
Let ${\cal B}$ be the set of $F\in \T _k\kx $ 
with $\deg F>|\w |$ for which there does not exist 
a sequence as claimed. 
Suppose to the contrary that 
${\cal B}$ is not empty. 
Then, 
we can find $F\in {\cal B}$ such that 
$\deg F=\min \{ \deg H\mid H\in {\cal B}\} >|\w |$, 
since $\Sigma $ is a well-ordered set by 
Lemma~\ref{lem:|omega|}(ii). 
Since $F$ is tame, 
there exists $G\in \T _k\kx $ 
which is an elementary reduction 
or a Shestakov-Umirbaev reduction of $F$ 
by Theorem~\ref{thm:main1}. 
Then, $\deg G<\deg F$ by (P6). 
Hence, 
$G$ does not belong to ${\cal B}$ by the minimality of $\deg F$. 
It follows from the definition of ${\cal B}$ 
that $\deg G=|\w |$ or 
there exists a sequence as claimed for $G$. 
In either case, 
$F$ cannot be an element of ${\cal B}$, 
a contradiction. 
Therefore, ${\cal B}$ is empty. 
\end{proof}

For each $F\in \T _k\kx $ with $\degw F>|\w |$ 
and a sequence ${\cal G}=(G_i)_{i=0}^r$ 
as in Corollary~\ref{cor:well-ordered}, 
we define $\SU _{\w }(F;{\cal G})$ 
to be the number of $i\in \{ 1,\ldots ,r\} $ 
such that $G_{i+1}$ is a Shestakov-Umirbaev reduction of $G_i$. 
We define the {\it Shestakov-Umirbaev number} 
$\SU _{\w }(F)$ for the weight $\w $ 
to be the minimum among $\SU _{\w }(F;{\cal G})$ for the sequences 
${\cal G}=(G_i)_{i=0}^r$ as in Corollary~\ref{cor:well-ordered}. 
It may be an interesting question 
to ask whether $\SU _{\w }(F;{\cal G})=\SU _{\w }(F)$ 
for any $F\in \T _k\kx $ and ${\cal G}=(G_i)_{i=0}^r$.

In case $G_i$ admits a Shestakov-Umirbaev reduction, 
the possibilities for $G_{i+1}$ are limited 
as described in the following propositions.

\begin{proposition}\label{prop:uniqueness}
If $(F,G^1)$ and $(F,G^2)$ satisfy 
the Shestakov-Umirbaev condition for $F,G^1,G^2\in \TT $, 
then $g_i^1=g_i^2$ for $i=1,2$, 
and $g_3^1-g_3^2$ is contained in $k[g_2^1]$, 
where $G^j=(g_1^j,g_2^j,g_3^j)$ 
for $j=1,2$. 
\end{proposition}
\begin{proof}\rm
By (SU1), 
there exist $a^j,b^j,c^j\in k$ 
such that $g_1^j=f_1+a^jf_3^2+c^jf_3$ 
and $g_2^j=f_2+b^jf_3$ for $j=1,2$. 
By the last statement of (P11), 
it follows that 
$a^1=a^2$, $b^1=b^2$ and $c^1=c^2$. 
Hence, 
we have $g_i^1=g_i^2$ for $i=1,2$. 
Put $\phi :=g_3^1-g_3^2=(g_3^1-f_3)+(f_3-g_3^2)$. 
Then, 
$\phi $ belongs to $k[g_1^1,g_2^1]=k[g_1^2,g_2^2]$, 
since so does $g_3^j-f_3$ for $j=1,2$ by (SU1). 
Suppose to the contrary that 
$\phi $ belongs to $k[g_1^1,g_2^1]\sm k[g_2^1]$. 
Then, since 
$\degw \phi \leq \max \{ \degw g_3^1, \degw g_3^2\} 
<\degw f_3\leq \degw g_1^1$ 
by (SU5) and (SU4), 
we get $\degw \phi <\degw ^U\phi $, 
where $U=\{ g_1^1,g_2^1\} $. 
In view of (SU3), 
it follows from Lemma~\ref{lem:degS1}(i) that 
\begin{align*}
\degw \phi \geq 2\degw g_1^1+\degw dg_1^1\wedge dg_2^1
-\degw g_1^1-\degw g_2^1=\degw g_1^1-\degw g_2^1+\degw dg_1^1\wedge dg_2^1. 
\end{align*}
Since $\deg \phi \leq \max \{ \degw g_3^1, \degw g_3^2\} $, 
this contradicts (SU6). 
Therefore, 
$g_3^1-g_3^2$ belongs to $k[g_2^1]$. 
\end{proof}

The following proposition gives a necessary condition 
on automorphisms 
to admit both an elementary reduction 
and a Shestakov-Umirbaev reduction simultaneously.

\begin{proposition}\label{prop:not er}
Assume that $(F,G)$ satisfies the Shestakov-Umirbaev condition 
for $F,G\in \TT $. 
Then, 
$f_i^{\w }$ does not belong to $k[S_i]^{\w }$ 
for $i=1$ if 
$f_1^{\w }\not\approx (f_3^{\w })^2$, 
for $i=2$, 
and for $i=3$ if $(f_1,f_2)\neq (g_1,g_2)$. 
\end{proposition}
\begin{proof}\rm
In each case, 
we will find 
$h_0,h_1\in k[S_i]$ such that $k[h_0,h_1]=k[S_i]$, 
$\gamma _i':=\degw dh_0\wedge dh_1>s\delta $, 
$h_j^{\w }$ does not belong to $k[h_l^{\w }]$ 
for $(j,l)=(0,1),(1,0)$, 
and $f_i^{\w }$ does not belong to $k[h_0^{\w },h_1^{\w }]$. 
Then, it follows that $f_i^{\w }$ 
does not belong to $k[S_i]^{\w }$. 
In fact, 
supposing that $f_i^{\w }=\phi ^{\w }$ 
for some $\phi \in k[S_i]=k[h_0,h_1]$, 
we have $\degw \phi <\degw ^{U}\phi $ for $U=\{ h_0,h_1\} $, 
since $\phi ^{\w }=f_i^{\w }$ 
does not belong to $k[h_0^{\w },h_1^{\w }]$. 
Since $h_j^{\w }$ does not belong to $k[h_l^{\w }]$ 
for $(j,l)=(0,1),(1,0)$, 
we get $\degw \phi >\gamma _i'$ by Lemma~\ref{lem:degS2}(i). 
Thus, 
$\deg f_i=\deg \phi >\gamma _i'>s\delta $. 
This contradicts (P7). 
Therefore, 
$f_i^{\w }$ 
does not belong to $k[S_i]^{\w }$ 
if such $h_0$ and $h_1$ exist.

We remark that 
$\gamma _i:=\deg f_j\wedge f_l>s\delta $ 
in each case, 
where $j,l\in \{ 1,2,3\} \sm \{ i\} $ with $j<l$. 
Actually, 
$\gamma _1>s\delta $ 
and $\gamma _2\geq \delta +\gamma _1>(s+1)\delta $ 
by the last two conditions of (P12). 
If $i=3$, then $(f_1,f_2)\neq (g_1,g_2)$ by assumption. 
Hence, the first condition of (P12) implies that
$\gamma _3$ is equal to one of 
$\deg f_3+\gamma _1$, $\gamma _2$ and $\gamma _1$, 
which are greater than $s\delta $. 

We set $(h_0,h_1)=(f_2,f_3)$ if $i=1$, 
and $(h_0,h_1)=(f_1,f_2)$ if $i=3$. 
Then, 
$k[h_0,h_1]=k[S_i]$ 
and $\gamma _i'=\gamma _i>s\delta $ 
in either case. 
Moreover, 
$h_j^{\w }$ does not belong to $k[h_l^{\w }]$ 
for $(j,l)=(0,1),(1,0)$ by (P8). 
We check that $f_i^{\w }$ does not belong to 
$k[h_0^{\w },h_1^{\w }]$. 
This holds for $i=3$ 
because $f_3^{\w }$ does not belong to 
$k[f_l^{\w }]$ for $l=1,2$ by (P8), 
and $\deg f_3\leq \deg f_1<\deg f_1f_2$ by (P7). 
Suppose to the contrary that 
$f_1^{\w }$ belongs to $k[f_2^{\w },f_3^{\w }]$. 
Then, 
$f_1^{\w }$ must belong to 
$k[f_2^{\w }]$ or $k[f_3^{\w }]$, 
since 
$$
\degw f_1\leq \degw g_1=s\delta =2\delta +(s-2)\delta 
<\degw f_2+\degw f_3 =\deg f_2f_3
$$
by (SU2) and (P2). 
It follows from (P8) that $f_1^{\w }$ 
does not belong to $k[f_2^{\w }]$, 
and so $f_1^{\w }$ belongs to $k[f_3^{\w }]$ 
and $f_1^{\w }\approx (f_3^{\w })^2$. 
This contradicts the assumption that 
$f_1^{\w }\not\approx (f_3^{\w })^2$. 
Thus, 
$f_1^{\w }$ does not belong to 
$k[h_0^{\w },h_1^{\w }]$ in case $i=1$. 
Therefore, 
$h_0$ and $h_1$ satisfy the required conditions, 
and thereby 
$f_i^{\w }$ does not belong to $k[S_i]^{\w }$ 
for $i=1,3$ as mentioned above.

In case $i=2$, 
set $h_0=f_3$, 
and $h_1=f_1$ if 
$f_1^{\w }\not\approx (f_3^{\w })^2$, 
while $h_1=f_1-cf_3^2$ otherwise, 
where $c\in k$ such that 
$f_1^{\w }=c(f_3^{\w })^2$. 
Then, %it follows that 
$k[h_0,h_1]=k[S_2]$ 
and $\gamma _2'=\gamma _2>(s+1)\delta $. 
If $f_1^{\w }\not\approx (f_3^{\w })^2$, 
then $h_1=f_1$, 
and so $h_j^{\w }$ does not belong to 
$k[h_l^{\w }]$ for $(j,l)=(0,1),(1,0)$ by (P8). 
If $f_1^{\w }\approx (f_3^{\w })^2$, then 
$f_1^{\w }$ belongs to $k[f_3^{\w }]$. 
By (P8), 
we get $s=3$ and $\deg h_0=\deg f_3=(3/2)\delta $. 
Since $\deg h_0+\deg h_1\geq \gamma _2'>(s+1)\delta =4\delta $ 
by (\ref{eq:ineq-wedge}), 
we have $\degw h_1>4\delta -(3/2)\delta =(5/2)\delta >\deg h_0$. 
Hence, 
$h_0^{\w }$ does not belong to $k[h_1^{\w }]$. 
It follows that 
$(5/2)\delta <\deg h_1=\deg (f_1-cf_3^2)<\deg f_3^2=3\delta $. 
Since $5/2 <(3/2)l<3$ does not hold for any $l\in \N $, 
we conclude that 
$h_1^{\w }$ does not belong to $k[h_0^{\w }]$. 
For both $h_1=f_1$ and $h_1=f_1-cf_3^2$, 
it holds that $\degw f_2=2\delta <\degw h_1$. 
Hence, 
$f_2^{\w }$ does not belong to 
$k[h_0^{\w },h_1^{\w }]\sm k[h_0^{\w }]$. 
By (P8), 
$f_2^{\w }$ does not belong to $k[h_0^{\w }]=k[f_3^{\w}]$. 
Thus, 
$f_2^{\w }$ does not belong to $k[h_0^{\w },h_1^{\w }]$. 
Therefore, $h_0$ and $h_1$ satisfy the required conditions, 
thereby $f_2^{\w }$ does not belong to $k[S_2]^{\w }$. 
\end{proof}

\section*{Appendix: Reductions of types I, II, III and IV}
\setcounter{equation}{0}

In this appendix, 
we explain that the following results are implicit in 
the theory of Shestakov-Umirbaev~\cite{SU2}:

(A) If $F\in \Aut _k\kx $ admits 
a reduction of one of the types I, II, III and IV, 
then $F$ admits none of the reductions of the other three types.

(B) If $F\in \Aut _k\kx $ admits a reduction of type IV, 
then there exists an elementary automorphism $E$ 
such that $F\circ E$ admits a reduction of type IV, 
but does not admit an elementary reduction.

From (A) and (B), 
it follows that, 
if there exists a tame automorphism admitting a reduction of type IV, 
then there exists a tame automorphism 
which is not affine and does not admit an elementary reduction 
nor any one of the reductions of types I, II and III\null. 
Actually, 
an automorphism admitting a reduction of type IV is not affine, 
and admits none of the reductions of types I, II and III by (A)\null. 
Theorem~\ref{thm:main1}, 
together with Proposition~\ref{prop:sured to 123}, 
implies that 
each tame automorphism but an affine automorphism 
admits an elementary reduction or 
a reduction of one of the types I, II and III\null. 
Thus, we obtain another proof of Theorem~\ref{thm:type IV} 
that no tame automorphism admits a reduction of type IV\null.

First, we show (A)\null. 
Recall the definitions of reductions of types I--IV 
(see the conditions (1)--(4) listed in Section~\ref{sect:remark}). 
If $F$ satisfies (1), 
then $\deg f_1<\deg f_3\leq \deg f_2$. 
Moreover, 
(1) implies that $\deg df_1\wedge df_2=\deg df_1\wedge df_3$ 
(cf.\ \cite[Proposition 1 (1)]{SU2}). 
If $F$ satisfies one of (2), (3) and (4), 
then $\deg f_3\leq \deg f_1<\deg f_2$, 
where $\deg f_3=\deg f_1$ holds only in case (2). 
Moreover, it follows that 
\begin{equation}\label{eq:67}
\deg df_1\wedge df_3=\deg dg_1\wedge dg_2+3l,\quad 
\deg df_2\wedge df_3=\deg df_1\wedge df_3+l
\end{equation}
in these cases (cf.\ \cite[Equations (6) and (7)]{SU2}). 

Now, 
suppose that 
$F$ satisfies one of (2), (3) and (4), 
but admits a reduction of type~I, 
i.e., 
$F_{\tau }$ satisfies (1) for some $\tau \in \sym _3$. 
Then, 
$\deg df_{\tau (1)}\wedge df_{\tau (2)}
=\deg df_{\tau (1)}\wedge df_{\tau (3)}$ as mentioned. 
It follows from the condition on the degrees of 
$f_1$, $f_2$ and $f_3$ that $\tau =(1,3)$. 
Hence, 
$\deg df_3\wedge df_2=\deg df_3\wedge df_1$, 
which contradicts the second equation of (\ref{eq:67}). 
If $F$ satisfies (3) or (4), 
and admits a reduction of type II, 
then $F$ satisfies (2) 
by the conditions on the degrees of $f_1$, $f_2$ and $f_3$. 
This is impossible, 
because $(3/2)l<\deg f_3$ in case (2), 
while $\deg f_3\leq (3/2)l$ in cases (3) and (4). 
Finally, 
we show that $F$ does not admit reductions of types III and IV 
simultaneously. 
Suppose that $F$ satisfies (4), 
and admits a reduction of type III\null. 
Then, $F$ satisfies (3), 
since $\deg f_3<\deg f_1<\deg f_2$ in both cases. 
We remark that $\alpha ,\beta ,\gamma \in k$ 
appearing in (3) and (4) are uniquely determined by $F$ 
(cf.~\cite[Proposition 3 (1), (2) and (3)]{SU2}), 
and hence so are $g_1$ and $g_2$. 
There exist $\sigma ^1,\sigma ^2\in k\sm \zs $ 
and $g^1,g^2\in k[g_1,g_2]\sm k$ such that 
$\deg dg_1\wedge dg_{3,i}<3l+\deg dg_1\wedge dg_2$ for $i=1,2$, 
$\deg g_{3,1}<l+\deg dg_1\wedge dg_2$, 
and $\deg (g_2-\mu g_{3,2}^2)\leq 2l$ for some $\mu \in k\sm \zs $, 
where $g_{3,i}=\sigma ^if_3+g^i$ for $i=1,2$. 
We claim that $\deg g_{3,1}<\deg f_3$. 
In fact, 
$\deg g_{3,1}<l+\deg dg_1\wedge dg_2$, 
while the first equation of (\ref{eq:67}) 
implies $\deg f_3\geq l+\deg dg_1\wedge dg_2$, 
since $\deg f_1+\deg f_3\geq \deg df_1\wedge df_3$ 
and $\deg f_1=2l$. 
Hence, 
$\deg g_{3,1}<\deg f_3\leq (3/2)l$. 
From $\deg (g_2-\mu g_{3,2}^2)\leq 2l$, 
we get $\deg g_{3,2}=(3/2)l$. 
It follows that 
$\phi :=\sigma ^2g^1-\sigma ^1g^2
=\sigma ^2g_{3,1}-\sigma ^1g_{3,2}$ 
is an element of $k[g_1,g_2]$ such that 
$\deg dg_1\wedge d\phi <3l+\deg dg_1\wedge dg_2$ 
and $\deg \phi =(3/2)l$. 
Since $\deg \phi <\deg g_i$ for $i=1,2$, 
and since $\phi $ is not an element of $k$, 
we have $\deg \phi <\deg ^U\phi $, 
where $U=\{ g_1,g_2\} $. 
As $\deg g_1=2l$ and $\deg g_2=3l$, 
it follows from Lemma~\ref{lem:degS2}(ii) that 
$\deg d\phi \wedge dg_1\geq 3l+\deg dg_1\wedge dg_2$, 
a contradiction. 
Therefore, 
$F$ does not admit reductions of types III and IV 
simultaneously. 
This completes the proof of (A).

Next, assume that $F$ satisfies (4). 
From the proof of \cite[Lemma 12]{SU2}, 
we know that each 
$a\in k[S_i]$ with $\deg a\leq \deg f_i$ 
can be written as follows: 
If $i=1$, then $a=\delta _1f_3$ 
(up to a constant term) for some $\delta _1\in k$. 
If $i=2$, then $a=\delta _1f_3^2+\sigma _1f_3+\mu _1f_1$ 
(up to a constant term) for some $\delta _1,\sigma _1,\mu _1\in k$. 
If $i=3$ and $(\alpha ,\beta ,\gamma )\neq (0,0,0)$, 
then $a$ is an element of $k$. 
It is also mentioned in the proof of \cite[Lemma 12]{SU2} that 
$(f_1,f_2+a,f_3)$ satisfies (4) 
for each $a\in k[S_2]$ with $\deg a\leq \deg f_2$. 
In fact, it is claimed that $(g_1,g_2+\mu _1g_1,g_3)$ 
is a ``predreduction" of type IV of $(f_1,f_2+a,f_3)$.

We deduce (B) from the facts above. 
The assertion is clear if 
$F$ does not admit an elementary reduction. 
So, 
assume that $\deg F\circ E<\deg F$ for some $E\in \E _i$, 
where $i\in \{ 1,2,3\} $. 
Then, 
$(F\circ E)(x_i)=f_i+a$ and $\deg (f_i+a)<\deg f_i$ 
for some $a\in k[S_i]$. 
Since $\deg a=\deg f_i$, 
we can write $a$ as stated in the preceding paragraph. 
Hence, 
if $i=1$, then $\deg a=\deg \delta _1f_3\leq (3/2)l$. 
Since $\deg a=\deg f_1=2l$, 
this is impossible. 
Thus, $i\neq 1$. 
If $i=2$, then $\deg a=\deg f_2>(5/2)l$. 
Since $\deg f_3\leq (3/2)l$, 
we have $\delta _1\neq 0$ and $\deg f_2=\deg a=2\deg f_3$. 
This implies that 
$\deg f_2=3l$ and $\deg f_3=(3/2)l$, 
for 
$\deg f_2=3l$ if $\deg f_3<(3/2)l$, 
and $\deg f_3=(3/2)l$ if $\deg f_2<3l$. 
If $i=3$, 
then $\alpha=\beta =\gamma =0$. 
and so $g_1=f_1$ and $g_2=f_2$. 
We show that $F\circ E$ admits a reduction of type IV, 
but does not admit an elementary reduction 
in cases $i=2$ and $i=3$.

Assume that $i=2$. 
Then, 
$\deg (f_2+a)<\deg f_2=3l$. 
Moreover, 
$F\circ E=(f_1,f_2+a,f_3)$ satisfies (4) as mentioned, 
in which $\alpha \in k$ involved in the condition cannot be zero, 
since $\deg (f_2+a)<3l$. 
By applying to $F\circ E$ the argument in the preceding paragraph, 
we know that there does not exist $E'\in \E _j$ 
with $\deg F\circ E\circ E'<\deg F\circ E$ 
for $j=1$, 
for $j=2$, since $\deg (f_2+a)\neq 3l$, 
and for $j=3$, since the constant $\alpha $ is not zero. 
Thus, 
$F\circ E$ does not admit an elementary reduction.

Assume that $i=3$. 
Without loss of generality, 
we may assume that 
$(F\circ E)(x_3)^{\w }$ 
does not belong to $k[f_1,f_2]^{\w }$ 
by replacing $E$ if necessary. 
We show that $F\circ E=(f_1,f_2,f_3+a)$ satisfies (4) 
by using the assumption that $F$ satisfies (4) 
for $\alpha=\beta =\gamma =0$. 
We claim that $\deg (f_3+a)\geq l+\deg dg_1\wedge dg_2$. 
In fact, if not, 
we can check that $(f_1,f_2+f_3,f_3)$ satisfies (3) and (4) 
by the assumption that $F$ satisfies (4) for $\alpha=\beta =\gamma =0$. 
This contradicts (A)\null. 
Hence, $l<\deg (f_3+a)\leq (3/2)l$, 
as required in the assumption of (4). 
Let $g':=g_3-\sigma (f_3+a)=\sigma f_3-g-\sigma (f_3+a)$. 
Then, 
$g'$ belongs to $k[g_1,g_2]=k[f_1,f_2]$, 
since so do $g$ and $a$. 
It follows that $\deg g'=(3/2)l$, 
since $\deg (f_3+a)<\deg f_3\leq (3/2)l$ and $\deg g_3=(3/2)l$. 
Hence, 
$g$ is not an element of $k$. 
Moreover, 
we can express 
$g_3=\sigma (f_3+a)+g'$. 
This shows that $F\circ E$ satisfies (4). 
Consequently, 
there does not exist $E'\in \E _j$ 
such that $\deg F\circ E\circ E'<\deg F\circ E$ 
for $j=1$, 
and for $j=2$, 
since $\deg (f_3+a)\neq (3/2)l$. 
This also holds for $j=3$ 
as we choose $E$ so that 
$(F\circ E)(x_3)^{\w }$ 
does not belong to $k[f_1,f_2]^{\w }$. 
Therefore, 
$F\circ E$ does not admit an elementary reduction. 
This completes the proof of~(B)\null.

\address{
Department of Mathematics \\ and Information Sciences\\ 
Tokyo Metropolitan University\\
1-1  Minami-Ohsawa, Hachioji \\
Tokyo 192-0397, Japan\\
}{
kuroda@tmu.ac.jp}

\end{document}